\batchmode
\makeatletter
\def\input@path{{C:/Users/Amitay/Dropbox/Thesis/}}
\makeatother
\documentclass[oneside,english]{amsart}
\usepackage[T1]{fontenc}
\usepackage[latin9]{inputenc}
\usepackage{geometry}
\usepackage{color}
\usepackage{babel}
\usepackage{amstext}
\usepackage{amsthm}
\usepackage{amssymb}
\PassOptionsToPackage{normalem}{ulem}
\usepackage{ulem}
\usepackage[unicode=true,pdfusetitle,
 bookmarks=true,bookmarksnumbered=true,bookmarksopen=true,bookmarksopenlevel=1,
 breaklinks=false,pdfborder={0 0 1},backref=section,colorlinks=true]
 {hyperref}

\makeatletter

\providecommand{\tabularnewline}{\\}

\numberwithin{equation}{section}
\numberwithin{figure}{section}
\numberwithin{table}{section}
\theoremstyle{plain}
\newtheorem{thm}{\protect\theoremname}[section]
  \theoremstyle{definition}
  \newtheorem{defn}[thm]{\protect\definitionname}
  \theoremstyle{plain}
  \newtheorem{cor}[thm]{\protect\corollaryname}
  \theoremstyle{plain}
  \newtheorem{prop}[thm]{\protect\propositionname}
  \theoremstyle{definition}
  \newtheorem{example}[thm]{\protect\examplename}
  \theoremstyle{plain}
  \newtheorem{lem}[thm]{\protect\lemmaname}
  \theoremstyle{remark}
  \newtheorem{rem}[thm]{\protect\remarkname}
  \theoremstyle{definition}
  \newtheorem{problem}[thm]{\protect\problemname}
  \theoremstyle{plain}
  \newtheorem*{thm*}{\protect\theoremname}

\makeatother

  \providecommand{\corollaryname}{Corollary}
  \providecommand{\definitionname}{Definition}
  \providecommand{\examplename}{Example}
  \providecommand{\lemmaname}{Lemma}
  \providecommand{\problemname}{Problem}
  \providecommand{\propositionname}{Proposition}
  \providecommand{\remarkname}{Remark}
  \providecommand{\theoremname}{Theorem}
\providecommand{\theoremname}{Theorem}

\begin{document}

\title{$L_{p}$-Expander Complexes}

\author{Amitay Kamber}

\thanks{Institute of Mathematics, The Hebrew University of Jerusalem, amitay.kamber@mail.huji.ac.il}
\begin{abstract}
We discuss two combinatorical ways of generalizing the definition
of expander graphs and Ramanujan graphs, to quotients of buildings
of higher dimension. The two possible definitions are equivalent for
affine buildings, giving the notion of an $L_{p}$-expander complex.
We calculate explicit spectral gaps on many combinatorical operators,
on any $L_{p}$-expander complex. 

We associate with any complex a natural ``zeta function'', generalizing
the Ihara-Hashimoto zeta function of a finite graph. We generalize
a well known theorem of Hashimoto, showing that a complex is Ramanujan
if and only if the zeta function satisfies the Riemann hypothesis.
\end{abstract}

\maketitle

\section{Introduction}

Expander graphs in general and Ramanujan graphs in particular, have
been a topic of great interest in the last four decades. In recent
years, a theory of high dimensional expanders has emerged (see \cite{lubotzky2014ramanujan}
and the references therein). Namely, high dimensional simplicial complexes
which resemble the properties of expander graphs when specialized
to dimension one. Several different notions of high dimensional expanders
have been proposed (which in general are not equivalent), each with
its own goal and motivation.

The goal of this paper is to propose yet another such a generalization,
based on the representation theory of all dimensional Hecke Algebras
(to be defined later) associated with buildings. One of the advantages
of our approach is that it gives a generalization of expanders and
Ramanujan graphs in a unified way. In addition, we deduce that quotients
of buildings associated with semisimple groups over local fields with
property (T) are indeed high dimensional expanders, as one expects,
recalling the classical result of Margulis showing how to get expander
graphs from groups with property (T) (\cite{margulis1973explicit}). 

In \cite{kamber2016lpgraph} we presented an alternative definition
of expander graphs. The goal of this work is to generalize it to higher
dimensions.

Consider a regular, locally finite thick affine building $B$, with
parameter $q$. The building $B$ has a corresponding irreducible
affine Coxeter group $(W,S)$ of type $\tilde{A}_{n}$, where $S$
is a finite set of generators of $W$. Denote by $B_{\phi}$ the set
of chambers of $B$, i.e. the highest dimensional faces. Every panel
(i.e. a codimension one face) $\sigma$ is contained in $q+1$ chambers.
The building $B$ is a \emph{colored} pure simplicial complex, which
means that each panel $\sigma$ has a natural color (or cotype) $t(\sigma)\in S$
and each chamber contains $\left|S\right|=n+1$ panels of different
colors. The \emph{color} or\textbf{ }$\tau(\sigma)\subset S$ of a
general face $\sigma$ is the union of the colors of the panels containing
it.

The \emph{$W$-metric approach} to buildings allows us to define a
distance between every two chambers $C_{0},C_{1}\in B_{\phi}$ by
$d(C_{0},C_{1})\in W$. Each chamber $C$ has $q_{w}=q^{l(w)}$ chambers
$C^{\prime}$ with $d(C,C^{\prime})=w\in W$.

Let $\Gamma$ be a cocompact torsion free lattice in $G$. In this
case the quotient space $X=\Gamma\backslash B$ is a finite colored
simplicial complex. Identify its chambers by $X_{\phi}.$ Let $\pi:B\rightarrow X$
be the projection and define for $f\in\mathbb{C}^{X_{\phi}}$ the
pullback $\pi^{\ast}f\in\mathbb{C}^{B_{\phi}}$. Let $C_{0}$ be a
fixed chamber of $B$, and let $\rho_{C_{0}}:\mathbb{C}^{B_{\phi}}\rightarrow\mathbb{C}^{B_{\phi}}$
be the spherical average operator defined by $\rho_{C_{0}}(f)(C)=\frac{1}{q_{d(C_{0},C)}}\sum_{C^{\prime}:d(C_{0},C^{\prime})=d(C_{0},C)}f(C^{\prime})$.
Finally, define the non trivial space $L_{2}^{0}(X_{\phi})=\left\{ f\in\mathbb{C}^{X_{\phi}}:\,\sum_{C\in X_{\phi}}f(C)=0\right\} $.
We can now define:
\begin{defn}
\label{def:Expander definition}For $2\le p\le\infty$ call $X$ an
$L_{p}$\emph{-expander} if for every $f\in L_{2}^{0}(X_{\phi})$
and $C_{0}\in B_{\phi}$, $\rho_{C_{0}}(\pi^{\ast}f)\in L_{p+\epsilon}(B_{\phi})$
for every $\epsilon>0$.

Call $X$ \emph{Ramanujan} if it is an $L_{2}$-expander.
\end{defn}
The definition is equivalent to the fact that every matrix coefficient
of every subrepresentation of $L^{2}(\Gamma\backslash G)$ with Iwahori
fixed vector is in $L_{p+\epsilon}(G)$ for every $\epsilon>0$. Generalizing
the $p=2$ case, we say that a $G$-representation satisfying this
property is \emph{$p$-tempered}.

Let us say right away that the Ramanujan complexes constructed in
\cite{lubotzky2005explicit,lubotzky2005ramanujan} using Lafforgue
work, are also Ramanujan in our sense (See \textbf{\uline{\mbox{\cite{first2016ramanujan}}}}).
However, our definition is aperiori stronger than the one given in
\cite{lubotzky2005explicit} as it requires the $L_{2+\epsilon}$
condition on functions on faces of all colors, in contrast with \cite{lubotzky2005explicit}
where only functions on vertices are considered. This is reflected
by the fact that \cite{lubotzky2005ramanujan} works only with the
classical spherical Hecke algebra, while we consider the Iwahori-Hecke
algebra of the all dimensional Hecke algebra (see below). This allows
us in turn to analyze higher dimensional faces and not just vertices
(compare \cite{evra2015mixing}). The definition is equivalent to
\emph{strongly-Ramanujan} in \cite{kang2016riemann} and \emph{flag-Ramanujan}
in \cite{first2016ramanujan}.

The definition of $p$-temperedness is intimately connected to property
(T) of reductive groups. The results of Oh in \cite{oh2002uniform}
express the quantitative property $(T)$ of reductive groups in the
following way:
\begin{thm}
\label{thm:Oh's Theorem-intro}(Oh \cite{oh2002uniform}) Let $k$
be a non-Archimedean local field with $\mathrm{char}k\ne2$. Let $G$
be the group of $k$-rational points of a connected linear almost
$k$-simple algebraic group with $k$-rank $\ge2$. Then every irreducible
infinite dimensional unitary representation of $G$ is $p_{0}$-tempered
for some explicit $p_{0}$ depending only on the affine Coxeter group
$W$. Explicitly, the bounds are:

\begin{tabular}{|c|c|c|c|c|c|c|c|c|c|c|}
\hline 
$W$ &
$\tilde{A}_{n}$ &
$\tilde{B}_{n}$ &
$\tilde{C}_{n}$ &
$\tilde{D}_{n}$, $n$ even &
$\tilde{D}_{n}$, $n$ odd &
$\tilde{E}_{6}$ &
$\tilde{E}_{7}$ &
$\tilde{E}_{8}$ &
$\tilde{F}_{4}$ &
$\tilde{G}_{2}$\tabularnewline
\hline 
$p_{0}$ &
$2n$ &
$2n$ &
$2n$ &
$2(n-1)$ &
$2n$ &
16 &
18 &
29 &
11 &
6\tabularnewline
\hline 
\end{tabular}
\end{thm}
As a corollary we get:
\begin{cor}
Every finite quotient complex $X$ of the building corresponding to
$G$ as above is a $L_{p_{0}}$-expander. 
\end{cor}
Let us now relate the above to the representation theory of Hecke
algebras. For $w\in W$ we define the operator $h_{w}:\mathbb{C}^{B_{\phi}}\rightarrow\mathbb{C}^{B_{\phi}}$
by $h_{w}f(C)=\sum_{C^{\prime}:d(C,C^{\prime})=w}f(C^{\prime})$.
\begin{defn}
\emph{The Iwahori-Hecke algebra $H_{\phi}$ of $B$} is the linear
span of all the $h_{w}$, $w\in W$. 
\end{defn}
Those operators satisfy for $w\in W,\,s\in S$ the Iwahori-Hecke relations:
\[
\begin{array}{cccc}
h_{s}^{2} & = & q\cdot Id+(q-1)h_{s}\\
h_{w}h_{s} & = & h_{ws}\,\,\,\,\,\,\,\,\,\,\,\,\, & if\,\,l(ws)=l(w)+1
\end{array}
\]

In terms of $G$, the Iwahori-Hecke algebra can be written as $H_{\phi}\cong C_{c}\left[G_{\phi}\backslash G/G_{\phi}\right]$,
where $G_{\phi}$ (usually denoted $I$) is the Iwahori subgroup of
$G$ and $C_{c}\left[G_{\phi}\backslash G/G_{\phi}\right]$ has the
algebra structure given by convolution. The algebra $H_{\phi}$ can
also be defined as the set of all row and column finite (see definition
\ref{def:Row and Column finite}) linear operators acting on $\mathbb{C}^{B_{\phi}}$
and commuting with the action of $G$ on the space.

The Iwahori-Hecke algebra $H_{\phi}$ acts naturally on $\mathbb{C}^{X_{\phi}}$,
i.e. functions on chambers of the complex. In fact, the standard inner
product on $\mathbb{C}^{X_{\phi}}\cong L_{2}(X_{\phi})$ gives it
the structure of a finite dimensional \emph{unitary representation}
of $H_{\phi}$ (as a $\ast$-algebra), and $L_{2}^{0}(X_{\phi})$
is a proper subrepresentation.
\begin{defn}
\label{def:Weakly-Contained}We say that a finite dimensional representation
$V$ of $H_{\phi}$ is $p$-tempered if for every $v\in V$, $u\in V^{\ast}$
and $\epsilon>0$ we have $\sum_{w\in W}q^{l(w)(1-p-\epsilon)}\left|\left\langle u,h_{w}v\right\rangle \right|^{p+\epsilon}<\infty$.
\end{defn}
An easy calculation shows that definition \ref{def:Expander definition}
is actually equivalent to the $p$-temperedness of the $H_{\phi}$-representation
$L_{2}^{0}(X_{\phi})$.

There is an alternative possible definition of an $L_{p}$-expander.
It is based on the following:
\begin{defn}
Assume $V$ is a finite dimensional representations of $H_{\phi}$.
We say that $V$ is \emph{weakly contained} in $L_{p}(B_{\phi})$
if for each operator $h\in H_{\phi}$ and eigenvalue $\lambda$ of
$h$ on $V$, $\lambda$ belongs to the approximate point spectrum
of $h$ on $L_{p}(B_{\phi})$.
\end{defn}
The following is the main theorem of this work. It says that the two
possible definitions are equivalent in the case we consider (Compare
\cite{cowling1988almost}, for $p=2$ only).
\begin{thm}
\label{thm:Main thm}A finite dimensional representation $V$ of $H_{\phi}$
is $p$-tempered if and only if it is weakly contained in $L_{p}(B_{\phi})$.

Therefore $X$ is an $L_{p}$-expander if and only if the $H_{\phi}$-representation
$L_{2}^{0}(X_{\phi})$ is weakly contained in $L_{p}(B_{\phi})$.
\end{thm}
In theorem \ref{thm:Main thm}, the Iwahori-Hecke algebra $H_{\phi}$
can be extended to a larger algebra , which we call the \emph{all
dimensional Hecke (ADH) algebra} $H$. The simplest way to define
it is:
\begin{defn}
\label{def:Hecke-algbera-def-by-commute}Identify $B$ with the set
of all its faces. The ADH algebra $H$ of $B$ is the algebra of all
row and columns finite linear operators acting on $\mathbb{C}^{B_{f}}$
and commuting with the action of $G$. 
\end{defn}
The algebra $H$ contains many interesting operators, such as boundary
and coboundary operators, Laplacians and adjacency operators. It is
described explicitly in section \ref{sec:The-All-Dimensional} (see
also \cite{abramenko2015distance}). Shortly, distances between general
faces are parameterized by $d\in W_{I_{1}}\backslash W/W_{I_{2}}$,
for $I_{1},I_{2}\subsetneq S$, where $W_{I_{1}},W_{I_{2}}\subset W$
are the corresponding parabolic subgroups. The algebra $H$ is spanned
by operators $h_{d}$, for $d\in W_{I_{1}}\backslash W/W_{I_{2}}$. 

Explicit bounds on the operators of $H$ can be given by the following
theorem. Notice that we have a length function $l:W\rightarrow\mathbb{N}$.
Each operator $h_{w}\in H_{\phi}$ sums $q^{l(w)}$ different chambers
of $B_{\phi}$ and therefore $q^{l(w)}$ is its trivial eigenvalue.
Now:
\begin{thm}
\label{thm:Bounds_Theorem}The norm of $h_{w}\in H_{\phi}$ is bounded
on $L_{p}(B_{\phi})$ by $D(w,l(w))q^{l(w)(p-1)/p}$, where $D(q,l)=\left|W_{0}\right|2^{l(\tilde{w}_{0})}q^{4\cdot l(\tilde{w}_{0})}\left(l+1+l(\tilde{w}_{0})\right)^{l(\tilde{w}_{0})}$,
$W_{0}$ is the spherical Coxeter group corresponding to $W$ and
$\tilde{w}_{0}$ is the longest element of $W_{0}$.

Therefore the same bound applies to the action of $h_{w}\in H_{\phi}$
on $L_{2}^{0}(X_{\phi})$.
\end{thm}
In conjugation with Oh's theorem \ref{thm:Oh's Theorem-intro} this
theorem gives an explicit spectral gap of the operators $h_{w}\in H_{\phi}$,
in any quotient of the building.

There exists a direct application of the last theorem . We measure
distance between chambers in a quotient complex $X$ by \emph{gallery
distance, }i.e the length of the shortest gallery connecting the two
chambers. 
\begin{thm}
\label{thm:Distance-and-Diameter}Let $X$ be an $L_{p}$-expander
of with $N$ chambers and $C_{0}\in X_{\phi}$. Let $n$ be the dimension
of $X$ and $\tilde{w}_{0}$ is the longest element of the spherical
Coxeter group $W_{0}$. Then all but $o(N)$ chambers $C\in X_{\phi}$
are of gallery distance $l(C_{0},C)$ which satisfies 
\[
l(C_{0},C)\le\frac{p}{2}\log_{q}N+\left(l(\tilde{w}_{0})+1\right)\log_{q}\log_{q}N+1
\]

and 
\[
l(C_{0},C)\ge\log_{q}N-\left(n+1\right)\log_{q}\log N_{q}-1
\]

In addition, the diameter of $X$ is at most $p\log_{q}N+2\left(l(\tilde{w}_{0})+1\right)\log_{q}\log_{q}N+1$.
\end{thm}
Compare the graph case in \cite{lubetzky2015cutoff} corollary 2,
or \cite{sardari2015diameter}.

As a final result, recall that for a graph, the expander property
is connected to the eigenvalues of both the vertex adjacency operator
and of Hashimoto's non backtracking operator. For the high dimensional
case, we can give a generalization of the non-backtracking operator.
As a preliminary, one can extend the Iwahori-Hecke algebra $H_{\phi}$
to an extended Iwahori-Hecke algebra $\hat{H}_{\phi}$, given by operators
$h_{w},w\in\hat{W}$, the extended Iwahori-Hecke algebra. The operators
of $\hat{H}_{\phi}$ act naturally on function on ``colored'' chambers
of $B$- $\mathbb{C}^{\hat{B}_{\phi}}=\mathbb{C}{}^{B_{\phi}\times\Omega}$.
Within $\hat{H}_{\phi}$ we have $n$ Bernstein-Luzstig operators
$h_{\beta_{1}},...,h_{\beta_{n}}$, corresponding to the simple coweights
$\beta_{1},...,\beta_{n}$ of the root system of $\hat{W}$. The one
dimensional case agrees with Hashimoto's non-backtracking operator.
Now:
\begin{thm}
\label{thm:Zeta_Function_Thm_intro}Let $V$ be a finite dimensional
representation of $\hat{H}_{\phi}$. Then $V$ is $p$-tempered if
and only if for every $i=1,..,n$ every eigenvalue $\lambda$ of $h_{\beta_{i}}$
on $V$ satisfies $\left|\lambda\right|\le q^{l(\beta_{i})(p-1)/p}$.
\end{thm}
The theorem encourages the following definition:
\begin{defn}
Consider the $\hat{H}_{\phi}$-representation $L_{2}\left(\hat{B}_{\phi}\right)$.
Let
\[
\zeta_{\hat{B}_{\phi}}(u)=\frac{1}{\det(1-h_{\beta_{1}}u^{l(\beta_{1})})\cdot...\cdot\det(1-h_{\beta_{n}}u^{l(\beta_{n})})}
\]
\end{defn}
\begin{cor}
The complex $X$ is an $L_{p}$-expander if and only if every pole
$\lambda$ of $\zeta_{\hat{B}_{\phi}}(u)$ satisfies $\left|\lambda\right|\le q^{(p-1)/p}$
or $\left|\lambda\right|=q$.
\end{cor}
The theorem is a generalization of the well known connection between
the expander property and the graph Zeta function, i.e. a $q+1$ regular
graph is an $L_{p}$-expander (in the notions of \cite{kamber2016lpgraph}),
if and only if every non trivial eigenvalue $\lambda$ of Hashimoto's
non backtracking operator, satisfies $\left|\lambda\right|\le q^{(p-1)/p}$.
See \cite{kamber2016lpgraph}, theorem 10.1.

\subsection*{Related Works}

Expander graphs are classical and we will not discuss their history
here. There are various works on how to extend the theory to high
dimensions, and in particular on how to extend the definition of a
Ramanujan graph to the definition of a Ramanujan complex, a quotient
of an affine building of type $\tilde{A}_{n}$. All the different
works are motivated (implicitly or explicitly) by the notion of a
tempered representation of a reductive group. 

The extension to ``cubical complexes'', i.e. quotient of buildings
of type $W=\tilde{A}_{1}\times...\times\tilde{A}_{1}$was considered
in \cite{jordan1999ramanujan}. This case requires considering the
adjacency operator for each summand separately. 

Based on previous works on the geometry of $\tilde{A}_{n}$buildings
(\cite{cartwright1999harmonic}), in \cite{cartwright2003ramanujan},
it was suggested to study the representation theory of the spherical
Hecke algebra acting on functions on the vertices of the complex.
The definition was slightly changed in \cite{lubotzky2005ramanujan},
definition 1.1, so it was equivalent to the fact that every spherical
non trivial subrepresentation of $L_{2}(\Gamma\backslash G)$ is tempered.
An Alon-Boppana type theorem was proved in \cite{li2004ramanujan}. 

Following Laffourge's work Ramanujan complexes were constructed in
\cite{li2004ramanujan,lubotzky2005explicit,lubotzky2005ramanujan,sarveniazi2007explicit},
satisfying the above definition.

The action of the Iwahori-Hecke algebra on functions on chambers of
the building and its quotients is classical, and was considered in
the seminal work of Borel (\cite{borel1976admissible}), from an algebraic
group point of view. A combinatorical theory, applied to any locally
finite regular building, appeared in \cite{parkinson2006buildings}.
Recently, the construction of the all dimensional Hecke algebra $H$
appeared in \cite{abramenko2015distance}. 

An approach to high-dimensional expanders is given in \cite{first2016ramanujan},
and is similar in spirit to definition \ref{def:Weakly-Contained}.
The approach there is slightly more general, dealing with arbitrary
simplicial complexes, but focuses on the Ramanujan case only (i.e.
$p=2$), and does not contain the explicit results for affine buildings.

\subsection*{Philosophy and Context of the Work}

Most of this work deals with general locally finite regular buildings
(see section \ref{sec:Definition-of-a-building}), generally without
the assumption of the existence of an automorphism group $G$. We
allow buildings with arbitrary parameter system $\vec{q}=(q_{s})_{s\in S}$,
not just a single parameter $q$ (the introduction is stated with
a single $q$ for simplicity). In particular, the only if case of
theorem \ref{thm:Main thm} holds for any locally finite regular building,
affine or not (see corollary \ref{cor:Main Theorem proof}), although
some change is required in definition \ref{def:Expander definition}
to deal with the thin case (see definition \ref{def:tempered-definition}
and lemma \ref{claim:Equivalence of temperedness definitions}). The
if case of theorem \ref{thm:Main thm} holds for any affine building,
and is actually based on theorem \ref{thm:Bounds_Theorem}.

In the affine case, a well known theorem of Tits shows that in thick
irreducible affine locally finite regular buildings of dimensions
greater than $2$ always comes from algebraic Lie group over a non-Archimedean
local field. However, this is not the case in dimensions 1 and 2.
Since our proofs are combinatorical, theorem \ref{thm:Bounds_Theorem}
applies to any affine building, while theorem \ref{thm:Zeta_Function_Thm_intro}
actually applies even more generally to arbitrary affine Iwahori-Hecke
algebras, with parameter systems $q_{s}>1$.

Note, however, that some of the theorems we cite, most notably Oh's
theorem \ref{thm:Oh's Theorem-intro}, are only known for algebraic
Lie groups over a non-Archimedean local field. We call this case shortly
the algebraic-group case.

\subsection*{Structure of the Work}

We divided this work into 4 parts. Very generally, in part I we present
the all-dimensional Hecke algebra $H$ and some of its basic properties.
Part II is devoted to the basic representation theory of $H$, and
in particular to $p$-tempered representations and $L_{p}$-expanders.
Part III is devoted to the spectrum of operators of $H$. While the
first three parts do not assume in general that $W$ is affine, part
IV contains the specific results to the affine case, which are the
main results of this work.

\subsubsection*{Part I- The Hecke Algebra}

In section \ref{sec:Definition-of-a-building} we present the $W$-metric
approach to buildings. It is standard and used mainly to set notations
for the rest of the work. In section \ref{sec:Distances-in-Buildings}
we discuss distances in buildings between two faces. This topic also
appears in \cite{abramenko2008buildings}. In section \ref{sec:The-All-Dimensional}
we define the algebra $H$ and describe it explicitly. The main result
here is proposition \ref{prop:GH algebra is an algebra} showing that
$H$ is indeed an algebra. Similar description of the algebra was
given in \cite{abramenko2015distance}, although we follow a slightly
different approach We also deduce that $H$ commutes with spherical
average operators, such as $\rho_{C_{0}}$ in the introduction.

In section \ref{sec:Building-Automorphisms-I} and section \ref{sec:Building-Automorphisms-I}
we relate $H$ to a sufficiently transitive (i.e Weyl transitive complete)
automorphism group $G$ of the building. We also show that $H$ can
be defined, as in the introduction, as the algebra of row and column
finite operators acting on $\mathbb{C}^{B}$ and commuting with the
natural $G$-action on this space (see proposition \ref{defined_by_automorphisms},
\cite{kamber2016lpgraph} proposition 2.3, and compare the approach
in \cite{first2016ramanujan}). We prove that the Iwahori-Hecke algebra
$H_{\phi}$ is isomorphic to the Hecke algebra $H_{G_{\phi}}(G)$
of $G$ with respect to the Iwahori subgroup $G_{\phi}$ (see \ref{prop:Isomorphism_to_G_Hecke}.
In the algebraic-group case this claim appeared back in \cite{borel1976admissible}).
The extension to $H$ is straightforward and is given in proposition
\ref{prop:Isomorphic_to_G_Hecke_Full}. 

\subsubsection*{Part III- Basic Representation Theory}

In section \ref{sec:Unitary-Representations} we discuss unitary representations
of the algebras $H$ and $H_{\phi}$, which is rather standard $\ast$-algebras
subject. In section \ref{sec:Induction-and-restriction} we prove
that there is a strong bijection between isomorphism classes of irreducible
representations of the two algebras (proposition \ref{prop:Induction-and-restriction-full-proposition},
similar result also appears in \cite{first2016ramanujan}, proposition
4.30). 

In section \ref{sec:Matrix-Coefficients} we show that matrix coefficients
allow us to consider every $H$-representation as a subrepresentation
of the action of $H$ on $\mathbb{C}^{B_{f}}$. This is analog to
the standard matrix coefficients argument which enables us to see
each group representation as a subrepresentation of the action of
$G$ on $\mathbb{C}^{G}$. Matrix coefficients lead us to study tempered
representations in section \ref{sec:-Tempered-Representations}. Our
definition of $p$-tempered or ``almost $L_{p}$'', definition \ref{def:tempered-definition},
is a little different then the standard $L_{p+\epsilon}$ , but is
equivalent in the algebraic-group case. This is done to handle the
amenable case, which happens if the building is a single apartment.
See also the corresponding definition for a (compactly generated)
group in section \ref{sec:Representations-of-the-automorphism-group}.
The definition of temperedness allows us to give the definition of
an expander complex (definition \ref{def:Expander_definition}).

In section \ref{sec:Representations-of-the-automorphism-group} we
discuss the connection between the representations of $G$ and of
$H$. A basic bijection between the right equivalent classes in simple
and well known (proposition \ref{Induction from HK to G- basic}).
However, it seems to be unknown in general if a finite dimensional
(respectively unitary) representation of $H_{\phi}$ induces to an
admissible (resp. unitary) representation of $G$. In theorem \ref{thm:Connection between Hphi and G rep}
we cite two strong results of Borel (\cite{borel1976admissible})
and Barbasch and Moy (\cite{barbasch1996unitary}) showing the answer
is yes in the algebraic-group case. Oh's theorem \ref{thm:Oh's Theorem-intro}
is discussed in section \ref{sec:Oh's-Theorem}. As said above, we
were unable to translate the proof of this theorem into the methodology
of this work. Therefore the theorem is only cited under the algebraic-group
assumption. 

\subsubsection*{Part III- Spectrum of Operators}

We present and discuss our definition of weak containment in section
\ref{sec:Spectrum-and-Weak}. Notice that our definition also covers
non-unitary representations, which is not standard. A more complete
treatment of weak containment, on the unitary case only, is given
in \cite{first2016ramanujan}. In section \ref{sec:The-Point-Spectrum}
we prove the ``only if'' part of theorem \ref{thm:Main thm}, which
is rather abstract, works in a general settings and is analog to \cite{cowling1988almost},
theorem 1. 

We then move to prove two generalizations of the ``Alon-Boppana theorem''.
In graphs there are two similar results connecting the spectrum of
the adjacency operator $A$ on a $q+1$ regular graph and on the $q+1$
regular tree. The first (sometimes called Serre's theorem) shows that
as the injectivity radius of the graph grows the spectrum of $A$
dense in $[-2\sqrt{q},2\sqrt{q}]$, and actually converges in distribution
to the spectral distribution of $A$ on the $q+1$ regular tree (see
\cite{mckay1981expected}). The density part of this theorem was generalized
in \cite{li2004ramanujan}, see also \cite{first2016ramanujan}, theorem
5.1. We present another version of this theorem in theorem \ref{Generelized-Serre}.
The classical Alon-Boppana theorem itself assume only that the graph
is connected and is large enough, and concerns only the largest eigenvalue
(sometimes is absolute value) of $A$. We prove a generalization of
this theorem in theorem \ref{Alon Boppana}.

\subsubsection*{Part IV- The affine case}

Before discussing the affine case we show how to extend the theory
to color rotating automorphisms in section \ref{sec:Color-Rotations}.
Since it adds some confusion we did not start with it, but it is essential
to the affine case since it allows working with the extended Iwahori-Hecke
algebra. This algebra acts naturally on ``recolored'' chambers of
the building, a subject not treated usually in works about buildings
and Hecke algebras.

In section \ref{sec:Affine-Buildings} we discuss root systems and
their connection to affine Coxeter groups. Most of the results are
standard and presented without proof. Theorem \ref{thm:Structure Theorem}
is a structure theorem of affine Coxeter group, which is of its own
interest. Similar result appears in \cite{gashi2012looping}.

In section \ref{sec:Temperedness-in-the} we discuss temperedness
in the affine case. Using the polynomial growth of $W$ in this case
we give a couple of different equivalent conditions for temperedness
in proposition \ref{prop:Equivalent temperdness conditions}. Then
we use theorem \ref{thm:Structure Theorem} to prove theorem \ref{thm:Zeta_Function_Thm_intro}.
We also explain the connection of the results to the generalized Poincare
series of $\hat{W}$, a notion from \textbf{\uline{\mbox{\cite{gyoja1983generalized}}}}
and \cite{hoffman2003remarks}. 

Sections \ref{sec:Bounds-on-Hecke-operators}-\ref{sec:How-to-Bound-operators}
are devoted to theorem \ref{thm:Bounds_Theorem}. Section \ref{sec:Bounds-on-Hecke-operators}
contains some consequences of this theorem. First, it derives the
if part of theorem \ref{thm:Main thm}. Secondly, we discuss some
versions of the Kunze-Stein theorem. In section \ref{sec:Application:-Average-Distance}
we prove theorem \ref{thm:Distance-and-Diameter} using theorem \ref{thm:Bounds_Theorem}.

We then turn to the proof of the theorem itself, which is based on
\cite{cowling1988almost}, theorem 2. Section \ref{sec:Retraction-into-Apartments}
is devoted to the connection between the well known Bernstein presentation
of the Iwahori-Hecke algebra and the building construction known as
sectorial retraction. Both ideas are versions of the Iwasawa decomposition
used in \cite{cowling1988almost}. The Bernstein presentation allows
us to write every operator as a sum of ``sectorial operators''.
We then show how to bound sectorial operators in section \ref{sec:How-to-Bound-operators},
thus proving theorem \ref{sec:Bounds-on-Hecke-operators} using some
bounds provided by the Bernstein presentation. 

\section*{Acknowledgments }

The author would like to thank his adviser Prof. Alex Lubotzky for
his guidance, support, and his unrelenting insistence on completing
this paper. Uriya First has read an early version of this article
and suggested many improvements, for which we are grateful.

\part{The Hecke Algebra}

\section{\label{sec:Definition-of-a-building}Buildings}

This section discusses the definition and basic properties of buildings.
We will follow the $W$-metric approach to buildings (as in Ronan's
book \cite{ronan2009lectures}).

\subsection*{Simplicial Complexes}

A \emph{simplicial complex} is $\left(B,V\right)$, $V$ some set,
$B\subset P(V)$, such that if $\phi\ne\sigma_{1}\subset\sigma_{2}\in B$
then $\sigma_{1}\in B$. The elements of $B$ are called \emph{faces}.
If a face is a subset of another face we say that the first face is
contained in the second. The \emph{dimension}\textbf{ }of a face is
the number of elements it has minus 1. We always assume dimensions
are finite.

A face is called \emph{maximal} if it is not a proper subset of another
face. We say that a simplicial complex is \emph{pure of dimension
$n$} if all its maximal faces have the same dimension $n$. Maximal
faces in a pure simplicial complex are called \emph{chambers}. Faces
of dimension $n-1$ are called \emph{panels}. Faces of dimension $0$
are called \emph{vertices}.

A face that is contained in a finite number of chambers is called
\emph{spherical}. The complex $B$ is called \emph{vertex spherical
}if every vertex (and hence every face) is contained in a finite number
of faces (such a simplicial complex is sometimes called locally finite,
but we reserve this term for a locally finite building. See below).
Two chambers $C_{1},C_{2}$ are \emph{adjacent} if they contain a
common panel. A pure simplicial complex is \emph{connected} if the
equivalence relation on chambers generated by adjacency has a single
equivalence class.

We say that a pure simplicial complex of dimension $n$ is \emph{colored}
if each panel is colored by a singleton $\{i\}\subset[n]=\{0,...,n\}$,
such that if two panels belong to the same chamber they have different
colors. The \emph{color} (sometimes called \emph{cotype}\textbf{ }in
building theory) $t(\sigma)$ of a face $\sigma$ is the union of
the colors of the panels containing it. It is a subset of $[n]=\{0,...,n\}$
and the color of a chamber is the empty set. We usually denote a color
by $I\subset[n]$. We denote the faces of color $I\subset[n]$ by
$B_{I}$. For example, the set of chambers is $B_{\phi}$. 

Two adjacent chambers $C_{1},C_{2}$ in a simplicial complex are called
\emph{$j$-adjacent} for $j\in[n]$ if they share a panel of color
$\{j\}$.

A \emph{color preserving isomorphism} (or simply an \emph{isomorphism})
between two colored complexes $f:\Sigma_{1}\rightarrow\Sigma_{2}$
is a bijection from the faces of $\Sigma_{1}$ to the faces of $\Sigma_{2}$
that preserves colors and containment of faces. In particular a \emph{(color
preserving) automorphism} is an isomorphism from a colored complex
to itself.

A \emph{color rotating isomorphism} $f:\Sigma_{1}\rightarrow\Sigma_{2}$
is a bijection that preserve containment and such that there exists
a bijection $\tau:[n]\rightarrow[n]$ satisfying that the color of
$f(\sigma)$ is $\tau(t(\sigma))$.

\subsection*{Coxeter Groups}

A \emph{Coxeter group} $\left(W,S\right)$ is given by a group $W$
and a finite set of generators $S=\{s_{0},...,s_{n}\}$, such that
$W$ is the group defined by the relations 
\[
W=\left\langle s_{i},i\in[n]\left|s_{i}^{2}=1,\left(s_{i}s_{j}\right)^{m_{i,j}}=1\right.\right\rangle 
\]

Whenever we write $W$ in this work we will implicitly assume we also
have a fixed set of generators $S$. We will always assume $S$ has
$n+1$ elements $s_{0},...,s_{n}$, so to each element $s_{i}\in S$
corresponds a color $i\in[n]$. We generally identify $S$ with $[n]$.
Therefore by abuse of notation we may relate to $I\subset[n]$ as
$I\subset S$.

The parameters $m_{i,j},\,i,j\in[n]$, $m_{i,i}=1$ are called the
\emph{Coxeter numbers} of the group. For every $I\subset[n]$ we define
a subgroup $W_{I}=\left\langle s_{i}:i\in I\right\rangle $. Such
a subgroup is called \emph{a parabolic subgroup} and it is also a
Coxeter group. 

A Coxeter group, or a parabolic subgroup $W_{I}$, is called \emph{spherical}
if it is finite, and in the parabolic subgroup case we also say that
the color $I$ is spherical. If every $I\subsetneq[n]$ is spherical,
we say that the Coxeter group $W$ is \emph{vertex spherical}. 

There exists a \emph{length function} $l:W\rightarrow\mathbb{N}$.
The length $l(w)$ of $w\in W$ is the length of the shortest word
in the generators $s_{i}$ expressing $w$.

A Coxeter group is called \emph{irreducible affine} if the following
conditions hold: 1.$W$ is an infinite subgroup of the isomorphism
group of a euclidean vector space $V$, generated by affine reflections.
2. $V$ has no nontrivial $W$ invariant subspace. 3. $W$ is discrete,
i.e. the number of $w\in W$ fixing some point $p\in V$ is finite.
If $W$ is irreducible affine then $W$ is vertex spherical.

A Coxeter group is called \emph{affine} if it is a finite direct sum
of irreducible affine Coxeter groups. For classifications of spherical
and affine Coxeter groups, also called Weyl groups, see \cite{ronan2009lectures}. 
\begin{prop}
A Coxeter group $(W,S)$ with $\left|S\right|=n+1$ has a geometric
realization as a connected colored simplicial complex of dimension
$n$- \emph{the} \emph{Coxeter complex} $\mathbb{W}$. The chambers
of $\mathbb{W}$ correspond to elements $w\in W$. Two chambers $w,w^{\prime}$
are adjacent if $w=w^{\prime}s_{i}$ and in this case the color of
the common panel is $i$. Faces of color $I\subsetneq S$ correspond
to cosets $wW_{I}$.
\end{prop}
\begin{proof}
See \cite{ronan2009lectures}, p.10.
\end{proof}
Notice that a face of type $I$ of $\mathbb{W}$ is spherical if and
only if $I$ is a spherical color. In particular, $W$ is vertex spherical
(as a Coxeter group) if and only if the Coxeter complex $\mathbb{W}$
is vertex spherical (as a simplicial complex).

\subsection*{Buildings}

Fix a Coxeter group $\left(W,S\right)$. We may identify the free
monoid on $n+1$ elements with $\left\{ id\right\} \cup_{m\ge1}[n]^{m}$.
If $\left|S\right|=n+1$ there exists a unique projection $p:\left\{ id\right\} \cup_{m\ge1}[n]^{m}\rightarrow W$
sending $i\rightarrow s_{i}.$
\begin{defn}
Let $B$ be a colored simplicial complex of dimension $n$. A \emph{gallery}
$\mathbb{G}$ is a finite sequence of chambers $\mathbb{G}=(C_{0},C_{1},...,C_{m})$,
such that $C_{i},C_{i+1}$are adjacent and $C_{i}\ne C_{i+1}$

For every gallery $\mathbb{G}$ we define the \emph{color} $t(\mathbb{G})=\left(\alpha_{0},...,\alpha_{m-1}\right)\in[n]^{m}$
when $C_{i},C_{i+1}$ are $\alpha_{i}$-adjacent. The \emph{Coxeter
color }of the gallery $\mathbb{G}$ is $t_{W}(\mathbb{G})=p(t(\mathbb{G}))\in W$.
\end{defn}

\begin{defn}
(See \cite{ronan2009lectures} chapter 3) A \emph{building} $(B,W,d)$
is given by:

1. A connected colored simplicial complex $B$ in which each panels
belongs to at least 2 chambers.

2. A Coxeter group $W$ with $n+1$ generators $S$.

3. A\textbf{ }\emph{distance function} $d:B_{\phi}\times B_{\phi}\rightarrow W$
.

Such that for every gallery of minimal length $\mathbb{G}$ between
$C,C^{\prime}$, the distance $d(C,C^{\prime})\in W$ equals $t_{W}(\mathbb{G})\in W$
.
\end{defn}
\begin{example}
The Coxeter complex $\mathbb{W}$ is a building, with the distance
function $d(w,w^{\prime})=w^{-1}w^{\prime}$.
\end{example}
If every panel in $B$ belongs to exactly 2 chambers we say that $B$
is \emph{thin}. If every panel in $B$ belongs to at least 3 chambers
we say that $B$ is \emph{thick}.

If every panel in $B$ belongs to a finite number of chambers we say
that $B$ is \emph{locally finite}. If $B$ is locally finite and
$W$ is vertex spherical (as a Coxeter group) then $B$ is also vertex
spherical (as a Coxeter group). 

A building is called\textbf{ }\emph{locally finite regular} if every
chamber $C$ has a constant number $q_{i}<\infty$ of adjacent chambers
of type $i$ where $q_{i}$ does not depend on $C$. The numbers $\overrightarrow{q}=\left(q_{i}\right)_{i\in[n]}$
are called the \emph{parameter system} of the building. We also write
$q_{i}=q_{s_{i}}$ if $s_{i}$ is the $i$-th element of $S$ (using
the identification of $S$ and $[n]$).
\begin{example}
Let $T$ be a tree (i.e. a graph without cycles) with no leaves. Color
its vertices with $0$ and $1$ such that each edge contains a vertex
of each color. The tree $T$ is an affine building with Coxeter group
$\tilde{A}_{1}=D_{\infty}=\left\langle s_{0},s_{1}:s_{0}^{2}=s_{1}^{2}=1\right\rangle $.
If each vertex is contained in a finite number of edges it is locally
finite and vertex spherical. If each vertex of type $i$ is contained
in $q_{i}$ edges (i.e $T$ is a biregular graph) then it is a locally
finite regular building. 
\end{example}
From now on we assume the building is locally finite regular with
parameter system $\overrightarrow{q}$ .
\begin{example}
The Coxeter complex $\mathbb{W}$ is a thin building and every thin
building is isomorphic to $\mathbb{W}$. The Coxeter complex is always
a locally finite building, even if it is not a locally finite simplicial
complex.
\end{example}
An \emph{apartment} $A$ in a building $B$ is a colored subcomplex
that is isomorphic to the Coxeter complex.
\begin{lem}
Every two chambers belong to an apartment of the building.
\end{lem}
\begin{proof}
\cite{ronan2009lectures}, p. 32
\end{proof}
We are mainly interested in buildings that are vertex spherical. However,
some interesting examples are not vertex spherical:
\begin{example}
If $W$ is an irreducible affine Coxeter group then every color $I\subsetneq[n]$
is spherical, the Coxeter group $W$ is vertex spherical and the Coxeter
complex has a structure of a locally finite simplicial complex. However,
since a general affine Coxeter group is a finite direct sum of irreducible
affine Coxeter groups, the Coxeter complex$\mathbb{W}$ is not a locally
finite simplicial complex. This complex $\mathbb{W}$ can also be
considered as a a \emph{locally finite polysimplicial complex} in
which each vertex is contained in a finite number of chambers. The
two points of views are equivalent. 

Consider for example $W=\tilde{A}_{1}\times\tilde{A}_{1}$ . It is
the Coxeter group with 4 generators- $s_{0},s_{1},s_{0}^{\prime},s_{1}^{\prime}$
and relations 
\[
s_{0}^{2}=s_{1}^{2}=s_{0}^{\prime2}=s_{1}^{\prime2}=\left(s_{0}s_{0}^{\prime}\right)^{2}=\left(s_{0}s_{1}^{\prime}\right)^{2}=\left(s_{1}s_{0}^{\prime}\right)^{2}=\left(s_{1}s_{1}^{\prime}\right)^{2}=1
\]

The corresponding Coxeter complex $\mathbb{W}$ can be considered
as a ``cube complex'' which is a product of two trees, with squares
(of color $\phi$), edges (of color $s_{0},s_{1},s_{0}^{\prime},s_{1}^{\prime}$)
and vertices (of colors $\{s_{0},s_{0}^{\prime}\}$,$\{s_{0},s_{1}^{\prime}\}$,$\{s_{1},s_{0}^{\prime}\}$,$\{s_{1},s_{1}^{\prime}\}$)
as faces. It can also be considered as a simplicial complex with chambers
of dimension 3, in which only the colors $\phi$,$s_{0}$,$s_{1}$,$s_{0}^{\prime}$,$s_{1}^{\prime}$,$\{s_{0},s_{0}^{\prime}\}$,$\{s_{0},s_{1}^{\prime}\}$,$\{s_{1},s_{0}^{\prime}\}$,$\{s_{1},s_{1}^{\prime}\}$
are spherical. As said, both views are equivalent and we use the simplicial
one in this work. See \cite{jordan1999ramanujan} for an expander
theory for cube complexes.

In terms of semisimple algebraic groups, an almost simple group over
a non-Archimedean local field (e.g $SL_{n}(Q_{p})$) has an irreducible
affine Weyl group and acts as as automorphism group on a vertex spherical
building. A product of two almost simple groups (e.g. $SL_{n}(Q_{p})\times SL_{m}(Q_{p^{\prime}})$)
acts on a ``polysimplicial'' building- a non vertex spherical building
 which is the product of the two buildings.
\end{example}

\section{\label{sec:Distances-in-Buildings}Distances in Buildings}
\begin{defn}
Let $\sigma_{1}\in B_{I_{1}},\sigma_{2}\in B_{I_{2}}$ be two faces
in $B$. Choose $\sigma_{1}\subset C_{1},\sigma_{2}\subset C_{2}$
with minimal distance between them and define $\tilde{d}(\sigma_{1},\sigma_{2})=d(C_{1},C_{2})\in W$.
The \emph{distance} $d(\sigma_{1},\sigma_{2})\in W_{I_{1}}\backslash W/W_{I_{2}}$
is the projection of $\tilde{d}(\sigma_{1},\sigma_{2})$ from $W$
to $W_{I_{1}}\backslash W/W_{I_{2}}$.
\end{defn}
One should prove it is well defined. The picture is explained in details
in \cite{abramenko2008buildings}, section 5.3.2 and we base our discussion
on it. Let us start with Coxeter groups. The following lemma is well
known:
\begin{lem}
\label{Coxeter decomposition lemma}1. Let $I\subset S$ be fixed.
Each coset $d=wW_{I}\in W/W_{I}$ has a unique shortest element $\tilde{d}$
and similarly in $W_{I}\backslash W$. 

Write $W^{I}\subset W$ for the set of shortest elements in the cosets
$W/W_{I}$. Similarly, write $^{I}W\subset W$ for the set of shortest
elements in the cosets $W_{I}\backslash W$.

2. Every $w\in W$ can be written uniquely as $w=w^{I}w_{I}$, $w^{I}\in W^{I},\,w_{I}\in W_{I}$
and in this case $l(w)=l(w^{I})+l(w_{I})$.
\end{lem}
\begin{proof}
See \cite{humphreys1992reflection} 5.12.
\end{proof}
There exists a similar statement for double cosets of Coxeter groups
which is a refinement of the lemma above. It is less standard and
we are mainly interested in the first statement of the following lemma.
Notice that if $I_{1}\subset I_{2}$ then $W_{I_{2}}$ is a parabolic
subgroup of $W_{I_{1}}$. Therefore, $^{I_{2}}(W_{I_{1}})$ and $(W_{I_{1}})^{I_{2}}$
are well defined.
\begin{lem}
\label{lem:1. Distance lemma}Let $I_{1},I_{2}\subsetneq S$ be fixed.
Then each double coset $d=W_{I_{1}}wW_{I_{2}}\in W_{I_{1}}\backslash W/W_{I_{2}}$
has a unique shortest element $\tilde{d}\in W$ with $d=W_{I_{1}}\tilde{d}W_{I_{2}}$.
Write $^{I_{1}}W^{I_{2}}$ for the set of such shortest elements.

Let $I_{3}=I_{1}\cap\tilde{d}I_{2}\tilde{d}^{-1}\subset I_{1}$, $I_{4}=I_{2}\cap\tilde{d}^{-1}I_{1}\tilde{d}\subset I_{2}$
(multiplication takes place in $W$, as $S\subset W$). We have a
bijection $\sim_{d}:W_{I_{3}}\leftrightarrow W_{I_{4}}$, given by
$w_{3}\sim_{d}w_{4}$, $w_{3}\in W_{I_{3}}$, $w_{4}\in W_{I_{4}}$
if $w_{3}\tilde{d}=\tilde{d}w_{4}$. Every element $w\in W$ with
$W_{I_{1}}wW_{I_{2}}=W_{I_{1}}\tilde{d}W_{I_{2}}$ can be decomposed
in $\left|W_{I_{3}}\right|=\left|W_{I_{4}}\right|$ ways as $w=w_{1}w_{3}\tilde{d}w_{4}w_{2}$,
with: $w_{1}\in\left(W_{I_{1}}\right)^{I_{3}}$, $w_{3}\in W_{I_{3}}$,
$\tilde{d}\in{}^{I_{1}}W^{I_{2}}$, $w_{4}\in W_{I_{4}}$, $w_{2}\in{}^{I_{4}}\ensuremath{(W_{I_{2}})}$
and in this case $l(w)=l(w_{1})+l(w_{3})+l(\tilde{d})+l(w_{4})+l(w_{2})$.
All the different decompositions are given by $w_{3}\rightarrow w_{3}\hat{w_{3}},\,w_{4}\rightarrow\hat{w}_{4}w_{4}$
for $\hat{w}_{3}\sim_{d}\hat{w}_{4}$.
\end{lem}
\begin{proof}
See \cite{abramenko2008buildings}, section 2.3.2, proposition 2.23.
\end{proof}
Let us return to buildings. The following proposition is a generalization
of lemma \ref{lem:1. Distance lemma} to buildings.
\begin{prop}
\label{claim:Distance claims}1. Let \textup{$p:W\rightarrow W_{I_{1}}\backslash W/W_{I_{2}}$
be the projection. Then }$d(\sigma_{1},\sigma_{2})=p(d(C_{1},C_{2}))$
does not depend on the chambers \textup{$\sigma_{1}\subset C_{1}\in B_{\phi},\sigma_{2}\subset C_{2}\in B_{\phi}$. }

2. The unique shortest representative $\tilde{d}$ of \textup{$d=d(\sigma_{1},\sigma_{2})$}
is the shortest distance between two chambers containing the faces.

3. Let $I_{3}=I_{1}\cap\tilde{d}I_{2}\tilde{d}^{-1},\,I_{4}=I_{2}\cap\tilde{d}^{-1}I_{1}\tilde{d}$.
There exists a face $\sigma_{3}$ of color $I_{3}$ containing $\sigma_{1}$
and a face of color $\sigma_{4}$ containing $\sigma_{2}$ such that:

3.a. Every two chambers $\sigma_{1}\subset C_{1},\sigma_{2}\subset C_{2}$
with $d(C_{1},C_{2})=\tilde{d}$ contain $\sigma_{3},\sigma_{4}$
respectively.

3.b. There exists a bijection $F:C_{\sigma_{3}}\rightarrow C_{\sigma_{4}}$
between the faces containing $\sigma_{3}$ and the faces containing
$\sigma_{4}$ such that $d(C,F(C))=\tilde{d}$.
\end{prop}
\begin{proof}
See \cite{abramenko2008buildings}, section 5.3.2.
\end{proof}
Notice that the distance we defined is not symmetric. However, for
$C_{1},C_{2}\in B_{\phi}$ if $d(C_{1},C_{2})=w$ then $d(C_{2},C_{1})=w^{-1}$. 
\begin{defn}
for $d=W_{I_{1}}wW_{I_{2}}\in W_{I_{1}}\backslash W/W_{I_{2}}$ we
define $d^{*}=W_{I_{2}}w^{-1}W_{I_{1}}\in W_{I_{2}}\backslash W/W_{I_{1}}$.
\end{defn}
\begin{prop}
If $d(\sigma_{1},\sigma_{2})=W_{I_{1}}wW_{I_{2}}=d$ then $d(\sigma_{2},\sigma_{1})=W_{I_{2}}w^{-1}W_{I_{1}}=d^{*}$.
\end{prop}
\begin{proof}
A minimal gallery from $C_{2}$ to $C_{1}$ is the reverse of a minimal
gallery from $C_{1}$ to $C_{2}.$ Now use the definition of distance
using minimal galleries.
\end{proof}
Recall that we we assume the building is locally finite regular with
parameter system $\overrightarrow{q}=\left(q_{i}\right)_{i\in[n]}$
. 
\begin{defn}
Assume that $I_{1}$ is spherical. The number of faces of distance
$d\in W_{I_{1}}\backslash W/W_{I_{2}}$ from a face $\sigma\in B_{I_{1}}$
is denoted by $q_{d}$.
\end{defn}
Our next goal is to prove that the definition of $q_{d}$ does not
depend on $\sigma$ and to calculate $q_{d}$ explicitly. It will
be done in proposition \ref{claim:Distance calculation}. 

Note that if $I_{1}$ is not spherical it is contained in an infinite
number of chambers, so $q_{d}$ is usually $\infty$.
\begin{defn}
For a finite subset $A\subset W$, denote $q_{A}=\sum_{w\in A}q_{w}$.
In particular $q_{W_{I}}=\sum_{w\in W_{I}}q_{w}$
\end{defn}
\begin{prop}
The number of chambers $C^{\prime}$ of distance $w\in W$ from a
chamber $C$ depends only on $w$. We denote it by $q_{w}$. If a
minimal decomposition is $w=s_{\alpha_{1}}\cdot...\cdot s_{\alpha_{l}}$
then $q_{w}=q_{\alpha_{1}}\cdot...\cdot q_{\alpha_{l}}$.
\end{prop}
\begin{proof}
If $w=w^{\prime}s$ with $l(w)=l(w^{\prime})+l(s)$ and $d(C,C^{\prime})=w=w^{\prime}s$
then there exists a single chamber $C^{\prime\prime}$ such that $d(C,C^{\prime\prime})=w^{\prime},d(C^{\prime\prime},C^{\prime})=s$
(it is standard- follows from \ref{claim:Distance claims} for example). 

On the other hand, if $d(C,C^{\prime\prime})=w^{\prime},d(C^{\prime\prime},C^{\prime})=s$
then $d(C,C^{\prime})=w=w^{\prime}s$. Therefore the number of such
$C^{\prime}$ is the number of pairs $C^{\prime\prime},C^{\prime}$
such that $d(C,C^{\prime\prime})=w^{\prime}$, $d(C^{\prime\prime},C^{\prime})=s$.
So inductively we have $q_{w}=q_{w^{\prime}}q_{s}$.
\end{proof}
\begin{prop}
Let $d=\tilde{d}W_{I_{2}}\in W/W_{I_{2}}$ (recall that $\tilde{d}$
is the shortest element in the coset). Then $q_{d}=q_{\tilde{d}}$.
\end{prop}
\begin{proof}
By \ref{claim:Distance claims}, every face of distance $d$ from
$C$ has a single chamber $C_{2}$ of distance $\tilde{d}$. On the
other hand every chamber of distance $\tilde{d}$ has a single face
of color $I_{2}$. The claim follows.
\end{proof}
\begin{prop}
\label{claim:Distance calculation}Assume $I_{1}$ is spherical. Let
$d=W_{I_{1}}\tilde{d}W_{I_{2}}$. Let $I_{3}=I_{1}\cap\tilde{d}I_{2}\tilde{d}^{-1},\,I_{4}=I_{2}\cap\tilde{d}^{-1}I_{1}\tilde{d}$.
Then $q_{d}=q_{W_{I_{1}}}/q_{W_{I_{3}}}\cdot q_{\tilde{d}}=q_{\left(W_{I_{1}}\right)^{I_{3}}}q_{\tilde{d}}$.
\end{prop}
\begin{proof}
Let $\sigma_{1}\in B_{I_{1}}$ be a face of color $I_{1}$. Look at
the pairs $\left(C_{1},\sigma_{2}\right)$ where $C_{1}$ is a chamber,
$\sigma_{1}\in C_{1}$ and $\sigma_{2}$ is of distance $\tilde{d}W_{I_{2}}$
from $C_{1}$. By the last claim the number of such pairs is $\left(\sum_{w\in W_{I_{1}}}q_{w}\right)q_{\tilde{d}}=q_{W_{I_{1}}}q_{\tilde{d}}$.
From \ref{claim:Distance claims} every face $\sigma_{3}$ is counted
$q_{W_{I_{3}}}$ times and the first equality follows. 

Finally, the decomposition $W_{I_{1}}=\left(W_{I_{1}}\right)^{I_{3}}W_{I_{3}}$,
its uniqueness, and the fact that it agrees with lengths of elements
give 
\[
q_{W_{I_{1}}}=\left(\sum_{w\in W_{I_{1}}}q_{w}\right)=\left(\sum_{w\in\left(W_{I_{1}}\right)^{I_{3}}}q_{w}\right)\left(\sum_{w\in W_{I_{3}}}q_{w}\right)=q_{\left(W_{I_{1}}\right)^{I_{3}}}q_{W_{I_{3}}}
\]
\end{proof}
\begin{defn}
\label{def:n_d definition}Using the notations of \ref{claim:Distance calculation},
denote $n_{d}=q_{W_{I_{3}}}$.
\end{defn}

\section{\label{sec:The-All-Dimensional}The All Dimensional Hecke Algebra}

We now define our algebra. We assume that all colors used below are
spherical. It is useful to let our algebra work on all spherical faces
simultaneously. 
\begin{defn}
From now on identify $B_{f}=\cup_{I:I\,spherical}B_{I}$.
\end{defn}

\begin{defn}
For $d\in W_{I_{1}}\backslash W/W_{I_{2}}$ we define the operator
$h_{d}:\mathbb{C}^{B_{f}}\rightarrow\mathbb{C}^{B_{f}}$ by 
\[
h_{d}(f)(\sigma_{1})=\begin{cases}
\sum_{\sigma_{2}:d(\sigma_{1},\sigma_{2})=d}f(\sigma_{2}) & \sigma_{1}\,\text{of\,color}\,I_{1}\\
0 & \sigma_{1}\,\text{not\,of\,color}\,I_{1}
\end{cases}
\]
\end{defn}
\begin{rem}
Notice that we assume here that the number of faces at distance $d$
is finite. This is a result of the regularity of the building.
\end{rem}
\begin{defn}
\emph{The all dimensional Hecke (ADH) algebra $H$} is 
\[
H=span\left\{ h_{d}:d\in W_{I_{1}}\backslash W/W_{I_{2}},I_{1},I_{2}\,\text{spherical}\right\} 
\]

The linear span of all the $h_{d},\,d\in W_{I_{1}}\backslash W/W_{I_{2}}$
is denoted $H_{I_{1},I_{2}}$. We also write $H_{I}=H_{I,I}$. 
\end{defn}
We identify $H_{I}$ and $H_{I_{1},I_{2}}$ with their natural embedding
in $H$.

We should prove that our Hecke algebra is indeed an algebra. Composition
of operators can be used to define multiplication $H_{I_{1},I_{2}}\times H_{I_{2},I_{3}}\rightarrow Hom_{\mathbb{C}}(\mathbb{C}^{B_{I_{3}}},\mathbb{C}^{B_{I_{1}}})$.
However, it is not so obvious why the result is in $H_{I_{1},I_{3}}$
and what it is. Let us start with a simple claim:
\begin{prop}
\label{prop:Iwahori Hecke }The algebra $H_{\phi}$ it is isomorphic
to the abstract \emph{Iwahori-Hecke algebra of $W$} - the algebra
generated by $h_{s},s\in S$ with the Iwahori-Hecke relations:

\[
\begin{array}{cccc}
h_{w}h_{s} & = & h_{ws}\,\,\,\,\,\,\,\,\,\,\,\,\,\,\,\, & if\,\,l(ws)=l(w)+1\\
h_{s}^{2} & = & q_{s}\cdot Id+(q_{s}-1)h_{s}
\end{array}
\]
\end{prop}
\begin{proof}
First, $H_{\phi}$ satisfies the Iwahori-Hecke relations. The relation
$h_{s}^{2}=q_{s}\cdot Id+(q_{s}-1)h_{s}$ is immediate. The relation
$h_{w}h_{s}=h_{ws}$ for $l(ws)=l(w)+1$ follows from the fact that
if $d(C,C_{1})=w,\,d(C_{1},C_{2})=s$ then $d(C,C_{2})=ws$.

Let $H_{\phi}^{\prime}$ be the Iwahori-Hecke algebra. It is well
known that $H_{\phi}^{\prime}$ is an algebra with basis $h_{w},\,w\in W$
(see\textbf{ \cite{humphreys1992reflection},} 7.1). Since $H_{\phi}$
satisfies the relations we have a homomorphism of algebra $H_{\phi}^{\prime}\rightarrow H_{\phi}$.
Since $H_{\phi}$ is spanned by the $h_{w}$, $w\in W$ this homomorphism
is onto and it remains to prove it has a trivial kernel. It is therefore
enough to prove that the $h_{w}$ are linearly independent in $H_{\phi}$.
This is immediate since every operator $h=\sum_{W}\alpha_{w}h_{w}\in H_{\phi}$
with $\alpha_{w}\ne0$ for some $w\in W$ acts non trivially on $\mathbb{C}^{B_{\phi}}$.
\end{proof}
Assume we have colors $I_{2}\subset I_{1}$. Notice that the larger
$I$ is, the face is smaller. That means that each face of color $I_{2}$
has exactly one subface of color $I_{1}$, and each face of color
$I_{1}$ is contained in a constant number of faces of color $I_{2}$.
We define:
\begin{defn}
\label{def:Induction-and-Restriction defnition}\emph{The (unsigned
colored) coboundary operator} $\delta_{I_{2},I_{1}}\in H_{I_{2},I_{1}}$,
$\delta_{I_{2},I_{1}}:\mathbb{C}^{B_{I_{1}}}\rightarrow\mathbb{C}^{B_{I_{2}}}$
is the element $h_{d}\in H_{I_{2},I_{1}}$ for $d=W_{I_{2}}\backslash id/W_{I_{1}}\in W_{I_{2}}\backslash W/W_{I_{1}}$.

\emph{The (unsigned colored) boundary operator} $\partial_{I_{1},I_{2}}\in H_{I_{1},I_{2}}$,
$\partial_{I_{1},I_{2}}:\mathbb{C}^{B_{I_{2}}}\rightarrow\mathbb{C}^{B_{I_{1}}}$
is the element $h_{d}\in H_{I_{1},I_{2}}$ for $d=W_{I_{1}}\backslash id/W_{I_{2}}\in W_{I_{1}}\backslash W/W_{I_{2}}$.

We denote $\delta_{I}=\delta_{\phi,I},\,\partial_{I}=\partial_{I,\phi},\,e_{I}=\delta_{I}\partial_{I}$.
\end{defn}
The coboundary operator $\delta_{I_{2},I_{1}}$ assigns to each face
of color $I_{2}$ the value of its subface of color $I_{1}$. The
boundary operator $\partial_{I_{1},I_{2}}$ assigns to each face of
color $I_{1}$ the sum of values of the faces of color $I_{2}$ containing
it.
\begin{rem}
It is worth noting that (when all faces of dimension $m$ are spherical),
the usual signed boundary and coboundary operators of $B$ between
dimensions $m,m+1$, belong to our algebra. Since the complex $B$
is colored, we have a natural ordering of the vertices of each simplex,
and therefore every simplex has a natural orientation, given (for
example) by ascending sequence of colors. The usual boundary and coboundary
operators are therefore sums with $\pm1$ coefficients of $\delta_{I_{2},I_{1}}$,
$\partial_{I_{1},I_{2}}$, $\left|I_{2}\right|=m$, $\left|I_{1}\right|=m+1$,
$I_{2}\subset I_{1}$. We will not use them at all in this work.
\end{rem}
\begin{lem}
\label{lem:Embedding_Claim}Let $I_{1},I_{2},I\subset[n]$ be spherical.

1. Let $d=W_{I_{1}}\backslash\tilde{d}/W_{I_{2}}\in W_{I_{1}}\backslash W/W_{I_{2}}$.
Then: $h_{d}=\left(1/n_{d}\right)\partial_{I_{1}}h_{\tilde{d}}\delta_{I_{2}}$.
(recall- $\tilde{d}\in W$ the shortest element in the double coset.
$n_{d}$ is defined in \ref{def:n_d definition}).

2. We have $\partial_{I}\delta_{I}=q_{W_{I}}1_{I}$, $e_{I}=\delta_{I}\partial_{I}=\sum_{w\in W_{I}}h_{w}\in H_{\phi}$
. Also $e_{I}^{2}=\left(\sum_{w\in W_{I}}q_{w}\right)e_{I}=q_{W_{I}}e_{I}$.

3. The algebra $H_{I}$ can be embedded in $H_{\phi}$ by $h_{d}\rightarrow q_{W_{I}}^{-1}n_{d}^{-1}e_{I}h_{\tilde{d}}e_{I}$.
\end{lem}
\begin{proof}
(1) follows from \ref{claim:Distance claims} and \ref{claim:Distance calculation}. 

For (2), Let $f\in\mathbb{C}^{B_{\phi}}$. Then $e_{I}f(C)=\delta_{I}\partial_{I}f(C)$
is equal to the sum of $f$ over all chambers sharing with $C$ its
face of color $I$. Every such chamber is of distance $w\in W_{I}$
from $C$, therefore $e_{I}=\sum_{w\in W_{I}}h_{w}$. Applying $e_{I}$
twice counts each element $\sum_{w\in W_{I}}q_{w}=q_{W_{I}}$ times.

(3) follows from (1) and (2).
\end{proof}
\begin{thm}
\label{prop:GH algebra is an algebra}The ADH algebra $H$ is indeed
an algebra. It is spanned by $h_{d}$, $d\in W_{I_{1}}\backslash W/W_{I_{2}}$,
for $I_{1},I_{2}$ spherical. The relations defining it are the Iwahori-Hecke
relations and the relations:

\textup{
\[
\begin{array}{cccc}
h_{d} & = & \left(1/n_{d}\right)\partial_{I_{1}}h_{\tilde{d}}\delta_{I_{2}}\\
\partial_{I}\delta_{I} & = & q_{W_{I}}1_{I}\\
\delta_{I}\partial_{I} & = & \sum_{w\in W_{I}}h_{w}\\
\delta_{I_{1}}\partial_{I_{2}} & = & 0 & \text{for }I_{1}\ne I_{2}
\end{array}
\]
}

The algebra $H$ is generated by the coboundary and boundary operators
$\delta_{I}=\delta_{\phi,I}$, $\partial_{I}=\partial_{I,\phi}$ ,
$I$ spherical, as well as the identity operator $1_{\phi}$ of $H_{\phi}$. 
\end{thm}
\begin{proof}
By lemma \ref{lem:Embedding_Claim} we have $H_{I_{1},I_{2}}=\mbox{span}\left\{ \partial_{I_{1}}\tilde{h}\delta_{I_{2}}:\tilde{h}\in H_{\phi}\right\} $.
Therefore for $h_{1}=\partial_{I_{1}}\tilde{h}_{1}\delta_{I_{2}}\in H_{I_{1},I_{2}},\,h_{2}=\partial_{I_{2}}\tilde{h}_{2}\delta_{I_{3}}\in H_{I_{2},I_{3}}$
we have $h_{1}h_{2}=\partial_{I_{1}}\tilde{h}_{1}\delta_{I_{2}}\partial_{I_{2}}\tilde{h}_{2}\delta_{I_{3}}$.But
$\tilde{h}_{1}\delta_{I_{2}}\partial_{I_{2}}\tilde{h}_{2}=\tilde{h}_{1}e_{I_{2}}\tilde{h}_{2}\in H_{\phi}$.
Therefore $h_{1}h_{2}\in H_{I_{1},I_{3}}$ and $H$ is an algebra.

The algebra $H_{\phi}$ is generated by $h_{s},s\in S$. Since $h_{s}=\delta_{\{s\}}\partial_{\{s\}}-1_{\phi}$
we get the result for it. Then lemma \ref{lem:Embedding_Claim} gives
the result for the entire algebra.
\end{proof}
The theorem actually shows that $H$ can be defined abstractly, for
any parameter system $\overrightarrow{q}$. The conditions on the
parameter system here are that $q_{s}=q_{s^{\prime}}$ when $m_{s,s^{\prime}}$
is odd (otherwise the Iwahori-Hecke algebra is not well defined, see\textbf{
\cite{humphreys1992reflection},} 7.1).
\begin{defn}
Denote by $1_{I}=h_{d}$, $d=W_{I}\backslash1/W_{I}$ the identity
operator of $H_{I}$.
\end{defn}
The ADH algebra has an identity element $1=\sum_{I\,spherical}1_{I}$.
It also has an adjunction, making it a $\ast$-algebra- the involution
$d=W_{I_{2}}wW_{I_{1}}\rightarrow d^{*}=W_{I_{1}}w^{-1}W_{I_{2}}$
extends to $\left(\alpha h_{d}\right)^{*}=\bar{\alpha}h_{d^{*}}$. 
\begin{prop}
\label{pro:Adjuction is involution}We have $\left(h_{1}h_{2}\right)^{*}=h_{2}^{*}h_{1}^{*}$
for every $h_{1},h_{2}\in H$.
\end{prop}
\begin{proof}
Consider the action of $H$ on $L_{2}(B_{f})$, i.e. the $L_{2}$
norm on $B_{f}$ defined by the inner product $\left\langle f,g\right\rangle =\sum_{\sigma\in B_{f}}\left\langle \bar{f}(\sigma),g(\sigma)\right\rangle $.
Then for $d\in W_{I_{1}}\backslash W/W_{I_{2}}$ , 
\[
\left\langle h_{d}f,g\right\rangle =\sum_{\sigma\in B_{f}}\sum_{\sigma':d(\sigma,\sigma')=d}\left\langle \bar{f}(\sigma'),g(\sigma)\right\rangle =\sum_{\sigma'\in B_{f}}\sum_{\sigma:d(\sigma',\sigma)=d^{*}}\left\langle \bar{f}(\sigma'),g(\sigma)\right\rangle =\left\langle f,h_{d^{*}}g\right\rangle 
\]

Therefore the $\ast$-operator on $H$ agrees with the $\ast$-operator
coming from the inner product on $L_{2}(B_{f})$. Since the homomorphism
$H\rightarrow\hom\left(L_{2}(B_{f}),L_{2}(B_{f})\right)$ is an embedding
of $H$, the result follows.
\end{proof}
The fact that the ADH algebra is well defined and its algebra relations
can help to understand the geometry of the building. In particular,
one can show (Compare \cite{abramenko2015distance} theorem 3.1):
\begin{lem}
\label{prop:strongly regular}Let $B$ be a locally finite regular
building. Let $\sigma_{0}\in B_{I_{0}}$, $\sigma_{1}\in B_{I_{1}}$
be spherical faces of $B$ of distance $d_{1}\in W_{I_{0}}\backslash W/W_{I_{1}}$.
Let $d_{2},\tilde{d}$ be distances $d_{2}\in W_{I_{0}}\backslash W/W_{I_{2}}$,
$\tilde{d}\in W_{I_{1}}\backslash W/W_{I_{2}}$. Let $M$ be the number
of faces $\sigma_{2}\in B_{I_{2}}$ with $d(\sigma_{0},\sigma_{2})=d_{2}$,
$d(\sigma_{1},\sigma_{2})=\tilde{d}$. Then $M$ is a polynomial function
on the parameter system $\left(q_{i}\right),\,i\in S$ which depends
only on $d_{0},d_{1}$ and \textup{$\tilde{d}$}. More precisely:
Let $\alpha$ be the coefficient of $h_{d_{1}}$ in the decomposition
of $h_{d_{2}}h_{\tilde{d}^{*}}$into a sum of basis elements. Then
$M=\alpha$.
\end{lem}
\begin{proof}
To show that $M=\alpha$ choose a function $f_{\sigma_{1}}$ with
$f_{\sigma_{1}}(\sigma_{1})=1$, $f_{\sigma_{1}}(\sigma)=0$ for $\sigma\ne\sigma_{1}$.
Then $M=h_{d_{2}}h_{\tilde{d}^{*}}f_{\sigma_{1}}(\sigma_{0})=\alpha h_{d_{1}}f_{\sigma_{1}}(\sigma_{0})=\alpha$.
\end{proof}
\begin{lem}
\label{Strongly Regular Lemma}Let $\sigma_{0}$ be fixed, $\sigma_{1}$
be at distance $d_{1}=d(\sigma_{0},\sigma_{1})$ and $\sigma_{2}$
be at distance $d_{2}=d(\sigma_{0},\sigma_{2})$. Let $\tilde{d}$
be some distance. Let $M_{1}$ be the number of $\sigma_{1}^{\prime}$
with $d(\sigma_{0},\sigma_{1}^{\prime})=d_{1}$, $d(\sigma_{1}^{\prime},\sigma_{2})=\tilde{d}$.
Let $M_{2}$ be the number of $\sigma_{2}^{\prime}$ with $d(\sigma_{0},\sigma_{2}^{\prime})=d_{2}$,
$d(\sigma_{1},\sigma_{2}^{\prime})=\tilde{d}$. Then $M_{1}q_{d_{2}}=M_{2}q_{d_{1}}$.
\end{lem}
\begin{proof}
Consider the number $M$ of pairs $(\sigma_{1}^{\prime},\sigma_{2}^{\prime})$
with $d(\sigma_{0},\sigma_{1}^{\prime})=d_{1}$, $d(\sigma_{0},\sigma_{2}^{\prime})=d_{2}$,
$d(\sigma_{1}^{\prime},\sigma_{2}^{\prime})=\tilde{d}$. By proposition
\ref{prop:strongly regular} the number $M_{2}$ of $\sigma_{2}^{\prime}$
corresponding to a single $\sigma_{1}$ does not depend on $\sigma_{1}$.
Therefore $M=M_{2}q_{d_{1}}$ and by symmetry also $M=M_{1}q_{d_{2}}$.
\end{proof}
We can now show that the operators of $H$ commute with spherical
operators.
\begin{defn}
\label{def: spherical operators}For any spherical face $\sigma_{0}$
we define a \emph{spherical average operator} $\rho_{\sigma_{0}}:\mathbb{C}^{B_{f}}\rightarrow\mathbb{C}^{B_{f}}$
by 

\[
\rho_{\sigma_{0}}f(\sigma)=\frac{1}{q_{d(\sigma_{0},\sigma)}}\sum_{\sigma':d(\sigma_{0},\sigma)=d(\sigma_{0},\sigma')}f(\sigma')=\frac{1}{q_{d(\sigma_{0},\sigma)}}h_{d(\sigma_{0},\sigma)}f(\sigma_{0})
\]
\end{defn}
\begin{prop}
\label{Commutes_With_Spherical_Average}The ADH algebra commutes with
spherical average operators. That is- for any $h\in H$ and any face
$\sigma_{0}$ we have $\rho_{\sigma_{0}}h=h\rho_{\sigma_{0}}$.
\end{prop}
\begin{proof}
Let $\tilde{d}\in W_{I}\backslash W/W_{I_{2}}$ be a distance. We
need to prove that $(h_{\tilde{d}}\rho_{\sigma_{0}}f)(\sigma_{1})=(\rho_{\sigma_{0}}h_{\tilde{d}}f)(\sigma_{1})$
for every $\sigma_{1}\in B_{f}$. It if enough to prove it for the
functions $f_{\sigma_{2}}$, $\sigma_{2}\in B_{f}$ fixed, defined
by $f_{\sigma_{2}}(\sigma_{2})=1$, $f(\sigma)=0$ for $\sigma\ne\sigma_{2}$.

Fix $\sigma_{0},$$\sigma_{1}$, $\sigma_{2}$, $\tilde{d}$. By definition,
$(h_{\tilde{d}}\rho_{\sigma_{0}}f_{\sigma_{2}})(\sigma_{1})$ equals
the number of $\sigma_{2}^{\prime}$ with $d(\sigma_{0},\sigma_{2}^{\prime})=d(\sigma_{0},\sigma_{2})=d_{2}$,
$d(\sigma_{1},\sigma_{2}^{\prime})=\tilde{d}$, divided by $q_{d_{2}}$.
Similarly, $(\rho_{\sigma_{0}}h_{\tilde{d}}f_{\sigma_{2}})(\sigma_{1})$
equals the number of $\sigma_{1}^{\prime}$ with $d(\sigma_{0},\sigma_{1}^{\prime})=d(\sigma_{0},\sigma_{1})=d_{1}$,
$d(\sigma_{1}^{\prime},\sigma_{2})=\tilde{d}$, divided by $q_{d_{1}}$.
By lemma \ref{Strongly Regular Lemma}, both numbers are equal.
\end{proof}

\section{\label{sec:Building-Automorphisms-I}Building Automorphisms I}

Recall that an automorphism $\gamma$ of the building is an automorphism
of the simplicial complex respecting colors of the faces.
\begin{defn}
Let $G$ be a subgroup of the automorphism group of the building.
The group $G$ is called:
\begin{itemize}
\item \emph{Chamber transitive} if for every 2 chambers $C_{1},C_{,2}$
we have an automorphism $\gamma\in G$ such that $\gamma(C_{1})=C_{2}$.
\item \emph{Weyl transitive} if for every 4 chambers $C_{1},C_{,2},C_{3},C_{4}$
such that $d(C_{1},C_{2})=d(C_{3},C_{4})$ we have an automorphism
$\gamma\in G$ such that $\gamma(C_{1})=C_{2},\gamma(C_{3})=C_{4}$.
\item \emph{Strongly transitive} , or has a \emph{BN- pair} if for every
2 chambers $C_{1},C_{,2}$ and two apartments containing them $C_{1}\in A_{1},C_{2}\in A_{2}$
we have an automorphism $\gamma\in G$ such that $\gamma(C_{1})=C_{2},\gamma(A_{1})=A_{2}$.
\end{itemize}
\end{defn}
Notice that the strongly transitive notion actually depends on the
choice of apartments of the building. 
\begin{lem}
An automorphism group that is strongly transitive is Weyl transitive.
An automorphism group that is Weyl transitive is chamber transitive.
A building that has a chamber transitive automorphism group is regular.
\end{lem}
\begin{proof}
Follows from the definitions.
\end{proof}
\begin{lem}
The distance between the faces $\sigma_{1},\sigma_{2}$ is preserved
by every (color preserving) automorphism $\gamma:B\rightarrow B$.
\end{lem}
\begin{proof}
Such automorphisms preserve distances between chambers. Therefore
the chambers $\sigma_{1}\subset C_{1},\sigma_{2}\subset C_{2}$ with
minimal distance between them go to two chambers $\gamma(\sigma_{1})\subset\gamma(C_{1}),\gamma(\sigma_{2})\subset\gamma(C_{2})$
with minimal distance between them.
\end{proof}
\begin{rem}
A color rotating automorphism $\gamma$ defines a permutation $\omega:S\rightarrow S$.
Then the distance is changed according to the extension $\omega:W\rightarrow W$,$\omega:W_{I}\rightarrow W_{\omega(I)}$
given by $\omega(s_{i})=s_{\omega(i)}$.
\end{rem}
\begin{lem}
An automorphism group $G$ is Weyl transitive if and only if for every
4 faces $\sigma_{1},\sigma_{,2},\sigma_{3},\sigma_{4}$ such that
$d(\sigma_{1},\sigma_{2})=d(\sigma_{3},\sigma_{4})$ we have an automorphism
$\gamma\in G$ such that $\gamma(\sigma_{1})=\sigma_{2},\gamma(\sigma_{3})=\sigma_{4}$.
\end{lem}
\begin{proof}
Follows from the definition of a distance between faces.
\end{proof}
\begin{defn}
\label{def:Row and Column finite}Let $S$ be a discrete set. We say
that a linear operator $h:\mathbb{C}^{S}\rightarrow\mathbb{C}^{S}$
is \emph{row and column finite} if it can be written as $hf(x)=\sum_{y\in S}\alpha_{x,y}f(y)$,
for some $\alpha:S\times S\rightarrow\mathbb{C}$, with $\left|\left\{ y:\alpha_{x.y}\ne0\right\} \right|<\infty$,
$\left|\left\{ y:\alpha_{y,x}\ne0\right\} \right|<\infty$ for every
$x\in S$. 
\end{defn}
If $\gamma$ is an automorphism of the building we let $\gamma$ act
on $\mathbb{C}^{B_{f}}$, $\gamma:\mathbb{C}^{B_{f}}\rightarrow\mathbb{C}^{B_{f}}$
by $\gamma f(\sigma)=f(\gamma^{-1}(\sigma))$.
\begin{prop}
\label{defined_by_automorphisms}1. The ADH Algebra commutes with
every color preserving automorphism of the building.

2. Assume that $G$ is Weyl transitive. If a row and column finite
linear transform $h:\mathbb{C}^{B_{f}}\rightarrow\mathbb{C}^{B_{f}}$
commutes with every automorphism $\gamma\in G$ then it belongs to
the ADH algebra $H$.
\end{prop}
\begin{proof}
Claim (1) follows from the fact that automorphisms respect distances
between faces. 

As for (2), write $h:\mathbb{C}^{B_{f}}\rightarrow\mathbb{C}^{B_{f}}$
as $hf(\sigma)=\sum_{y\in V_{T}}\alpha_{\sigma,\sigma'}f(\sigma')$,
as in the definition of a row and column finite operator. Assume $h$
commutes with every $\gamma\in\mbox{Aut}(T)$. Let $\sigma_{1},\sigma_{,2},\sigma_{3},\sigma_{4}$
be faces such that $d(\sigma_{1},\sigma_{2})=d(\sigma_{3},\sigma_{4})$.
By the last lemma there exists $\gamma\in\mbox{Aut}(T)$ such that
$\gamma(\sigma_{1})=\sigma_{3},\,\gamma(\sigma_{2})=\sigma_{4}$.
Let $f_{\sigma_{2}}$ be the characteristic function of $\{\sigma_{2}\}$
and write $h\gamma f_{\sigma_{2}}=\gamma hf_{\sigma_{2}}$. Then 
\[
\alpha_{\sigma_{1},\sigma_{2}}=(h\gamma f_{\sigma_{2}})(\sigma_{3})=(\gamma hf_{\sigma_{2}})(\sigma_{3})=\alpha_{\sigma_{3},\sigma_{4}}
\]

Therefore $\alpha_{\sigma,\sigma'}$ depends only on $d(\sigma,\sigma')$
and we can write $\alpha_{\sigma,\sigma'}=\alpha_{d(\sigma,\sigma')}$.
Therefore $hf(\sigma)=\sum_{y\in V_{T}}\alpha_{d(\sigma,\sigma')}f(\sigma')=\sum_{d}\alpha_{d}h_{d}f(\sigma')$
and $h\in H$.
\end{proof}
\begin{cor}
If $G$ is Weyl transitive then the ADH algebra $H$ is the algebra
of row and column finite operators $h:\mathbb{C}^{B_{f}}\rightarrow\mathbb{C}^{B_{f}}$
commuting with every automorphism $\gamma\in G$.
\end{cor}

\section{\label{sec:Building-Automorphisms-II}Building Automorphisms II-
Hecke Algebras of Groups}

Every semisimple group $G$ over a non-Archimedean local field acts
on a certain affine building as a color rotating automorphism group
which is Weyl transitive. There is usually some normal subgroup of
$G$ (containing the center) that acts trivially. By moving to a quotient
of $G$ by this normal subgroup we will generally identify $G$ with
its image in the automorphism group. We will further take the finite
index subgroup of $G$ which acts by color preserving automorphism.
On the other hand if the dimension of an affine building is $\ge3$
a well known theorem of Tits says that its full automorphism group
$G$ is an algebraic group (see \cite{ronan2009lectures} ).

It will be useful to define our Hecke algebra in terms of the automorphism
group directly. We start with some basic claims about general locally
profinite groups (or totally disconnected locally compact groups).
We follow\textbf{ }\textbf{\uline{\mbox{\cite{casselman1974introduction}}}}. 

Let $G'$ be a general locally profinite group, i.e $G'$ is a Hausdorff
topological group that has a basis of the identity composed of compact
open subgroups. Let $K$ be a compact open subgroup of $G'$. Fix
a Haar measure $\mu$ of $G'$. We will also assume $G'$ is unimodular. 
\begin{defn}
The \emph{Hecke algebra with respect to $K$}, $H_{K}(G')\subset H(G')$,
is the set of compactly supported functions $f:G'\rightarrow\mathbb{C}$
with $f(k_{1}gk_{2})=f(g)$ for every $k_{1},k_{2}\in K$, $g\in G'$.
It is spanned by the characteristic functions $1_{KgK}$, $KgK\in K\backslash G'/K$.

The \emph{Hecke Algebra} $H(G')$ of $G'$ is $H(G')=\cup_{K}H_{K}(G')$
where $K$ goes over the set of compact open subgroups.

We denote by $C(K\backslash G')$ the set of functions $f:G'\rightarrow\mathbb{C}$
such that $f(kg)=f(g)$ for every $k\in K,\,g\in G'$

We denote $C_{l}^{\infty}(G')=\cup_{K}C(K\backslash G')$. Notice
that $H(G')\subset C_{l}^{\infty}(G')$.
\end{defn}
\begin{prop}
We can define a convolution of $h\in H(G')$ on $f\in\mathbb{C}_{l}^{\infty}(G')$
by the integral $h\ast f(y)=\int_{G}h(x)f(x^{-1}y)dx$. When restricted
to $H_{K}(G')$ the convolution defines an algebra structure on it,
as $H_{K}(G')\subset H(G')\subset C_{l}^{\infty}(G')$ and $H_{K}(G')\ast H_{K}(G')=H_{K}(G')$.
It also defines an action of $H_{K}(G')$ on $C(K\backslash G')$.
The identity of $H_{K}(G')$ is $\mu^{-1}(K)\cdot1_{K}$. This convolution
defines an idempotented algebra structure on $H(G')$.

One can define adjunction in $H_{K}(G')$ . If $G'$ is unimodular,
the adjunction is given by $*:1_{KgK}\rightarrow1_{Kg^{-1}K}$ on
$H_{K}(G')$. 
\end{prop}
\begin{proof}
See \textbf{\uline{\mbox{\cite{casselman1974introduction}}}},
section 2.1 for everything except adjunction. Adjunction can be verified
directly and is left to the reader.
\end{proof}
\begin{rem}
As $K$ is open and $h$ is compactly supported the integral is actually
a finite sum.
\end{rem}
We now connect the general considerations above to an automorphism
group $G$ of the building and the Hecke algebras we defined in the
previous sections. Let $G$ be a subgroup of the color preserving
automorphism group of the building and assume it is Weyl transitive.
Let $C_{0}$ be a fixed chamber and $G_{\phi}\subset G$ its stabilizer
(i.e an Iwahori subgroup, usually denoted $I$). Let $G_{I}$ be the
stabilizer of the face $\sigma_{0}\subset C_{0}$ is of color $I$.
Give $B$ the discrete topology. We can topologize $G$ by the compact
open topology, i.e. a basis of open sets are sets containing, for
$A\subset B_{\phi}$ finite, $U\subset B_{\phi}$ arbitrary, all automorphisms
$\gamma$ with $\gamma(A)\subset U$. As $X$ and $Y$ are discrete,
this topology is equivalent to the topology defined by pointwise convergence
of a sequence of functions.
\begin{defn}
Call $G$ \emph{complete} if it is closed in the compact open topology
defined on the entire automorphism group of $B$. 
\end{defn}
Equivalently, $G$ is complete if and only for every sequence of automorphisms
$\{\gamma_{n}\}\subset G$ that converges pointwise to $\gamma:B\rightarrow B$
we have $\gamma\in G$. A complete Weyl transitive automorphism group
is also strongly transitive.
\begin{lem}
If $G$ is complete and $I$ is spherical then $G_{I}$ is an open
compact subgroup. 
\end{lem}
\begin{proof}
We will prove for $G_{\phi}$. The proof for $G_{I}$ is similar.
The fact that $G_{\phi}$ is open is by definition of the compact
open topology. 

The locally finite assumption ($q_{s}<\infty,s\in S$) guarantee that
the number of chambers at length at most $l$ from $C_{0}$ is finite.
Therefore every sequence $\{\gamma_{n}\}\subset G_{\phi}$ has a subsequence
agreeing on all chambers of distance at most $l$ from $C_{0}.$ Therefore
$\{\gamma_{n}\}$ has a subsequence converging pointwise to an automorphism
$\gamma$ and by completeness $\gamma\in G$ . Since $\gamma(C_{0})=C_{0}$,
$\gamma\in G_{\phi}$. 
\end{proof}
We will assume from now on that $G$ is complete. 

A basis for the identify of $G$ is composed of the compact open subgroups
$K_{m}$, $m\ge0$, where $K_{m}\subset G_{\phi}$ is the subgroup
containing all the automorphisms fixing all chambers $C$ with $l(d(C_{0},C))\le m$.
Since the $K_{m}$, $m\in\mathbb{N}$ are compact the topology gives
$G$ the structure of a locally profinite group. We now prove that
$G$ is unimodular.
\begin{prop}
Let $G$ be a complete Weyl transitive automorphism group of a locally
finite regular building. Then $G$ is unimodular.
\end{prop}
\begin{proof}
For $K'$ a compact open subgroup and $A$ a finite a sum of left
$K'$ cosets, let $[A:K']$ be the number of such cosets. Assume $\mu$
is left invariant. Let $K$ be a compact open subgroup and assume
$\mu(K)=1$. The modular character does not depend on this normalization.
We have the following formula for the modular character (see proof
below):

\[
\delta(g)=[KgK:K][Kg^{-1}K:K]^{-1}
\]

Choosing $K=G_{\phi}$ we have $[KgK:K]=q_{w}$, where $w=d(C_{0},g(C_{0}))$.
Since $d(C_{0},g^{-1}(C_{0}))=w^{-1}$ we have $[Kg^{-1}K:K]=q_{w^{-1}}=q_{w}$
and $\delta(g)=1.$ 

Proof of the formula: 
\[
\delta(g)=\mu(Kg)=\mu(g^{-1}Kg)=[K:K\cap g^{-1}Kg]^{-1}[g^{-1}Kg:K\cap g^{-1}Kg]
\]

Let $H_{1},H_{2}$ be subgroups of some big group. We have a set bijection
$H_{1}/(H_{1}\cap H_{2})\cong H_{1}H_{2}/H_{2}$ which is a weak version
of the second isomorphism theorem. Therefore:
\[
\begin{array}{ccc}
[g^{-1}Kg:K\cap g^{-1}Kg] & = & g^{-1}KgK:K]=[KgK:K]\\{}
[K:K\cap g^{-1}Kg] & = & gKg^{-1}:gKg^{-1}\cap K]=[Kg^{-1}K:K]
\end{array}
\]
\end{proof}
The fact that $G$ is chamber transitive allow us to identify the
set of chambers of the building \textbf{$B_{\phi}$} with cosets of
$G_{\phi}$ in $G$, i.e $B_{\phi}\cong G_{\phi}\backslash G$.. The
faces of the building of color $I$ correspond similarly to $B_{I}\cong G_{I}\backslash G$.
We can therefore identify $\mathbb{C}^{B_{I}}\cong C(G_{I}\backslash G)$.

The fact that $G$ is Weyl transitive allow us to identify distances
between chambers with double cosets $G_{\phi}\backslash G/G_{\phi}$,
and distances between faces of color $I$ by $G_{I}\backslash G/G_{I}$.
We can therefore identify $H_{I}\cong H_{G_{I}}(G)$ as a vector space.
\begin{prop}
\label{prop:Isomorphism_to_G_Hecke}Let $G$ be a complete Weyl transitive
automorphism group of the building $B$. Let $C_{0}$ be a fixed chamber
and let $G_{I}$ be the stabilizer of the face $\sigma_{I}^{0}$ of
color $I$ of $C_{0}$. 

Then $H_{I}\cong H_{G_{I}}(G)$ as an algebra and its action on $\mathbb{C}^{B_{I}}$
is the same as the action of $H_{G_{I}}(G)$ on $C(G_{I}\backslash G)$.
\end{prop}
\begin{proof}
Let $d\in W_{I}\backslash W/W_{I}$ and then consider the operator
$h_{d}\in H_{I}$. Let $G_{d}\subset G$ be the subset consisting
of all automorphisms $\gamma\in G$ such that $d(\sigma_{I}^{0},\gamma(\sigma_{I}^{0}))=d$.
We claim that the element $\tilde{h}_{d}=\frac{1}{\left|G_{I}\right|}1_{G_{d}}\in H_{G_{I}}(G)$
defines a homomorphism $H_{I}\rightarrow H_{G_{I}}(G)$, $h_{d}\rightarrow\tilde{h}_{d}$.
The following are immediate:

a. The set $G_{d}$ is a double coset- $G_{d}=G_{I}\gamma_{d}G_{I},$
where $\gamma_{d}\in G$ is some element of $G_{d}$, i.e sending
$\sigma_{I}^{0}$ to a face of distance $d$. 

b. The set $G_{d}$ is a sum of $q_{d}$ right cosets in $G/G_{I}$,
each right coset $\gamma G_{I}$ containing all automorphisms sending
$\sigma_{I}^{0}$ to a specific face of distance $d$. Therefore $\left|G_{d}\right|=q_{d}\left|G_{I}\right|$

c. For $\gamma_{0}\in G$, the set $\gamma_{0}G_{d}\subset G$ is
the set of all automorphisms sending $\gamma_{0}(\sigma_{I}^{0})$
to a face of distance $d$ of it.

Let $f\in C(G_{I}\backslash G)$. Since $G_{I}\backslash G\cong B_{I}$,
there corresponds a function $f_{B}\in\mathbb{C}^{B_{I}}$. Explicitly,
$f(\gamma)=f_{B}(\gamma^{-1}\cdot\sigma_{I}^{0})$. We have:

\[
\tilde{h}_{d}\ast f(\gamma_{0})=\int_{G}h(\gamma)f(\gamma^{-1}\gamma_{0})d\gamma=\frac{1}{\left|G_{I}\right|}\int_{G_{d}}f(\gamma^{-1}\gamma_{0})d\gamma=\frac{1}{\left|G_{I}\right|}\int_{\gamma_{0}^{-1}G_{d}}f(\gamma^{-1})d\gamma
\]

By b. and c. above and the correspondence between $f$ and $f_{B}$,
the last value is the sum of $f_{B}$ over the $q_{d}$ faces of distance
$d$ from $\gamma_{0}^{-1}(\sigma_{I}^{0})$. Therefore the action
of $\tilde{h}_{d}\in H_{G_{I}}(G)$ on $C(G_{I}\backslash G)$ and
the action of $h_{d}\in H_{I}$ on $f_{B}\in\mathbb{C}^{B_{I}}$ agree,
i.e. $\left(\tilde{h}_{d}\ast f\right)_{B}=h_{d}\cdot f_{B}$ and
we are done.
\end{proof}
We now extend the above to show similar results for $H$. In the general
context as in the beginning of this section define:
\begin{defn}
For $K_{1},K_{2}\subset G'$ compact open let $H_{K_{1},K_{2}}(G')$
be the set of compactly supported functions $f\in\mathbb{C}_{c}(G')$
with $f(k_{1}gk_{2})=f(g)$ for every $k_{1}\in K$, $k_{2}\in K$.

Define convolution on $*:H_{K_{1},K_{2}}(G')\times H_{K_{3},K_{4}}(G')\rightarrow H_{K_{1},K_{4}}$
by: 
\[
f_{1}*f_{2}=\begin{cases}
0 & K_{2}\ne K_{3}\\
\mbox{as in \ensuremath{H(G')}} & K_{2}=K_{3}
\end{cases}
\]

Define convolution on $*:H_{K_{1},K_{2}}(G')\times C(K_{3}\backslash G')\rightarrow C(K_{1}\backslash G')$
by:
\[
f_{1}*f_{2}=\begin{cases}
0 & K_{2}\ne K_{3}\\
\mbox{as in \ensuremath{H(G')}} & K_{2}=K_{3}
\end{cases}
\]
\end{defn}
If $K_{1},K_{2}$ are two different compact open subgroups the above
can give an algebra structure on $H_{K_{1},K_{1}}\oplus H_{K_{1},K_{2}}\oplus H_{K_{2},K_{1}}\oplus H_{K_{2},K_{2}}$.
We can therefore state:
\begin{prop}
\label{prop:Isomorphic_to_G_Hecke_Full}Let $G$ be a complete Weyl
transitive automorphism group of the building $B$. Let $C_{0}$ be
a fixed chamber and $G_{I}$ the stabilizer of the face of color $I$
contained in $C_{0}$. 

Then the ADH algebra $H$ is isomorphic as algebra to $\oplus_{I,J\,finite\,types}H_{G_{I},G_{J}}(G)$.
Its action on the building is given by the action of the algebra $\oplus_{I,J\,\text{spherical}}H_{G_{I},G_{J}}(G)$
on $\mathbb{C}^{B_{f}}\cong\oplus_{I\,\text{spherical}}\mathbb{C}^{B_{I}}\cong\oplus_{I\,\text{spherical}}C(G_{I}\backslash G)$.
\end{prop}
\begin{rem}
If $B$ has a complete Weyl transitive automorphism group $G$ we
can prove proposition \ref{Commutes_With_Spherical_Average} using
proposition \ref{defined_by_automorphisms} as follows: we have a
left action of $H$ on $\mathbb{C}^{B_{f}}\cong\oplus_{I\,\text{spherical}}C(G_{I}\backslash G)$.
The same space has a commuting right action. Let $G_{\sigma_{0}}$
be the stabilizer of $\sigma_{0}$. Then the element $\frac{1}{\left|G_{\sigma_{0}}\right|}1_{G_{\sigma_{0}}}\in H(G)$
defines a projection $\oplus_{I\,\text{spherical}}C(G_{I}\backslash G)$
into $\oplus_{I\,\text{spherical}}C(G_{I}\backslash G/G_{\sigma_{0}})$,
commuting with the left $H$ action. This projection is exactly $\rho_{\sigma_{0}}$.
\end{rem}
We continue the discussion of this section and address the representation
theory consequences in section \ref{sec:Representations-of-the-automorphism-group}.

\section{\label{sec:Finite-Quotients}Finite Quotients}

Assume we have a discrete subgroup $\Gamma\subset Aut(B)$ (recall
that unless otherwise stated automorphisms are color preserving).
We may then construct the quotient complex $X\cong B/\Gamma$. Since
$\Gamma$ is color preserving $X$ is a colored complex.

We will assume two further properties:

1. We assume that $X$ is \emph{finite} (i.e- the number of faces
of $X$ is finite). Such a subgroup $\Gamma$ is called a \emph{cocompact
lattice}. Since $X$ is finite $\mathbb{C}^{X_{f}}\cong\mathbb{C}[X_{f}]$.

2. We assume that $\Gamma$ is \emph{torsion free}.

The discreteness of $\Gamma$ means that its intersection with every
compact group is finite and since we assume $\Gamma$ is torsion free,
the intersection is actually trivial. Therefore for every spherical
face $\sigma$, $G_{\sigma}\cap\Gamma=\{id\}$ where $G_{\sigma}$
is the stabilizer of $\sigma$.  Therefore $X$ looks locally like
the building- each face of spherical color $I$ is contained in $q_{W_{I}}$
chambers. We remark that $X$ is not necessarily a simplicial complex,
as there might be double simplices (e.g two edges between the same
vertices). This will not interfere with the analysis but one should
remember this point.

Assume also we have a function $f\in\mathbb{C}^{X_{f}}\cong\mathbb{C}[X_{f}]$,
i.e a function that assigns a value to every face of a spherical color
on the quotient. Using the projection operator $\pi:B_{f}\rightarrow X_{f}$
we can define a function $\hat{f}\in\mathbb{C}^{B_{f}}$ by $\hat{f}(\sigma)=f(\pi(\sigma))$.
Now for some element $h\in H$ we can act on $\hat{f}$ and get a
new function $h\hat{f}\in\mathbb{C}^{B_{f}}$. Since the ADH algebra
commutes with automorphisms, $\hat{f}$ is $\Gamma$ invariant and
therefore so is $h\hat{f}$. So we can project and receive a function
$hf\in\mathbb{C}^{B_{f}}$ defined by $h\hat{f}(\sigma)=hf(\pi(\sigma))$.
Therefore $H$ acts on $\mathbb{C}^{X}$.

The algebra $H$ can also be defined directly- the fact that $X$
is colored means that we can define the coboundary and boundary operators
$\delta_{I}:\mathbb{C}^{X_{I}}\rightarrow\mathbb{C}^{X_{\phi}}$,
$\partial_{I}:\mathbb{C}^{X_{\phi}}\rightarrow\mathbb{C}^{X_{I}}$
for every spherical color $I$. Since $\Gamma$ is torsion free there
exist $q_{W_{I}}$ chambers containing every face of color $I$. Therefore
the two definitions are the same. Now we can define the rest of $H$
using the generating elements.

Since $X$ is finite the action of $H$ on it is a finite dimensional
representation. Moreover:
\begin{prop}
\label{prop:L^2X is unitary}Consider the inner product $\left\langle f,g\right\rangle =\sum_{\sigma\in X_{f}}\left\langle \bar{f}(\sigma),g(\sigma)\right\rangle $
on $\mathbb{C}^{X_{f}}$. The action of $H$ on $L_{2}(X_{f})$ is
a unitary representation, i.e. $\left\langle hf,g\right\rangle =\left\langle f,h^{\ast}g\right\rangle $
\end{prop}
\begin{proof}
2. It is enough to prove that $\left\langle \partial_{I}f,g\right\rangle =\left\langle f,\delta_{I}g\right\rangle $
for every $f,g\in L_{2}(X_{f})$ and $I$ spherical. It is immediate,
since both sides calculate $\sum_{\sigma\subset C,\,\sigma\in X_{I},\,C\in X_{\phi}}\bar{f}(\sigma)g(C)$
where $\sigma$ is of color $I$.
\end{proof}
\begin{rem}
It is sometimes useful to use the inner product $\left\langle f,g\right\rangle _{w}=\sum_{\sigma\in X_{f}}w(\sigma)\left\langle \bar{f}(\sigma),g(\sigma)\right\rangle $,
where $w(\sigma)=\#\{C\in X_{\phi}:\sigma\subset C\}$. In this definition
$\left\Vert \delta_{I_{2},I_{1}}f\right\Vert _{w}=\left\Vert f\right\Vert _{w}$
if $f\in\mathbb{C}^{X_{I_{1}}}$. However, the representation of $H$
on the $L_{2}$-norm is not unitary as we defined it. To fix it, one
should change the involution to $\delta_{I_{1}}^{*}=\partial_{I_{1}}/q_{W_{I_{1}}}$
(and extend to the other elements of $H$). We will not use this inner
product in our work, but see \cite{evra2015systolic} for example.
\end{rem}
\begin{rem}
The arguments of this section apply more generally then stated; as
a matter of fact we have an action of $H$ on $\mathbb{C}^{X_{f}}$
for every quotient $X\cong B/H$ for any subgroup $H\subset Aut(B)$.
\end{rem}

\part{Representation Theory}

In this part we will present some basic results about the representation
theory of $H$ and its connection to the building.

\section{\label{sec:Two-Simple-Representations}The Trivial Representation
and the Steinberg Representation}
\begin{defn}
Let $V$ be a vector space over $\mathbb{C}$. A \emph{representation}
$\left(\pi,V\right)$ of an algebra with unit $H$ is an homomorphism
of algebras $\pi:H\rightarrow Hom_{\mathbb{C}}(V,V)$ such that $\pi(1)=Id_{V}$.
\end{defn}
We will usually omit $\pi$ and write ``let $V$ be a representation
of $H$'' implicitly assuming $\pi$ is given as well. We will also
let elements of $h\in H$ act directly on the vector space $V$. We
start this part by presenting two simple and important representations:
\begin{prop}
\label{claim:Trivial-rep_definition}The sets of functions that depend
only on the color of each face form a representation space of $H$.
Its dimension is the number of finite colors of faces (i.e. $2^{n+1}-1$
in the vertex spherical case). This is called the \emph{trivial representation}.
\end{prop}
\begin{proof}
For $f\in\mathbb{C}^{B_{f}}$ depending only on the color of the face,
define $f(I)$ to be the value of $f$ on a face of color $I$. If
$d\in W_{I_{1}}\backslash W/W_{I_{2}}$ then 
\[
h_{d}f(\sigma)=\begin{cases}
0 & \sigma\text{ not of type }I_{2}\\
q_{d}f(I_{2}) & \sigma\text{ of type }I_{2}
\end{cases}
\]

It is therefore a representation, as required.
\end{proof}
\begin{lem}
\label{lem:1_phi is non zero}For every non zero representation $V$
of $H$ there exists $0\ne v\in V$ such that $1_{\phi}v=v$.
\end{lem}
\begin{proof}
We know that $1_{H}=\sum_{I}1_{I}$. Let $0\ne v\in V$. Since $1_{H}v=v$
there exists a finite color $I$ with $1_{I}v\ne0$. Since $1_{I}=q_{W_{I}}^{-1}\partial_{I}\delta_{I}$
also $\delta_{I}v\ne0$. Therefore $1_{\phi}\delta_{I}v=\delta_{I}v\ne0$.
\end{proof}
\begin{prop}
\label{claim:Steinberg Definition}There is only one irreducible representation
$V$ of $H$ on which $\partial_{\{s\}}$ acts by $0$ for every $s\in S$.
It is one dimensional. This representation is called the \emph{Steinberg
representation} or the \emph{special representation}.
\end{prop}
\begin{proof}
Assume that $V$ a representation on which $\partial_{\{s\}}$ acts
by $0$ for every $s\in S$. Let $0\ne v\in V$ be an element such
that $1_{\phi}v=v$. For $s\in S$ we have $h_{s}v=(\delta_{\{s\}}\partial_{\{s\}}-1_{\phi})v=(0-1_{\phi})v=-v$.
Inductively, $h_{w},w\in W$ acts on $v$ by $(-1)^{l(w)}$. For every
other $d\in W_{I_{1}}\backslash W/W_{I_{2}}$, $I_{1}\ne\phi$ or
$I_{2}\ne\phi$, $h_{d}$ acts by zero since every every restriction
operator is $0$. Therefore $V$ is one dimensional and uniquely determined
if it exists. To determine existence, notice that the relations above
define a representation, as required.
\end{proof}
\begin{rem}
Both the trivial and the Steinberg representations are one dimensional
representations of the Iwahori-Hecke algebra $H_{\phi}$. By the Iwahori-Hecke
relations in each such representation $h_{s}\rightarrow q_{s}$ or
$h_{s}\rightarrow-1$. In addition, if $m_{s,s'}$ is odd (and therefore
$q_{s}=q_{s'})$ we must have $h_{s},h_{s'}\rightarrow q_{s}$ or
$h_{s},h_{s'}\rightarrow-1$. On the other hand each such correspondence
defines a one dimensional representation. Therefore the number of
such representations is $2^{M}$, where $M$ is the number of equivalence
classes of $S$ generated by the relation $m_{s,s'}\,\mbox{is\,\ odd}\Longrightarrow s,\,s'\,\mbox{equivalent}$. 

For example in the bi-regular tree there are 4 such representations
and those are exactly the one dimensional $H_{\phi}$ representations
given in in \cite{kamber2016lpgraph}, proposition 11.6.
\end{rem}

\section{\label{sec:Unitary-Representations}Unitary Representations}

Recall we have an involution $*:H\rightarrow H$.
\begin{defn}
Let $V_{1},V_{2}$ be complex vector spaces. A map $\phi:V_{1}\times V_{2}\rightarrow\mathbb{C}$
is called \emph{sesquilinear} if it is additive and $\phi(av_{1},bv_{2})=\bar{a}b\phi(v_{1},v_{2})$.
A sesquilinear map $\phi:V\times V\rightarrow\mathbb{C}$ is called
\emph{positive definite Hermitian} (or \emph{an inner product}) if
we have $\phi(v_{1},v_{2})=\overline{\phi(v_{2},v_{1})}$ and $\phi(x,x)\in\mathbb{R}_{>0}$
for every $x\ne0$. The space $V$ is called a \emph{Hilbert space}
if it has an inner product and it complete with respect to the topology
it defines. A representation $V$ of $H$ is called \emph{unitary}
if $V$ is an Hilbert space and the representation satisfies $\left\langle hv_{1},v_{2}\right\rangle =\left\langle v_{1},h^{*}v_{2}\right\rangle $
for every $v_{1},v_{2}\in V$ and $h\in H$. 

A representation $V$ of $H$ is called \emph{normed} if $V$ is a
Banach space and every element $h\in H$ acts as a bounded linear
operator. 

If $V$ is finite dimensional or normed, \emph{the complex dual space}
$V^{*}$ is the vector space of continuous linear functionals on $V$.
We have an obvious map $\phi:V^{*}\times V\rightarrow\mathbb{C}$.
We define a $\mathbb{C}$ action of $V^{*}$ by $\left\langle \alpha v^{*},v\right\rangle =\left\langle v^{*},\bar{\alpha}v\right\rangle $
so the form $\phi:V^{*}\times V\rightarrow\mathbb{C}$ is sesquilinear.
\end{defn}
In practice, we deal in this work with representations that are either
finite dimensional (unitary or not), or the action of $H$ on $L_{p}(B_{f})$.
\begin{prop}
\label{claim:finite dimensional unitary decomposes}Every finite dimensional
unitary representation of $H$ decomposes into a direct sum of irreducible
representations.
\end{prop}
\begin{proof}
This result is standard in representation theory- assume $\{0\}\ne V'\subset V$
is a proper subrepresentation. Let $U=\{u\in V:\left\langle v,u\right\rangle =0,\,\forall v\in V'\}$.
Since it is an inner product we have $V=V'\oplus U$. Moreover if
$u\in U$, $h\in H$ then for every $v\in v'$, $\left\langle v,hu\right\rangle =\left\langle h^{*}v,u\right\rangle =0$.
Therefore $hu\in U$ and $U$ is also a subrepresentation. The claim
follows by a simple inductive argument.
\end{proof}
\begin{rem}
In general, a finite dimensional representation $V$ of $H$ does
not decompose into a sum irreducible representations, but into a sum
of indecomposable representations (i.e representations that cannot
be written as a sum of two proper subrepresentations). 
\end{rem}
\begin{prop}
1. Given a representation $(\pi,V)$ of $H$ we can define a representation
$(\pi^{*},V^{*})$, acting on the vector space $V^{*}$ by $(\pi^{*}(h)v^{*})(v)=v^{*}(\pi(h^{*})v)$. 

2. A unitary representation $(\pi,V)$ is isomorphic to $(\pi^{*},V^{*})$.
\end{prop}
\begin{proof}
Claim (1) is standard and left to the reader. Claim (2) follows from
the fact that the inner product gives a vector space isomorphism between
$V$ and $V^{*}$. It is easy to see that this isomorphism is also
an isomorphism of $H$ representations.
\end{proof}
\begin{defn}
The representation $(\pi^{*},V^{*})$ of the last proposition is called
\emph{the dual representation.}
\end{defn}
\begin{lem}
If $V$ is finite dimensional it is irreducible if and only if $V^{*}$
is irreducible.
\end{lem}
\begin{proof}
Since $\left(V^{*}\right)^{*}\cong V$ it is enough to prove that
if $V$ is irreducible so is $V^{*}$. Assume $V^{*}$ is not irreducible.
Then for some $0\ne v_{0}^{*}\in V^{*}$ the space $Hv_{0}^{*}$ is
a proper subspace of $V^{*}$- a linear subspace of smaller dimension.
Therefore there is some $0\ne v_{0}\in V$ such that $\left\langle hv_{0}^{*},v_{0}\right\rangle =0$
for every $h\in H$. But in this case $\left\langle v_{0}^{*},h^{*}v_{0}\right\rangle =0$
for every $h\in H$. Since $v_{0}^{*}\ne0$, $Hv_{0}$ is a proper
subrepresentation and $V$ is not irreducible.
\end{proof}

\section{\label{sec:Induction-and-restriction}Induction and Restriction of
Algebra Representations}

The ADH algebra $H$ is an algebra with a unit $1_{H}$. Let $H'\subset H$
be a subalgebra with unit $1_{H'}$. We do not assume in general that
$1_{H'}$ is equal to $1_{H}$.
\begin{defn}
Let $V$ be a representation of $H$. The subspace $V'=\{hv:v\in V,h\in H'\}=\{1_{H'}v:v\in V\}$
is a representation space of $H'$. We call this representation the
\emph{restriction }of $V$ from $H$ representation to $H'$ representation
and denote it $\mbox{res}_{H'}^{H}V$.

Let $V'$ be a representation of $H'$. The action of $H$ on the
space $H\otimes_{H'}V$ is called the \emph{induction} of $V$ from
$H'$ representation to $H$ representation. It is denoted $\mbox{ind}_{H'}^{H}V$.
\end{defn}
We will be mainly concerned with the case in which $H$ is the ADH
algebra of the building and $H'=H_{\phi}$. 
\begin{lem}
\label{prop:res-ind- is isomorphic}Assume that $1_{H'}H1_{H'}=H'$.
Let $V$ be a representation of $H'$. Then $\mbox{res}_{H'}^{H}\mbox{ind}_{H'}^{H}V$
is isomorphic to $V$. 
\end{lem}
\begin{proof}
Define a homomorphism $T:V\rightarrow\mbox{res}_{H'}^{H}\mbox{ind}_{H'}^{H}V$,
$v\rightarrow1_{H'}\otimes v$. $T$ is surely a $H'$-homomorphism.
Define the map $S:\mbox{res}_{H'}^{H}\mbox{ind}_{H'}^{H}V\rightarrow V$,
$1_{H'}h\otimes_{H'}v\rightarrow1_{H'}h1_{H'}v$ (using the fact that
$1_{H'}h1_{H'}\in H'$). We have $ST=id_{V}$. It remains to prove
that its image is $\mbox{res}_{H'}^{H}\mbox{ind}_{H'}^{H}V$. For
each $v\in V$ we have $1_{H'}v=v$. Therefore in $\mbox{ind}_{H'}^{H}V$,
for every $h\in H$, $1_{H'}(h\otimes v)=1_{H'}h1_{H'}\otimes v=1\otimes(1_{H'}h1_{H'}v)$.
Therefore every element of $\mbox{res}_{H'}^{H}\mbox{ind}_{H'}^{H}V$
can be written as $1_{H'}\otimes v$ for $v\in V$, as required.
\end{proof}
\begin{lem}
\label{lem:Ind H_phi Lemma}Let $V$ be a representation of $H_{\phi}$.
Each element of $\mbox{ind}_{H_{\phi}}^{H}V$ can be written as $\sum_{I}\partial_{I}\otimes_{H_{\phi}}v_{I},\,v_{I}\in V$,
$I\subset S$ is a (spherical) color and $\partial_{I}$ is the boundary
operator of definition \ref{def:Induction-and-Restriction defnition}.
\end{lem}
\begin{proof}
Let $h\otimes v\in\mbox{ind}_{H'}^{H}V$, $h\in H,\,v\in V$. Since
$h$ is a sum of basis elements, it is enough to assume $h=h_{d}$,$d\in W_{I_{1}}\backslash W/W_{I_{2}}$.
If $I_{2}\ne\phi$, $h_{d}\otimes_{H_{\phi}}v=h_{d}\otimes_{H_{\phi}}1_{\phi}v=h_{d}1_{\phi}\otimes_{H_{\phi}}v=0$.
Otherwise $I_{2}=\phi$ and $h_{d}=\partial_{I_{1}}h_{w}$ for some
$w\in W$. Then $h_{d}\otimes_{H_{\phi}}v=\partial_{I_{1}}h_{w}\otimes_{H_{\phi}}v=\partial_{I_{1}}\otimes_{H_{\phi}}h_{w}v$
and the claim proven.
\end{proof}
\begin{lem}
Let $V$ be a representation of $H$. Then $\mbox{ind}_{H_{\phi}}^{H}\mbox{res}_{H_{\phi}}^{H}V$
is isomorphic to $V$.
\end{lem}
\begin{proof}
Consider the linear transformation $T:V\rightarrow\mbox{ind}_{H_{\phi}}^{H}\mbox{res}_{H_{\phi}}^{H}V$,
$Tv=\sum_{I}(1/q_{W_{I}})\cdot\partial_{I}\otimes_{H_{\phi}}1_{\phi}\delta_{I}v$,
and the linear transformation $S:\mbox{ind}_{H_{\phi}}^{H}\mbox{res}_{H_{\phi}}^{H}V\rightarrow V$,
$S\left(\sum_{I}h\otimes_{H_{\phi}}1_{\phi}v\right)\rightarrow\sum h1_{\phi}v$.
By lemma \ref{lem:Ind H_phi Lemma} $T$ is onto. We have $1_{H}=\sum_{I}1_{I}=\sum(1/q_{W_{I}})\cdot\partial_{I}\delta_{I}$.
Therefore $ST=id_{V}$ and $T$ is an isomorphism of linear spaces.
Finally, for every $h\in H$ we have $ShTv=\sum_{I}(1/q_{W_{I}})\cdot h\partial_{I}1_{\phi}\delta_{I}v=hv$.
Therefore $hT=Th$ and $T$ is an $H$-isomorphism.
\end{proof}
\begin{lem}
Let $V$ be a representation of $H_{\phi}$. If $V$ is finite dimensional,
so is $\mbox{ind}_{H_{\phi}}^{H}V$. If $V$ is unitary, so is $\mbox{ind}_{H_{\phi}}^{H}V$.
\end{lem}
\begin{proof}
The finite dimensional case follows from lemma \ref{lem:Ind H_phi Lemma}.
As a matter of fact we have $\dim\mbox{ind}_{H_{\phi}}^{H}V\le\dim V\cdot\#\{I:\,\mbox{I spherical}\}$.
For the unitary case, define norm on $\mbox{ind}_{H_{\phi}}^{H}V$
by $\left\Vert \partial_{I}\otimes_{H_{\phi}}v\right\Vert =q_{W_{I}}^{-1}\left\Vert e_{I}v\right\Vert _{V}=q_{W_{I}}^{-1}\left\Vert \delta_{I}\partial_{I}v\right\Vert _{V}$.
\end{proof}
\begin{cor}
\label{prop:Induction-and-restriction-full-proposition}Induction
and restriction induce an equivalence of categories between $H_{\phi}$-representations
and $H$-representations. This equivalence preserves irreducible representations,
unitary representations and finite dimensional representations.
\end{cor}
\begin{rem}
This proposition is not true for other $H_{I}$, $I\ne\phi$. In particular,
the restriction of the Steinberg representation to $H_{I}$ representation
is 0 for every $I\ne\phi$.
\end{rem}
\begin{problem}
Is it true in general that induction sends finite dimensional $H_{I}$
representations to finite dimensional $H$ (or $H_{\phi}$) representations?

See also Problem \ref{prob:unitarity and admissibility of induction},
where a similar question is asked about admissibility of the induction
of a finite dimensional $H_{\phi}$ representation to $H(G)$ representation.
\end{problem}

\section{\label{sec:Matrix-Coefficients}Matrix Coefficients}

The following construction is well known in the representation theory
of groups: let $G$ be a group. The space $\mathbb{C}^{G}$ has a
$G\times G$ action given by $(g_{1},g_{2})f(x)=f(g_{1}^{-1}xg_{2})$.
Given any representation $V$ of $G$, $V^{*}$ the dual representation,
$0\ne v\in V,\,0\ne v^{*}\in V^{*}$, $c_{v^{*},v}(g)=\left\langle v^{*},gv\right\rangle $
is called a \emph{matrix coefficient} of the representation. Then
$v^{*}\otimes v\rightarrow c_{v^{*},v}$ is a homomorphism of representations
of $G\times G$, non zero if $V$ is irreducible. This allows us to
consider every irreducible representation as a subrepresentation of
$\mathbb{C}^{G}$. 

Similarly, let $H^{*}$ be the space of linear functionals $f:H\rightarrow\mathbb{C}$.
This space has a natural $H\otimes H$ action $(h_{1},h_{2})\psi(x)=\psi(h_{1}^{*}xh_{2})$
making it a representation of $H\times H$. We will focus on the right
$H$ action, i.e. the action given by $h\psi(x)=\psi(xh)$.
\begin{defn}
Let $V$ be a representation of $H$. Let $0\ne v^{*}$ be a non zero
vector. The functional $c_{v^{*},v}\in H^{*}$, $c_{v^{*},v}(h)=\left\langle v^{*},hv\right\rangle $
is called a \emph{matrix coefficient} of the representation.
\end{defn}
\begin{lem}
\label{prop:embedding into H*}Let $V$ be a representation of $H$.
Let $0\ne v^{*}\in V^{*}$ be a non zero vector. Then the correspondence
$V\rightarrow H^{*}$, $v\rightarrow c_{v^{*},v}$ is a non zero homomorphism
of representations of $H$.
\end{lem}
\begin{proof}
The fact that the correspondence is a homomorphism is given by definition:
\[
hc_{v^{*},v}(h')=c_{v^{*},v}(h'h)=\left\langle v^{*},h'hv\right\rangle =c_{v^{*},hv}(h')
\]
It is non zero since $v^{*}\ne0$, therefore there exists $v\in V$
with $\left\langle v^{*},v\right\rangle \ne0$ and then $c_{v^{*},v}(1_{H})=\left\langle v^{*},v\right\rangle \ne0$,
so $c_{v^{*},v}\ne0$.
\end{proof}
Let us shortly discuss representations of $H_{\phi}$. In this case
we have $H_{\phi}^{*}\cong\mathbb{C}^{W}$ as vector spaces (since
the $h_{w},w\in W$ are a basis for $H_{\phi}$). There is a bijection
between functions $f\in\mathbb{C}^{B_{\phi}}$ that are spherical
around $C_{0}$ and functions $f^{*}\in H_{\phi}^{*}$ given by $f^{*}(h_{w})=(h_{w}f)(C_{0})$.
By proposition \ref{Commutes_With_Spherical_Average} the set of $f\in\mathbb{C}^{B_{\phi}}$
spherical around $C_{0}$ is an $H_{\phi}$ representation and it
is easy to see directly that $f\rightarrow f^{*}$ is an isomorphism
of $H_{\phi}$ representations. Therefore, the chain of $H_{\phi}$
homomorphism:

\[
V\rightarrow H_{\phi}^{*}\leftrightarrow\mathbb{C}^{W}\leftrightarrow\left\{ f\in\mathbb{C}^{B_{\phi}}:\,f\,\mbox{spherical around }C_{0}\right\} 
\]

defines an embedding (i.e. a ``geometric realization'') of $V$
in $\mathbb{C}^{B_{\phi}}$. 

To extend the result to $H$ we will need the following definition:
\begin{defn}
A functional $\psi\in H^{*}$ is of color $\phi$ if it is is zero
on every base element $h_{d}$, $d\in W_{I_{1}}\backslash W/W_{I_{2}}$,
$I_{1}\ne\phi$. We denote the set of functionals of color $\phi$
by $H^{*}(\phi)$.
\end{defn}
The vector space $H^{*}(\phi)$ can be naturally identified with $\mathbb{C}^{\oplus_{I_{2}}W/W_{I_{2}}}$
(i.e. the set of functions on $\oplus_{I_{2}}W/W_{I_{2}}$)

Let $C_{0}$ be a chamber of $B$. proposition \ref{Commutes_With_Spherical_Average}
can be stated as follows: 
\begin{lem}
\label{H* and standard}1. The set of $f\in\mathbb{C}^{B_{f}}$ spherical
around $C_{0}$ is an $H$-representation.

2. The correspondence $f\rightarrow f^{*}$ of spherical functions
around $C_{0}$ into $H^{*}(\phi)$, $f^{*}(h_{d})=(h_{d}f)(C_{0})$
is an isomorphism of $H$-representations. The inverse of this isomorphism
is given by $f(\sigma)=q_{d(C_{0},\sigma)}^{-1}f^{*}(h_{d(C_{0},\sigma)})$.
\end{lem}
\begin{cor}
\label{cor:Geometric Realization}Each non zero representation $V$
of $H$ has a non zero homomorphism to a subrepresentation of the
action of $H$ on $\left\{ f\in\mathbb{C}^{B_{f}}:f\text{ spherical around \ensuremath{C_{0}}}\right\} $.
This homomorphism is given by a choice of vector $v^{*}\in V^{*}$
such that $1_{\phi}v^{*}=v^{*}$ and defining $f_{v^{*},v}(\sigma)=q_{d(C_{0},\sigma)}^{-1}\left\langle v^{*},h_{d(C_{0},\sigma)}v\right\rangle $.
If $V$ is irreducible it is an embedding.
\end{cor}
\begin{proof}
Choose $0\ne v^{*}\in V^{*}$ such that $1_{\phi}v^{*}=v^{*}$. Then
$c_{v^{*},v}(h)=\left\langle 1_{\phi}v^{*},hv\right\rangle =\left\langle v^{*},1_{\phi}hv\right\rangle $,
so $c_{v^{*},v}\in H^{*}(\phi)$. By lemma \ref{prop:embedding into H*}
it is a non zero homomorphism of $H$ representations. Thus, by lemma
\ref{H* and standard}, $v\rightarrow f_{v}$ is an isomorphism of
$H$ representations.
\end{proof}
\begin{defn}
We call each $H$-homomorphism of $V$ into $\mathbb{C}^{B_{f}}$
as in corollary \ref{cor:Geometric Realization} a \emph{geometric
realization} of the representation $V$.
\end{defn}

\section{\label{sec:-Tempered-Representations}$p$-Finite Representations
and $p$-Tempered Representations}
\begin{defn}
\label{def:Hphi tempered definition}We say that a finite dimensional
representation $V$ of $H_{\phi}$ is \emph{$p$-finite} if for every
$v\in V$, $u\in V^{\ast}$ we have $\sum_{w\in W}q_{w}^{(1-p)}\left|\left\langle u,h_{w}v\right\rangle \right|^{p}<\infty$.

We say that a finite dimensional representation $V$ of $H_{\phi}$
is \emph{$p$-tempered} if for every $v\in V$, $u\in V^{\ast}$ and
$\epsilon>0$ we have $\sum_{w\in W}q_{w}^{(1-p)}\left|\left\langle u,h_{w}v\right\rangle \right|^{p}(1-\epsilon)^{l(w)}<\infty$.
\end{defn}
We want to explain the geometry of this definition and extend it to
$H$-representations.

\begin{defn}
For distance $d=W_{I}\backslash\tilde{w}/W_{J}\in W_{I}\backslash W/W_{J}$,
$\tilde{w}\in W$ the shortest element in the double coset, we define
the \emph{distance length} $l(d)=l(\tilde{w})\in\mathbb{N}$.

For $f\in\mathbb{C}^{B_{f}}$ define $\left\Vert f\right\Vert _{p}=\left(\sum_{\sigma\in B_{f}}\left|f(\sigma)\right|^{p}\right)^{1/p}$,
$\left\Vert f\right\Vert _{\infty}=\sup{}_{\sigma\in B_{f}}\left|f(\sigma)\right|$.
Let $L_{p}(B_{j})=\{f\in\mathbb{C}^{B_{f}}:\left\Vert f\right\Vert _{p}<\infty\}$.
\end{defn}
Fix a chamber $C_{0}\in B_{f}$.
\begin{defn}
\label{def:F_delta}For every face $\sigma\in B_{f}$, we define the
\emph{distance length} $l(\sigma)=l^{C_{0}}(\sigma)=l(d(C_{0},\sigma))$.
For a function $f\in\mathbb{C}^{B_{f}}$ and $0<\delta<1$ define
$f_{\delta}=f_{\delta}^{C_{0}}\in\mathbb{C}^{B_{f}}$ as $f_{\delta}(\sigma)=(1-\delta)^{l(\sigma)}f(\sigma)$.

\label{def:tempered-definition}A function $f\in\mathbb{C}^{B_{f}}$
is called \emph{$p$-tempered}\textbf{ }and we write $f\in T_{p}(B_{f})=T_{p}^{C_{0}}(B_{f})$
if $f_{\delta}^{C_{0}}\in L_{p}(B_{f})$ for every $\delta>0$. Define
$T_{p}(B_{I})=T_{p}^{C_{0}}(B_{I})$ for $I$ spherical as $T_{p}(B_{I})=T_{p}(B)\cap\mathbb{C}^{B_{I}}\subset\mathbb{C}^{B_{f}}$.
\end{defn}
\begin{lem}
The definition does not depend on the choice of $C_{0}$, that is
$T_{p}^{C_{0}}(B_{f})=T_{p}^{C_{0}^{\prime}}(B_{f})$ for every chamber
$C_{0}^{\prime}$.
\end{lem}
\begin{proof}
Assume we replace $C_{0}$ by $C_{0}^{\prime}$ with $l(d(C_{0},C_{0}^{\prime}))=L$.
Then for every face $\sigma$ we have $\left|l^{C_{0}}(\sigma)-l^{C_{0}^{\prime}}(\sigma)\right|\le L$.
Therefore $f_{\delta}^{C_{0}}(\sigma)(1-\delta)^{L}\le f_{\delta}^{C_{0}^{\prime}}(\sigma)\le f_{\delta}^{C_{0}}(\sigma)(1-\delta)^{-L}$
for every $\delta>0$. The claim follows.
\end{proof}
The following proposition explains definition \ref{def:Hphi tempered definition}.
\begin{prop}
A finite dimensional representation $V$ of $H_{\phi}$ is \emph{$p$-finite}
(resp.\emph{ $p$-tempered)} if and only if every function $f\in\mathbb{C}^{B_{\phi}}$,
in every geometric realization of $V$ , is in $L_{p}(B_{\phi})$.
(resp. $T_{p}(B_{\phi})$). 
\end{prop}
\begin{proof}
There are $q_{w}$ chambers of distance $w$ from $C_{0}$. Therefore
every function $f$ in a geometric realization of $V$ is $p$-finite
if and only if for every $v\in V$, $u\in V^{\ast}$ 
\[
\sum_{w\in W}q_{w}\left(\left|\left\langle u,h_{w}v\right\rangle \right|/q_{w}\right)^{p}=\sum_{w\in W}q_{w}^{(1-p)}\left|\left\langle u,h_{w}v\right\rangle \right|^{p}<\infty
\]

The $p$-tempered case is very similar.
\end{proof}
We can now define $p$-finite and $p$-tempered $H$-representations:
\begin{defn}
A finite dimensional representation $V$ of $H$ is called \emph{$p$-finite}
(resp. \emph{$p$-tempere}d) if every function $f\in\mathbb{C}^{B_{f}}$,
in every geometric realization of $V$ , is in $L_{p}(B_{f})$ (resp.
$T_{p}(B_{f})$).
\end{defn}
The following lemma is immediate and left to the reader. It will allow
us to work with $H_{\phi}$ instead of $H$.
\begin{lem}
\label{lem:Temperdness_depends_only _on_phi}A function $f\in\mathbb{C}^{B_{f}}$
satisfies $f\in L_{p}(B_{f})$ (respectively $T_{p}(B_{f})$) if and
only if for every spherical color $I$ (including $\phi$) $\delta_{I}f\in T_{p}(B_{\phi})$
(respectively $\delta_{I}f\in L_{p}(B_{\phi})$). Therefore the equivalence
of categories between $H_{\phi}$-representations and $H$-representations
also respects $p$-finiteness and $p$-temperedness.
\end{lem}
The following claim relates our definition of temperedness to the
definition stated in the introduction. Recall that a building is thick
if $q_{s}>1$ for all $s\in S$.
\begin{lem}
\label{claim:Equivalence of temperedness definitions}If $f\in L_{p+\epsilon}(B_{f})$
for every $\epsilon>0$ then $f\in T_{p}(B_{f})$. 

If the building is thick and the function $f\in T_{p}(B_{f})$ is
spherical around $C_{0}$, then $f\in L_{p+\epsilon}(B_{f})$ for
every $\epsilon>0$.
\end{lem}
\begin{proof}
Using the last lemma it is enough to prove this claim for $f\in\mathbb{C}^{B_{\phi}}$
(since $f\in L_{p+\epsilon}(B_{f})$ if and only if every spherical
color $I$, $\delta_{I}f\in L_{p+\epsilon}(B_{\phi})$). 

Assume $f\in L_{p+\epsilon}(B_{\phi})$ for every $\epsilon>0$. Notice
that the number of chambers of distance $l$ from $C_{0}$ is bounded
by $r_{2}^{l}$, for some $r_{2}>0$, since for $\alpha_{2}=\max_{s\in W}q_{s}$,
$q_{w}\le\alpha_{2}^{l(w)}$ and $\#\{w:l(w)=m\}\le\left|S\right|^{m}$.
Therefore $\sum_{C}r_{2}^{-l(C)}$ converges and if we define $g(C)=\max\left\{ \left|f(C)\right|,r_{2}^{-p^{-1}l(C)}\right\} $
then $g\in L_{p+\epsilon}(B_{\phi})$ for every $\epsilon>0$. For
every $\delta>0$ there exists some $\epsilon>0$ such that $(1-\delta)<r_{2}^{-\epsilon p^{-2}}$.
Then 
\[
\left|f_{\delta}^{C_{0}}(C)\right|{}^{p}\le g_{\delta}^{C_{0}}(C)^{p}=g(C)^{p}(1-\delta)^{l(C)p}<g(C)^{p}r_{2}^{-\epsilon p^{-1}l(C)}\le g(C)^{p}g(C)^{\epsilon}\le g(C)^{p+\epsilon}
\]

Therefore $f_{\delta}^{C_{0}}\in L_{p}(B_{\phi})$ and $f\in T_{p}(B_{\phi})$.

For the other direction, assume the building is thick, and the function
$f\in T_{p}(B_{f})$ is spherical around $C_{0}$. We may therefore
define $f_{W}\in\mathbb{R}_{\ge0}^{W}$ by $f_{W}(w)=\left|(h_{w}f)(C_{0})\right|$.

Since $f\in T_{p}(B_{\phi})$ and $f$ is spherical around $C_{0}$,
the series $\sum_{w\in W}q_{w}\left(f_{W}(w)/q_{w}\right)^{p}(1-\delta)^{p\cdot l(w)}$
converges for every $0<\delta<1$. Since $\alpha_{1}=\min_{s\in S}q_{s}>1$,
$q_{w}>\alpha_{1}^{l(w)}$ and since the series converges, $\left(f_{W}(w)/q_{w}\right)^{p}\le\alpha_{1}^{-l(w)}$
for almost all $w\in W$. For every $\epsilon>0$ choose $1>\delta>0$
such that $(1-\delta)^{p}\ge\alpha_{1}^{-p\epsilon}$. Then $(1-\delta)^{p\cdot l(w)\epsilon}\ge\alpha_{1}^{-pl(w)\epsilon}\ge\left(f_{W}(w)/q_{w}\right)^{\epsilon}$.

Then 
\[
\sum_{w\in W}q_{w}\left(f_{W}(w)/q_{w}\right)^{p+\epsilon}\le\sum_{w\in W}q_{w}\left(f_{W}(w)/q_{w}\right)^{p}(1-\delta)^{p\cdot l(w)}
\]

And therefore $f\in L_{p+\epsilon}$.
\end{proof}
If the building is thin or the function is not spherical then the
lemma has simple counter examples.
\begin{lem}
\label{lem:-H acts as bounded operators on Tp}$T_{p}(B)$ is a representation
of $H$. Moreover for every $h\in H$ there exists a number $M(h)\in R_{\ge0}$
such that for every $f\in T_{p}(B)$, $0.5>\delta>0$, $\left\Vert (hf)_{\delta}\right\Vert _{p}\le M(h)\left\Vert f_{\delta}\right\Vert _{p}$.
\end{lem}
\begin{proof}
The second part clearly imply the first. It is enough to prove it
for $h=\delta_{I}$ and $h=\partial_{I}$. Let $L$ be the length
of the longest element of $W_{I}$. Therefore for every $\sigma\subset C$
of color $I$ we have $l(\sigma)\le l(C)\le l(\sigma)+L$. Then we
have:

\begin{eqnarray*}
\left\Vert \left(\delta_{I}f\right)_{\delta}\right\Vert _{p}^{p} & \le & q_{W_{I}}\left\Vert f_{\delta}\right\Vert _{p}^{p}\\
\left\Vert \left(\partial_{I}f\right)_{\delta}\right\Vert _{p}^{p} & \le & (1-\delta)^{-pL}q_{W_{I}}^{p-1}\left\Vert f_{\delta}\right\Vert _{p}^{p}\le2^{pL}q_{W_{I}}^{p-1}\left\Vert f_{\delta}\right\Vert _{p}^{p}
\end{eqnarray*}
\end{proof}
\begin{lem}
Let \textup{$V$} be a finite dimensional representation of $H$.
Then $V$ is $p$-finite (resp.\emph{ $p$}-tempered) if and only
if $V^{*}$ is $p$-finite (resp. $p$-tempered).
\end{lem}
\begin{proof}
This is an immediate corollary of the definition.
\end{proof}
\begin{lem}
\label{lem:Irreducible temperedness}An irreducible finite dimensional
representation is $p$-finite (resp. $p$-tempered) if a single function
$f\ne0$ in some geometric realization is in $L(B)$ (resp. $T_{p}(B)$).
\end{lem}
\begin{proof}
Fix $0\ne v_{0}^{*}\in V^{*},v_{0}\in V$, $1_{\phi}v_{0}^{*}=v_{0}^{*}$.
Assume that the geometric realization $f_{v_{0}^{*},v_{0}}\in\mathbb{C}^{B}$
of $v_{0}$ corresponding to $v_{0}^{*}$ is $p$-tempered ($p$-finite
respectively). Consider changing $v_{0}$ to $v_{0}^{\prime}$. Since
$V$ is irreducible there exists $h\in H$ with $hv_{0}=v_{0}^{\prime}$.
Therefore $f_{v_{0}^{*},v_{0}^{\prime}}=f_{v_{0}^{*},hv_{0}}=hf_{v_{0}^{*},v_{0}}$.
Since $T_{p}(B)$ ($L_{p}(B)$ respectively) is a representation of
$H$, $f_{v_{0}^{*},v_{0}^{\prime}}$ is also $p$-tempered ($p$-finite
respectively). To prove that it does not depend on $v_{0}^{*}$ switch
the roles of $V,V^{*}$ and use the fact that $V^{*}$ is also irreducible.
\end{proof}

\section{\label{sec:Expander-Family-of}Expander Family of Complexes}

Let $X=B/\Gamma$ be a finite quotient of $B$. We wish to understand
the action of $H$ on $\mathbb{C}[X_{f}]=L_{2}(X_{f})$. Recall that
this representation is unitary and finite dimensional (see proposition
\ref{prop:L^2X is unitary}) and therefore decomposes into a finite
direct sum of irreducible representations.

Recall that $\rho_{C_{0}}$ is the spherical average around the chamber
$C_{0}\in B_{\phi}$ from definition \ref{def: spherical operators}.
\begin{prop}
\label{H action Lemma}Let $f\in\mathbb{C}[X_{f}]$, $C_{0}\in X_{0}$.
Let $\tilde{f}\in\mathbb{C}^{B_{f}}$ be the lift of $f$ from $X$
to $B$. Let $\tilde{C_{0}}$ be a chamber covering $C_{0}$. 
\begin{enumerate}
\item The correspondence $h\rightarrow(hf)(C_{0})$ is a matrix coefficient
of the $H$-representation $\mathbb{C}[X_{f}]$. 
\item A geometric realization around $\tilde{C}_{0}$ is given by $\rho_{\tilde{C}_{0}}\tilde{f}$. 
\item For every irreducible representation $V$ there exists $C_{0}\in X_{\phi}$
such that the matrix coefficient defined on $V$ is non zero.
\end{enumerate}
\end{prop}
\begin{proof}
(1) follows by definition, since $\mathbb{C}[X_{f}]$ is finite dimensional
and $f\rightarrow f(C_{0})$ is a functional on $f\in\mathbb{C}[X_{f}]$.
For (2) , notice that $h\rightarrow(hf)(C_{0})$ is a functional in
$H_{\phi}^{*}$ and therefore has a geometric realization around $C_{0}$
which equals exactly $\rho_{\tilde{C}_{0}}\tilde{f}$. For (3), note
that in every nonzero subrepresentation $V$ there exists a non zero
$f\in\mathbb{C}[X_{\phi}]\cap V$ by lemma \ref{lem:1_phi is non zero}.
\end{proof}
\begin{prop}
The trivial representation appears exactly once in $L_{2}(X)$. It
is the subrepresentation of $H$ which contains the sets of functions
that depend only on the color of each face.
\end{prop}
\begin{proof}
The fact that the constant functions on every color span the trivial
representation is immediate. To prove it is the only such representation,
choose some function $f\in\mathbb{C}[X_{\phi}]$ which spans a representation
isomorphic to the trivial representation. Let $C_{0}\in X_{\phi}$
be the chamber on which $f$ gets its maximal values. Since $h_{s},s\in S$
acts by $q_{s}$, all the chambers adjacent to $C_{0}$ must have
the same value. Therefore $f$ is constant on the chambers and the
representation spanned by $f$ is the representation which contains
the constant functions on every color of face.
\end{proof}
\begin{defn}
The \emph{non trivial representation} of $H$ on $\mathbb{C}[X_{f}]$
is the action of $H$ on $L_{2}^{00}(X_{f})=\left\{ f:X\rightarrow\mathbb{C}:\forall I\,\sum_{\sigma\in X_{I}}f(\sigma)=0\right\} $.
This is the space perpendicular to the trivial representation.
\end{defn}
\begin{prop}
The number of times the Steinberg representation appears in $L_{2}^{00}(X_{f})$
is the dimension of the subspace $\left\{ f\in\mathbb{C}[X_{\phi}]:\,\partial_{\{s\}}f=0\text{ for all }s\in S\right\} $. 
\end{prop}
\begin{proof}
The claim follows immediately from proposition \ref{claim:Steinberg Definition}.
\end{proof}
\begin{defn}
\label{def:Expander_definition}The complex $X$ is an \emph{$L_{p}$-expander}
if the representation of $H$ on $L_{2}^{00}(X)$ is $p$-tempered.

The complex $X$ is a \emph{Ramanujan complex} if it is an $L_{2}$-expander. 
\end{defn}
\begin{cor}
\label{claim:Chambers are enough for Lp-xpander}The following are
equivalent:

1. $X$ is an $L_{p}$-expander.

2. For every $f\in L_{2}^{00}(X)$ and $C_{0}\in B$, $\rho_{C_{0}}(\tilde{f})\in L_{p+\epsilon}(B)$
for every $\epsilon>0$.

3. The action of $H_{\phi}$ on $L_{2}^{0}(X_{\phi})=\left\{ f:X_{\phi}\rightarrow\mathbb{C}:\forall I\,\sum_{C\in X_{\phi}}f(C)=0\right\} $
is $p$-tempered.
\end{cor}
\begin{proof}
Follows from proposition \ref{H action Lemma} and lemma \ref{lem:Temperdness_depends_only _on_phi}.
\end{proof}

\section{\label{sec:Representations-of-the-automorphism-group}Representations
of the Automorphism Group}

In this section we continue the discussion of section \ref{sec:Building-Automorphisms-II},
this time looking at the representation theory involved. As in section
\ref{sec:Building-Automorphisms-II} let $G'$ be a general locally
profinite group and $K$ is a compact open subgroup with Haar measure
$1$. Let $H_{K}=\mathbb{C}_{c}[K\backslash G'/K]$ be the corresponding
Hecke algebra with respect to convolution. 

We want to understand the connection between the representation theory
of $G'$ and $H_{K}$. We will denote a representation of $G'$ by
$U$ and a representation of $H_{K}$ by $V$. We base the general
discussion mainly on \cite{casselman1974introduction}, part 2.

The case we will be interested in is when $G'=G$ a Weyl transitive
color preserving complete automorphism group, $K=G_{\phi}$ is a chamber
stabilizer and $H_{K}=H_{\phi}$. We mainly follow \cite{borel1976admissible}
for the results specific to this case.
\begin{defn}
A representation $U$ of $G'$ is called:
\begin{itemize}
\item \emph{Smooth} if every $v\in U$ is fixed by an open subgroup of $G'$.
\item \emph{Admissible} if for every compact open subgroup $K'\subset G$
the subspace $U^{K'}$ of vectors fixed by $K'$ is finite dimensional.
\item \emph{Unitary} if there exists a $G'$ invariant inner product on
$U$ (so that the completion of $U$ is a Hilbert space).
\end{itemize}
\end{defn}
\begin{prop}
\label{prop: Restriction Proposition}Let $U$ be a smooth representation
of $G'$. Let $U^{K}$ be the vectors of $U$ fixed by $K$. Then
$U^{K}$ is a representation space of $H_{K}$. If $U$ is irreducible
as representation of $G$ then $U^{K}$ is an irreducible representation
of $H_{K}$. If $U$ is unitary so is $U^{K}$.
\end{prop}
\begin{proof}
Define $e_{K}:U\rightarrow U{}^{K}$ by $e_{K}v=\int_{K}kv\,dk$.
This integral is actually a finite sum since by smoothness there exists
a finite index compact open subgroup $K'\subset K$ with $k'v=v$
for every $k'\in K'$. Then the integral becomes $1/[K:K']\cdot\sum_{k\in K/K'}kv$.
It is easy to see that $e_{K}v\in U^{K}$, $e_{K}|_{K}=id$ and $e_{K}^{2}=e_{K}$
(we use here the fact that $\left|K\right|=1$).

Define an action of the algebra $H(G')$ on $U$ by $f\cdot v=\int_{G}f(g)gv\,dg$.
It is standard to verify that it defines an algebra representation
of $H(G')$. In the action above, the element $1_{KgK}\in H_{K}(G')\subset H(G')$
acts on $U$ by $e_{K}ge_{K}$. Since $e_{K}U=U^{K}$ we see that
$U^{K}$ is a representation of $H_{K}(G')$ .

Suppose $U$ is irreducible. Let $0\ne u\in\subset U^{K}$. Since
$U$ is irreducible, for every $v'\in U^{K}$ there exist $g_{1},...,g_{m}\in G$
and $\alpha_{1},...,\alpha_{m}\in\mathbb{C}$ such that $v'=\sum\alpha_{i}g_{i}u$.
Since $u,v'\in U^{K}$ we have $v'=\sum\alpha_{i}e_{K}g_{i}e_{K}u\in H_{K}u$.
Therefore $H_{K}u=U^{K}$ and $U^{K}$ is an irreducible $H_{K}$
representation. See \textbf{\uline{\mbox{\cite{casselman1974introduction}}}}
and in particular proposition 2.2.4(a) there for more details.

Finally if $U$ is unitary one can use the same inner product on $U^{K}$
which gives it a unitary $H_{K}$ structure. This verification is
immediate.
\end{proof}
To discuss temperedness of $G'$ representations, we define:
\begin{defn}
Let $U$ be an admissible representation of $G'$. Its \emph{dual
representation} $\hat{U}$ is the action of $G$ on the algebraic
dual $\hat{U}$ of $U$ by $(gv^{*})(v)=v^{*}(gv)$. 

The \emph{contragredient} \emph{representation} $\tilde{U}$ of $U$
are the smooth vectors of $\hat{U}$, i.e vectors $\tilde{v}\in\hat{U}$
that have a compact open subgroup $K'$ with $k\tilde{v}=\tilde{v}$
for any $k\in K'$.
\end{defn}
\begin{lem}
Let $U$ be an admissible representation of $G'$. Then $\hat{U}^{K}=\tilde{U}^{K}=\hat{\left(U^{K}\right)}$
and $\tilde{\tilde{U}}\cong U$.
\end{lem}
\begin{proof}
\textbf{\cite{casselman1974introduction}} 2.1.10.
\end{proof}
\begin{defn}
Let $U$ be an admissible representation of $G'$, $\tilde{U}$ the
contragredient representation and $v\in U$, $\tilde{v}\in\tilde{U}$.
The function $c_{\tilde{v},v}:G'\rightarrow\mathbb{C}$, $c_{\tilde{v},v}(g)=\left\langle \tilde{v},gv\right\rangle $
is called a \emph{matrix coefficient} of the representation $U$.
\end{defn}
Since we have a Haar measure on $G'$ the space $L_{p}(G')$ is well
defined. We also define temperedness similar to definition \ref{def:tempered-definition}:
\begin{defn}
\label{def:temperedness-definition-group}Assume that $G'$ is generated
by an open set $A_{0}$ with compact closure. Then define for $g\in G'$:
$l_{A_{0}}(g)=\min\{n:g\in A_{0}^{n}\}$.

Define the space 
\[
T_{p}(G')=T_{p}^{A_{0}}(G')=\left\{ f:G'\rightarrow\mathbb{C}\,\text{measurable}:\int\left|f(g)\right|^{p}\left(1-\epsilon\right)^{l_{A_{0}}(g)}dg<\infty\,\text{for every \ensuremath{\epsilon>0}}\right\} 
\]
\end{defn}
\begin{lem}
1.The set $T_{p}^{A_{0}}(G')$ does not depend on the generating set
$A_{0}$.

2. We have $\cap_{\epsilon>0}L_{p+\epsilon}(G')\subset T_{p}(G')$.
\end{lem}
\begin{proof}
Let $A_{0}^{\prime}$ be another open generating set with compact
closure. Since it is covered by $\cup A_{0}^{n}$ it is covered by
a finite subset of which and therefore there exists $C>0$ such that
$l_{A_{0}}(g')\le C$ for every $g'\in A_{0}^{\prime}$. Therefore
$l_{A_{0}^{\prime}}(g)\le C\cdot l_{A_{0}}(g)$ for every $g\in G'$.
Then for every $\epsilon>0$, $(1-\epsilon)^{l_{A_{0}^{\prime}}(g)}\ge(1-\epsilon)^{Cl_{A_{0}}(g)}$
for every $g\in G'$. Therefore $T_{p}^{A_{0}}(G')\subset T_{p}^{A_{0}^{\prime}}(G')$.
By symmetry $T_{p}^{A_{0}^{\prime}}(G')\subset T_{p}^{A_{0}}(G')$
and we have equality.

For (2), we claim that $\mu(A_{0}^{n})$ grows at most exponentially-
there exists $r>0$, such that $\mu(A_{0}^{n})\le r^{n-1}\mu(A_{0})$.
By compactness, cover $A_{0}^{2}$ by a finite number of translations
$y_{i}A_{0}$, $y_{i}\in G'$, $i=1,...,R$. Then 
\[
A_{0}^{n}\subset\cup_{i_{j}\in\{1,...,N\}}y_{i_{1}}\cdot...\cdot y_{i_{n-1}}A_{0}
\]

and $\mu(A_{0}^{n})\le R^{n-1}\mu(A_{0})$. The rest of the proof
is as in lemma \ref{claim:Equivalence of temperedness definitions}.
\end{proof}
\begin{defn}
\label{def:tempered as group representation}An admissible representation
$U$ of $G'$ is \emph{$p$-finite (resp. $p$-tempered)} if for every
$v\in U$, $\tilde{v}\in\tilde{U}$ we have $c_{\tilde{v},v}\in L_{p}(G')$
(resp. $c_{\tilde{v},v}\in T_{p}(G')$).
\end{defn}
\begin{lem}
An irreducible representation $U$ is $p$-finite (resp.$p$-tempered)
if for some $0\ne v\in U$, $\tilde{v}\in\tilde{U}$ we have $c_{\tilde{v},v}\in L_{p}(G')$
(resp. $c_{\tilde{v},v}\in T_{p}(G')$).
\end{lem}
\begin{proof}
As in lemma \ref{lem:Irreducible temperedness}.
\end{proof}
\begin{prop}
\label{prop:Tempereddness and G}Assume the building $B$ is thick.
Let $G$ be a complete Weyl transitive automorphism group of $B$.
Let $U$ be an irreducible representation of $G$, with $U^{G_{\phi}}\ne\{0\}$.
Then $U$ is $p$-finite (resp. $p$-tempered) if and only if $U^{G_{\phi}}$
is $p$-finite (resp. $p$-tempered) as a representation of $H_{\phi}$.
\end{prop}
\begin{proof}
By proposition \ref{prop: Restriction Proposition}, $U^{G_{\phi}}$
is irreducible. Using lemma \ref{lem:Irreducible temperedness}, $U^{G_{\phi}}$
is $p$ tempered as $H_{\phi}$ representation if and only if for
some $0\ne v\in U{}^{G_{\phi}}$, $0\ne\tilde{v}\in\hat{\left(U^{K}\right)}$
we have 

\begin{equation}
\sum_{w\in W}q_{w}^{1-p}\left|\left\langle \tilde{v},h_{w}v\right\rangle \right|^{p}(1-\epsilon)^{pl(w)}=\sum_{w\in W}q_{w}\left|\left\langle \tilde{v},h_{w}/q_{w}v\right\rangle \right|^{p}(1-\epsilon)^{pl(w)}<\infty\label{eq:tempered}
\end{equation}
 for every $\epsilon>0$.

Using the last lemma, the $G$-representation $U$ is $p$-tempered
if and only if $c_{\tilde{v},v}\in T_{p}(G')$ for every $\epsilon>0$.
Choose an a compact generating set $A_{0}=G_{\phi}\cup\left(\cup_{s\in S}G_{\phi}g_{s}G_{\phi}\right)$,
where $g_{s}\in G$ is an element sending $C_{0}\in B_{\phi}$ to
$C'\in B_{\phi}$ with $d(C_{0},C')=s$. One can easily see that unless
$g\in G_{\phi}$, $l_{A_{0}}(g)=l(d(C_{0},gC_{0}))$. Therefore for
every $\epsilon>0$:

\begin{align}
\int_{G}\left|c_{\tilde{v},v}(g)\right|^{p}(1-\epsilon)^{l_{A_{0}}(g)}dg & =\int_{G}\left|\left\langle \tilde{v},gv\right\rangle \right|^{p}(1-\epsilon)^{l_{A_{0}}(g)}dg=\sum_{G_{\phi}gG_{\phi}\in G_{\phi}\backslash G/G_{\phi}}\mu(G_{\phi}gG_{\phi})\left|\left\langle \tilde{v},gv\right\rangle \right|^{p}(1-\epsilon)^{l_{A_{0}}(g)}=\label{eq:tempered2}\\
 & =\sum_{Id\ne w\in W}q_{w}\left|\left\langle \tilde{v},h_{w}/q_{w}v\right\rangle \right|^{p}(1-\epsilon)^{l(w)}+\left|\left\langle \tilde{v},v\right\rangle \right|^{p}(1-\epsilon)\nonumber 
\end{align}

(we used the facts that if $d(C_{0},gC_{0})=w$, then $\mu(G_{\phi}gG_{\phi})=q_{w}$
and $\left\langle \tilde{v},gv\right\rangle =\left\langle \tilde{v},e_{K}ge_{K}v\right\rangle =\left\langle \tilde{v},h_{w}/q_{w}v\right\rangle $,
since $v,\tilde{v}$ are $K$-fixed). Since the two conditions \ref{eq:tempered},\ref{eq:tempered2}
are equivalent, we are done.
\end{proof}
\begin{lem}
Every unitary and admissible representation $U$ of $G'$ decomposes
into a countable direct sum of irreducible representations $U=\oplus U_{i}$.
For each open compact subgroup $K\subset G'$ only a finite number
of the $U_{i}$ have $U_{i}^{K}\ne0$.
\end{lem}
\begin{proof}
It is similar to \ref{claim:finite dimensional unitary decomposes}.
See \cite{casselman1974introduction} proposition 2.1.14.
\end{proof}
Let us now discuss how to induce $H_{K}$ representations to $G'$
representations.
\begin{prop}
\label{Induction from HK to G- basic}Let $V$ be a representation
of $H_{K}$.
\begin{enumerate}
\item The space $\mathbb{C}[G'/K]\otimes_{H_{K}}V$ is a representation
of $G'$ with the natural left action of $G'$\textbf{ }on \textbf{$\mathbb{C}[G'/K]$.}
\item The space $\mathbb{C}[G'/K]\otimes_{H_{K}}V$ is generated as a $G'$
module by its $K$ fixed vectors and $\left(\mathbb{C}[G'/K]\otimes_{H_{K}}V\right)^{K}$
is naturally isomorphic to $V$ as a $H_{K}$ representation.
\item The functors $\mathbb{C}[G'/K]\otimes_{H_{K}}$ and \textup{$(\cdot)^{K}$}
provide a natural bijection between equivalence classes of irreducible
$G$-representations with $K$ fixed vectors and equivalence classes
of irreducible $H_{K}$-representations.
\end{enumerate}
\end{prop}
\begin{proof}
For (1),(2) see \cite{borel1976admissible}. For (3) see \cite{bushnell1993admissible}
4.2.3.
\end{proof}
Proposition \ref{Induction from HK to G- basic} does not give a full
description of the connection between $G'$-representations and $H_{K}$-representations.
Two important ingredient that are missing are the admissibility of
$\mathbb{C}[G/K]\otimes_{H_{K}}V$ if $V$ is finite dimensional and
the unitarity of $\mathbb{C}[G/K]\otimes_{H_{K}}V$ if $V$ is unitary.
In the algebraic-group case we have an answer to those questions,
due to Borel (\cite{borel1976admissible}) and Barbasch and Moy (\cite{barbasch1993reduction}):
\begin{thm}
\label{thm:Connection between Hphi and G rep} Let $G$ be the rational
points of a connected semisimple algebraic group $\mathbb{G}$ over
a non-Archimedean local field $k$.

Let $B$ be the locally finite regular affine building corresponding
to $G$. Let $G_{\phi}$ be a chamber stabilizer (i.e an Iwahori subgroup).
Then: 
\begin{enumerate}
\item If $V$ is a finite dimensional representation of $H_{\phi}$, then
$\mathbb{C}[G/G_{\phi}]\otimes_{H_{\phi}}V$ is admissible.
\item The functors $\mathbb{C}[G/G_{\phi}]\otimes_{H_{\phi}}$ and $(\cdot)^{G_{\phi}}$
are exacts functors between admissible $G$ representations and finite
dimensional $H_{\phi}$ representations.
\item Every admissible representation $U$ of $G$ is a direct sum of the
representation $U_{1}$ generated by $U_{1}^{G_{\phi}}$ and a representation
$U_{2}$ with $U_{2}^{G_{\phi}}=0$.
\item If $V$ is finite dimensional and unitary then $\mathbb{C}[G/G_{\phi}]\otimes_{H_{\phi}}V$
is unitary.
\end{enumerate}
\end{thm}
\begin{proof}
(1),(2) and (3) are the main results of \cite{borel1976admissible}.
The unitarity question is harder and was open for a long time. It
was solved using the classification of unitary finite dimensional
$H_{\phi}$-representations. See \cite{barbasch1993reduction}.
\end{proof}
\begin{rem}
The representations considered in the theorem above are the well known
unramified principal series representations of $G$- i.e. the representations
induced from an unramified character of a maximal torus. See \cite{borel1976admissible}.
\end{rem}
\begin{problem}
\label{prob:unitarity and admissibility of induction}Can a similar
answer for unitarity and admissibility be given for arbitrary locally
finite regular buildings, perhaps assuming some transitivity property
of the automorphism group? In particular, can a similar theorem be
stated for right-angles buildings?
\end{problem}
Since theorem \ref{thm:Connection between Hphi and G rep} is rather
deep and does not always apply, we define:
\begin{defn}
A representation $V$ of $H$ is \emph{$G$-unitary} if it is a restriction
to $H$ of a unitary representation of $G$.
\end{defn}
From proposition \ref{prop: Restriction Proposition} if a representation
is $G$-unitary then it is unitary. The converse is true in the algebraic
group case by theorem \ref{thm:Connection between Hphi and G rep}.

\section{\label{sec:Oh's-Theorem}Oh's Theorem}

The results of \cite{oh2002uniform}, specifically theorem 7.4, state:
\begin{thm}
\label{thm:Oh's Theorem}Let $k$ be a non-Archimedean local field
with $\mathrm{char}k\ne2$. Let $G$ be the group of $k$-rational
points of a connected linear almost $k$-simple algebraic group with
$k$-rank $\ge2$. Then every non trivial infinite dimensional unitary
representation of $G$ is $p_{0}$-tempered for some explicit $p_{0}$
depending only on the affine Weyl group $W$. Explicitly, the bounds
are (for $n\ge2)$:

\begin{tabular}{|c|c|c|c|c|c|c|c|c|c|c|}
\hline 
$W$ &
$\tilde{A}_{n}$ &
$\tilde{B}_{n}$ &
$\tilde{C}_{n}$ &
$\tilde{D}_{n}$, $n$ even &
$\tilde{D}_{n}$, $n$ odd &
$\tilde{E}_{6}$ &
$\tilde{E}_{7}$ &
$\tilde{E}_{8}$ &
$\tilde{F}_{4}$ &
$\tilde{G}_{2}$\tabularnewline
\hline 
\hline 
$p_{0}$ &
$2n$ &
$2n$ &
$2n$ &
$2(n-1)$ &
$2n$ &
16 &
18 &
29 &
11 &
6\tabularnewline
\hline 
\end{tabular}
\end{thm}
\begin{rem}
The bounds given by the theorem are not in general optimal (but are
optimal in the $\tilde{A}_{n}$, $n\ge2$ case). See the discussion
in \cite{oh2002uniform}, after theorem 7.4.
\end{rem}
\begin{rem}
Recall from Example \ref{exa:Trivial repres} that the trivial representation
is not $p$-tempered for any $p<\infty$.
\end{rem}
Recall that every $G$ as above acts as an automorphism group on a
building, which by taking a finite index subgroup we may assume is
color preserving. As a corollary we can say:
\begin{cor}
Let $V$ be a non trivial $G$-unitary representation of $H$ corresponding
to $G$ as above. Then $V$ is $p_{0}$ tempered for some explicit
$p_{0}$.
\end{cor}
\begin{rem}
\label{rem:Ohs theorem result}The paper \cite{oh2002uniform} gives
an emphasis to the Coxeter group involved and the bounds are found
using the geometry of the root system. One may therefore expect that
the same results would apply to general Affine Iwahori-Hecke algebras,
with arbitrary parameter system $\overrightarrow{q}$ (under the assumption
that $q_{s}>1$, $s\in S$). We are not aware of such a result.
\end{rem}

\section{\label{sec:Representations-of-the}Representations of the Automorphism
Group - Quotients of Buildings}

In this section we connect the action of $H$ on quotients $X\cong B/\Gamma$
to the representation theory of $G$. This will allow us to use strong
results from the representation theory of reductive groups in our
combinatorical setting.

Recall we defined $C_{l}^{\infty}(G)$ as the set of functions $f:G\rightarrow\mathbb{C}$
such that there exists a compact open subgroup $K\subset G$ with
$f(kg)=f(g)$ for every $g\in G,\,k\in K$. Equivalently $f\in C_{l}^{\infty}(G)$
if for the left regular action of $G$ on functions $f'\in\mathbb{C}^{G}$,
$(g\cdot f')(x)=f'(g^{-1}x)$, $f$ is stabilized by some compact
open subgroup $K$.
\begin{prop}
\label{prop:Quotients of Buildings and G represen}Let $\Gamma\subset G$
be a discrete cocompact subgroup of the automorphism group of the
building. 

Let $C^{\infty}(G/\Gamma)=\left\{ f\in C_{l}^{\infty}(G)|\,f(g\gamma)=f(g)\text{\,for\,every\,}\gamma\in\Gamma,\,g\in G\right\} $.

1. $C^{\infty}(G/\Gamma)$ is a representation of $G$ given by the
left action 
\[
(g\cdot f)(x)=f(g^{-1}\cdot x)
\]

2. This representation is (2.a) smooth, (2.b) admissible and (2.c)
unitary.

3. Let $K\subset G$ be a compact open subgroup. Then $C^{\infty}(G/\Gamma)^{K}\cong\mathbb{C}[K\backslash G/\Gamma]$
as a finite dimensional $H_{K}$ representation.
\end{prop}
\begin{proof}
(1) It is enough to prove that if $f$ is smooth so is $g\cdot f$
for every $C_{l}^{\infty}(G)$. This is immediate since if $f$ is
fixed by $K$, $g\cdot f$ is fixed by $gKg^{-1}$.

(2.a.) By definition of a smooth function.

(2.b. + 3) Let $K$ be a compact open subgroup. By definition $C^{\infty}(G/\Gamma)^{K}$
is the set of functions $f:G\rightarrow\mathbb{C}$ such that $f(kx\gamma)=f(x)$,
$\gamma\in\Gamma,\,x\in G,\,k\in K$, which can be identified by $\mathbb{C}^{K\backslash G/\Gamma}$.
To prove that $C^{\infty}(G/\Gamma)^{K}\cong\mathbb{C}[K\backslash G/\Gamma]$
and it is finite dimensional it is enough therefore to prove that
$K\backslash G/\Gamma$ is finite. By definition it is true for $G_{\phi}$
and for general compact open subgroup $K$, $G_{\phi}\cap K$ has
finite index in $G_{\phi}$. Therefore 
\[
\left|K\backslash G/\Gamma\right|\le\left|K\cap G_{\phi}\backslash G/\Gamma\right|\le\left[G_{\phi}:K\cap G_{\phi}\right]\left|G_{\phi}\backslash G/\Gamma\right|
\]
and the last value is finite.

(2.c) We define an inner product: for every two functions $f_{1},f_{2}\in C^{\infty}(G/\Gamma)$
there exists a compact open subgroup $K\subset G_{\phi}$ such that
$f_{1},f_{2}$ can be written as a finite sum $f_{1}=\sum_{i}\alpha_{i}1_{Kg_{i}\Gamma}$,
$f_{2}=\sum_{i}\beta_{i}1_{Kg_{i}\Gamma}$. Define $\left\langle f_{1},f_{2}\right\rangle =\left[G_{\phi}:K\right]^{-1}\sum_{i}\bar{\alpha_{i}}\beta_{i}$.
It is easy to see that this is indeed an inner product and does not
depend on the choice of $K$. 
\end{proof}
\begin{cor}
\label{cor:Expander depends on G}Assume $B$ has a complete Weyl
transitive automorphism group $G$. Then the representation of $H_{\phi}$
on $\mathbb{C}[X_{\phi}]\cong\mathbb{C}[B_{\phi}/\Gamma]$ is $G$-unitary
and is the restriction of the action of $G$ on $\mathbb{C}^{\infty}(G/\Gamma)$.

The decomposition of \textup{$\mathbb{C}[X_{\phi}]$} into a sum of
irreducible representation of $H_{\phi}$ is given by:

1. The decomposition of $\mathbb{C}^{\infty}(G/\Gamma)$ into a sum
of irreducible representations of $G$.

2. Removing all representations without $G_{\phi}$ fixed vectors.

3. Restriction of the finite number of resulting representations to
$G_{\phi}$ fixed vectors. 

In addition, $X$ is an $L_{p}$-expander if and only if every non
trivial irreducible subrepresentation of $G$ on $\mathbb{C}^{\infty}(G/\Gamma)$
with non zero $G_{\phi}$ fixed vectors is $p$-tempered.
\end{cor}
\begin{proof}
The decomposition into irreducible representations follows from proposition
\ref{prop:Quotients of Buildings and G represen} and proposition
\ref{prop: Restriction Proposition}. The temperedness follows from
proposition \ref{prop:Tempereddness and G}.
\end{proof}
We can now state one of the main results of this work.
\begin{thm}
Let $k$ be a non-Archimedean local field with $\mathrm{char}k\ne2$.
Let $G$ be the group of $k$-rational points of a connected linear
almost $k$-simple algebraic group with $k$-rank $\ge2$. 

Let $B$ be the corresponding building on which $G$ acts and $\Gamma$
a cocompact torsion free lattice in $G$. 

Then $X\cong B/\Gamma$ is an $L_{p_{0}}$-expander, where $p_{0}=p_{0}(W)$
depends only on the affine Weyl group $W$ and is given in the table
in theorem \ref{thm:Oh's Theorem}.
\end{thm}
\begin{proof}
Follows from proposition \ref{cor:Expander depends on G} and theorem
\ref{thm:Oh's Theorem}.
\end{proof}
\begin{rem}
Every affine building $B$ of dimension $\ge3$ with irreducible Weyl
group $W$ corresponds to such a group $G$, so this theorem is quite
general. A positive answer to the question in \ref{rem:Ohs theorem result}
would show that the theorem is true for dimension $2$ as well. 
\end{rem}
\begin{rem}
By \cite{first2016ramanujan}, the complexes constructed in \cite{lubotzky2005explicit}
are Ramanujan, i.e. satisfy definition \ref{def:Expander definition}.
\end{rem}
\begin{rem}
The definition in \cite{lubotzky2005ramanujan} (and similarly the
definition in \cite{cartwright2003ramanujan}) only considers functions
on vertices of the graphs and the eigenvalues of the spherical Hecke
operators for buildings of type $\tilde{A}_{n}$. It is shown there
that this property is equivalent to the $2$-temperedness of any such
function which is not part of the trivial representation (assuming
for simplicity that $\Gamma$ is color preserving). This is equivalent
to the property that any subrepresentation of $H$ on $\mathbb{C}[X_{f}]$
which is not trivial and non zero on vertices is $2$-tempered (i.e.
its restriction to some $H_{I}$- $I$ a color of vertex- is non trivial).
In the context of the automorphism group, it is equivalent to considering
only subrepresentations of $\mathbb{C}^{\infty}(G/\Gamma)$ with $G_{I}$
fixed vectors, for some vertex color $I$. Therefore the definition
in this paper is aperiori stronger than the definition in \cite{lubotzky2005ramanujan},
although we do not know if there exists a complex $X$ lying in the
gap between the definitions. A necessary condition for the existence
of $X$ is the existence of a finite dimensional unitary representation
$V$ of $H$ that is nullified on vertices (i.e. $1_{I}V=\{0\}$ for
every vertex color $I$), but is not $2$-tempered. Such a $V$ does
not exist for $\tilde{A}_{n}$, $n\le2$, but does exist for $\tilde{A}_{n}$,
$n\ge3$. This result follows from the classification of unitary $\tilde{A}_{n}$
representations given by Tadi\'{c} (see \cite{tadic1986classification}).
Similar considerations appear in \cite{kang2016riemann} and \cite{first2016ramanujan}.

In any case if one wants to consider operators acting on all the colors
of faces, the definition in this paper is more adequate.
\end{rem}
\begin{rem}
\label{rem:Discussion on temperdness}Corollary \ref{cor:Expander depends on G}
suggests a generalized way of defining an $L_{p}$-expander- the requirement
that every non trivial irreducible subrepresentation of $G$ on $\mathbb{C}^{\infty}(G/\Gamma)$
is $p$-tempered. However, this definition seems to be dependent on
the group $G$ considered and not only the complexes $X$ and $B$.
See also \cite{first2016ramanujan}.
\end{rem}

\part{Spectrum of Operators}

\section{\label{sec:Spectrum-and-Weak}Spectrum and Weak Containment}
\begin{defn}
\label{def: Spectrum definition}Let $V$ be a representation of the
ADH algebra $H$ and $h\in H$. The \emph{point spectrum} $\Sigma_{V}^{p}(h)$
is the set of eigenvalues of $h$ on $V$.

We say that $V$ \emph{supports spectrum} if it is either finite dimensional
or normed with the elements of $H$ acting as bounded operators. If
$V$ is finite dimensional, we define $\Sigma_{V}(h)=\Sigma_{V}^{ap}(h)=\Sigma_{V}^{p}(h)$.
If $V$ is normed, the \emph{spectrum} $\Sigma_{V}(h)$ of $h$, is
the set of $\lambda\in\mathbb{C}$ such that $h-\lambda$ does not
have an inverse with bounded norm on $V$. The \emph{approximate point
spectrum} $\Sigma_{V}^{ap}(h)$ of $h$ is the set of $\lambda\in\mathbb{C}$
such that there exists a series of vectors $v_{n}\in U$, $\left\Vert v_{n}\right\Vert _{V}=1$,
such that $\left\Vert hv_{n}-\lambda v_{n}\right\Vert _{V}\rightarrow0$. 

Denote the \emph{spectral radius $\lambda_{V}(h)$} of $h$ by $\lambda_{V}(h)=\sup\left\{ \left|\lambda\right||\lambda\in\Sigma_{V}(h)\right\} $,
and if $V$ is normed the \emph{norm} $\left\Vert h\right\Vert _{V}$
of $h$ by $\left\Vert h\right\Vert _{V}=\sup\{\left\Vert hv\right\Vert _{V}:v\in V,\,\left\Vert v\right\Vert _{V}=1\}$.
\end{defn}
The following is standard:
\begin{lem}
\label{lem:Gelfand's formula}We have Gelfand's formula $\lambda_{V}(h)=\limsup\left\Vert h^{n}\right\Vert _{V}^{1/n}$
and if $V$ is unitary, then $\left\Vert h\right\Vert _{V}=\sqrt{\lambda_{V}(hh^{*})}$.
\end{lem}
Now we can now compare arbitrary representations that support spectrum:
\begin{defn}
Assume $V_{1},V_{2}$ are $H$-representations that support spectrum.
If for every $h\in H$, $\lambda_{V_{1}}(h)\le\lambda_{V_{2}}(h)$
we say that $V_{1}$ is \emph{weakly contained} in $V_{2}$.
\end{defn}
\begin{rem}
The above definition is not standard as far as we know. Similar definitions
exist for unitary representations of locally compact groups and $C^{*}$
algebras (see \cite{cowling1988almost} and references therein). One
of the equivalent definitions is that for every $h\in H$, $\left\Vert h\right\Vert _{V_{1}}\le\left\Vert h\right\Vert _{V_{2}}$.
For unitary representations (but not for general normed representations)
this is equivalent to our definition by lemma \ref{lem:Gelfand's formula}.
\end{rem}

\section{\label{sec:The-Point-Spectrum}The Point Spectrum of $T_{p}(B_{f})$}

We want to understand the eigenvalues of the action of $H$ on $T_{p}(B_{f})$.
Recall definition \ref{def:F_delta} of $f_{\delta}$ for $f\in\mathbb{C}^{B_{f}}$.
We will need the following lemma:
\begin{lem}
Let $f\in T_{p}(B_{f})$ and $h\in H$. Then \textup{$\left\Vert h\left(f_{\delta}\right)-(hf)_{\delta}\right\Vert _{p}/\left\Vert f_{\delta}\right\Vert _{p}\rightarrow0$
as $\delta\rightarrow0$.} Therefore for every $\lambda\in\mathbb{C}$,
$\left\Vert hf_{\delta}-\lambda f_{\delta}\right\Vert _{p}/\left\Vert f_{\delta}\right\Vert _{p}\rightarrow0$
if and only if $\left\Vert (hf)_{\delta}-\lambda f_{\delta}\right\Vert _{p}/\left\Vert f_{\delta}\right\Vert _{p}\rightarrow0$
.
\end{lem}
\begin{proof}
Write $h=\sum_{d}\alpha_{d}h_{d}$. Let $L=\max_{d:\alpha(d)\ne0}l(d)$
($l(d)$ is as in the temperedness definition),$\left|h\right|=\sum_{d}\left|\alpha_{d}\right|h_{d}$.

Let $F\in\mathbb{C}^{B_{f}}$, $F(\sigma)=\left(\left|h\right|\left|f\right|\right)(\sigma)$.
From \ref{lem:-H acts as bounded operators on Tp} we have $\left\Vert F_{\delta}\right\Vert _{p}\le M\left\Vert f_{\delta}\right\Vert _{p}$
for some $M\in\mathbb{R}_{>0}$ and $\delta>0$ small enough.

Let $\sigma$ be some face. We wish to understand $\left|hf_{\delta}(\sigma)-\left(hf\right)_{\delta}(\sigma)\right|$.
For $l(d(\sigma,\sigma'))\le L$ we have $f_{\delta}(\sigma')=(1-\delta)^{l(\sigma)}f(\sigma')+(1-\delta)^{l(\sigma)}r_{\sigma}(\sigma')f(\sigma')$
for some 
\[
1-(1-\delta)^{-L}\le r_{\sigma}(\sigma')=1-(1-\delta)^{l(\sigma')-l(\sigma)}\le1-\left(1-\delta\right)^{L}
\]

Notice that for $\delta$ small enough $\left|r_{\sigma}(\sigma')\right|\le4\delta L$.

Let $g_{\sigma}\in\mathbb{C}^{B_{f}}$ be defined by $g_{\sigma}(\sigma')=r_{\sigma}(\sigma')f(\sigma')$.
We have $f_{\delta}=(1-\delta)^{l(\sigma)}(f+g_{\sigma})$. For $l(d(\sigma,\sigma'))\le L$
we have $\left|g_{\sigma}\right|(\sigma')=\left|r_{\sigma}(\sigma')\right|\left|f(\sigma')\right|\le4L\delta\left|f\right|(\sigma')$. 

Now
\begin{align*}
\left|(hf_{\delta})(\sigma)-\left(hf\right)_{\delta}(\sigma)\right| & =\left|(1-\delta)^{l(\sigma)}\left(\left(hf\right)(\sigma)+\left(hg_{\sigma}\right)(\sigma)\right)-\left(hf\right)_{\delta}(\sigma)\right|=\\
 & =\left|(1-\delta)^{l(\sigma)}\left(hg_{\sigma}\right)(\sigma)\right|\,\le\,(1-\delta)^{l(\sigma)}\left(\left|h\right|\left|g_{\sigma}\right|\right)(\sigma)\\
 & \le(1-\delta)^{l(\sigma)}4L\delta\left(\left|h\right|\left|f\right|\right)(\sigma)\le4L\delta F_{\delta}(\sigma)
\end{align*}

Taking the $p$-s power and summing over all $\sigma\in B_{f}$, we
have:

\[
\left\Vert hf_{\delta}(C)-\left(hf\right)_{\delta}\right\Vert _{p}\le\delta4L\left\Vert F_{\delta}\right\Vert _{p}\le\delta4LM\left\Vert f_{\delta}\right\Vert _{p}
\]

and as $\delta\rightarrow0$, $\left\Vert hf_{\delta}(C)-\left(hf\right)_{\delta}\right\Vert _{p}/\left\Vert f_{\delta}\right\Vert _{p}\rightarrow0$
as required.
\end{proof}
\begin{cor}
The point spectrum of $h\in H$ on $T_{p}(B_{f})$ is contained in
the approximate point spectrum of $h$ on $L_{p}(B_{f})$.
\end{cor}
\begin{proof}
Assume $f\in T_{p}(B_{f})$ such that $hf=\lambda f$. Therefore we
have $\left\Vert (hf)_{\delta}-\lambda f_{\delta}\right\Vert _{p}/\left\Vert f_{\delta}\right\Vert _{p}=0$.
By the last lemma $\left\Vert hf_{\delta}-\lambda f_{\delta}\right\Vert _{p}/\left\Vert f_{\delta}\right\Vert _{p}\rightarrow0$
and $\lambda$ is in the approximate point spectrum of $h$ on $L_{p}(B_{f})$.
\end{proof}
\begin{cor}
Let $V$ be a representation of $H$, $0\ne v\in V$, $h\in H$ and
$hv=\lambda v$.

Assume that some non zero geometric embedding of $v$ is $p$-tempered.
Then $\lambda$ belongs to the approximate point spectrum of $h$
on $L_{p}(B_{f})$.
\end{cor}
\begin{cor}
\label{cor:Main Theorem proof}If a finite dimensional representation
$V$ is $p$-tempered then $V$ is weakly contained in $L_{p}(B_{f})$.
More precisely, for every $h\in H$ the set of eigenvalues of $h$
on $V$ is contained in the approximate point spectrum of $h$ on
$L_{p}(B_{f})$.
\end{cor}
\begin{rem}
The same logic allows us to compare arbitrary admissible $G$ representations
with the left action of $G$ on $L_{p}(G)\cap C_{l}^{\infty}(G)$,
using the $H(G)$ action on the two spaces. Notice that for every
$K\subset G$ the action of $H(G,K)$ on the $K$-fixed vectors of
an admissible representation is finite dimensional and its action
on the $K$-fixed vectors of $L_{p}(G)\cap C_{l}^{\infty}(G)$ is
normed. In particular, the same proof shows that if $V$ is admissible
and $p$-tempered then it is weakly contained in $L_{p}(G)\cap C_{l}^{\infty}(G)$.
It can be probably generalized to other locally profinite groups,
since no essential property of the building was used, Thus generalizing
theorem $1$ on \cite{cowling1988almost} in the locally profinite
case to $p\ne2$.
\end{rem}

\section{\label{sec:Generalized-Serre-Theorem}Generalized Serre Theorem}

The following proposition generalizes a well known theorem, usually
attributed to Serre (but also appears in \cite{mckay1981expected}),
for graphs with large injectivity radius (or girth). It applies to
any normal element of $H$. The proof is based on \cite{li2004ramanujan}.
Compare also \cite{first2016ramanujan}, theorem 5.1.
\begin{defn}
Let $X$ be a quotient of the building $B$. The \emph{injectivity
radius}\textbf{ }of $X$ is the length of the shortest distance $d\in W$
between two chambers $C_{1}\ne C_{2}$ of the building that cover
the same chamber in $X$.
\end{defn}
\begin{thm}
\label{Generelized-Serre}Let $h\in H$ normal operator and $\lambda$
in the spectrum of $h$ on $L_{2}(B_{f})$. Then there exists an $\epsilon(N)=\epsilon_{B,h}(N)$
with $\epsilon(N)\rightarrow0$ as $N\rightarrow\infty$, such that
for every finite quotient $X$ of the building $B$ with injectivity
radius greater than $N$, there exists $\lambda'$ in the spectrum
of $h$ on $\mathbb{C}^{X_{f}}$, with $\left|\lambda-\lambda'\right|<\epsilon(N)$.
\end{thm}
\begin{proof}
Let $l=l(d)$ be the largest length of an element $h_{d}$ appearing
in $h$. We claim that for any $\epsilon>0$ there exists an $N\in\mathbb{N}$
and $0\ne f_{N}\in\mathbb{C}[B_{f}]$, such that $f_{N}$ is supported
on faces at distance $N/2-l$ around $C_{0}$ and we have $\left\Vert hf_{N}-\lambda f_{N}\right\Vert _{2}/\left\Vert f_{N}\right\Vert _{2}<\epsilon$.
Take an approximate eigenfunction $f\in L_{2}(B_{f})$ with $\left\Vert hf-\lambda f\right\Vert /\left\Vert f\right\Vert <\epsilon/2$
and call its restriction to distance $N/2-l$ , $f_{N}\in\mathbb{C}[B_{f}]$.
Taking $N\rightarrow\infty$ we know that $\left\Vert hf_{N}-\lambda f_{N}\right\Vert /\left\Vert f_{N}\right\Vert \rightarrow\left\Vert hf-\lambda f\right\Vert /\left\Vert f\right\Vert <\epsilon/2$.
So there exists a finite $N$ with $\left\Vert hf_{N}-\lambda f_{N}\right\Vert /\left\Vert f_{N}\right\Vert <\epsilon$,
as needed.

By the assumptions on $f_{N}$, if $X_{f}$ has injectivity radius
greater than $N$, both $f_{N}$ and $hf_{N}$ can be projected to
$L_{2}(X_{f})$, i.e. we have $f\in L_{2}(X_{f})$ with $\left\Vert hf-\lambda f\right\Vert _{X}/\left\Vert f\right\Vert _{X}<\epsilon$.
By the normality of $h$, there exists an eigenvalue of $h-\lambda$
on $L_{2}(X_{f})$, with absolute value smaller than $\epsilon$,
and the claim follows.
\end{proof}
\begin{problem}
The original proof of \cite{mckay1981expected} shows that as the
injectivity radius grows, the spectrum of the adjacency operator converges
to the spectral measure of the adjacency operator on the tree. We
therefore ask if it holds here as well, i.e. the spectrum of every
normal (or self-adjoint) $h$ converges to the spectral measure of
$h$ on $L_{2}(B_{f})$. Compare (in slightly different settings)
\cite{abert2012growth}, theorem 1.2.
\end{problem}

\section{\label{sec:Alon-Boppana-Theorem}Alon-Boppana Theorem}

The following proposition generalizes directly the classical Alon-Boppana
theorem. For simplicity we consider operators of $H_{\phi}$ only.
Our treatment follows \cite{lubotzky1994discrete} proposition 4.5.4.
\begin{defn}
An element $h\in H_{\phi}$ is called a \emph{random-walk operator}
if it is self adjoint and a non-negative sum of the basis operators
$h_{w}$.
\end{defn}
A random walk operator defines (after normalization) a random walk
on $B_{\phi}$. Since we have symmetry among all all chambers, Kesten's
argument in \cite{kesten1959symmetric}, lemma 2.11, gives that $\lambda_{2}(h)=\left\Vert h\right\Vert _{2}=\limsup\left(\left\Vert h^{n}1_{C}\right\Vert _{2}\right)^{1/n}$.
We can now state:
\begin{thm}
\label{Alon Boppana}Let $X$ be a quotient of the building $B$.
Assume the largest distance (in gallery length) between two chambers
in $X$ is $N$. Let $h\in H_{\phi}$ be a random-walk operator.

Then there exists an $\epsilon(N)=\epsilon_{B,h}(N)$ with $\epsilon(N)\rightarrow0$
as $N\rightarrow\infty$, such that the largest eigenvalue of $h$
on $L_{2}^{0}(X_{\phi})$ is at least $\lambda_{2}(h)-\epsilon(N)$.
\end{thm}
\begin{proof}
Write $\left\Vert ..\right\Vert _{B},\,\left\Vert ..\right\Vert _{X}$
for the $L_{2}$ norms of the two spaces. Choose two chambers $C_{0}^{X},C_{1}^{X}\in X_{\phi}$
of distance $N$ and let $C_{0},C_{1}\in B_{\phi}$ be two chambers
which cover $C_{0}^{X},C_{1}^{X}$ and are of distance $N$. The fact
that $h$ is a non negative sum of $h_{w}$-s tells us that $h^{n}1_{C_{0}}$
is positive in every coordinate. Since $h^{n}1_{C_{0}^{X}}$ is the
projection of $h^{n}1_{C_{0}}$, the norm just grows, i.e. $\left\Vert h^{n}1_{C_{0}}\right\Vert _{B}\le\left\Vert h^{n}1_{C_{0}^{X}}\right\Vert _{X}$.

Let $l=l(w)$ the largest distance of an element $w\in W$ appearing
in $h$. Then $h^{n}1_{C_{0}^{X}}$ and $h^{n}1_{C_{1}^{X}}$ have
disjoint supports for $n<N/l$. Therefore:

\[
\left\Vert h^{n}(1_{C_{0}^{X}}-1_{C_{1}^{X}})\right\Vert _{X}=\left\Vert h^{n}1_{C_{0}^{X}}-h^{n}1_{C_{1}^{X}}\right\Vert _{X}=\left\Vert h^{n}1_{C_{0}^{X}}\right\Vert _{X}+\left\Vert h^{n}1_{C_{0}^{X}}\right\Vert _{X}\ge
\]

\[
\ge\left\Vert h^{n}1_{C_{0}}\right\Vert _{B}+\left\Vert h^{n}1_{C_{0}}\right\Vert _{B}\ge2(\lambda_{2}(h)-\epsilon(n))^{n}
\]

for some $\epsilon(n)\rightarrow0$ as $n\rightarrow\infty$.

Notice that $1_{C_{0}^{X}}-1_{C_{1}^{X}}\in L_{2}^{0}(X_{\phi})$.
Take the $n$-th root. We found that there exists a function $f\in L_{2}^{0}(X_{\phi})$
with $\left(\left\Vert h^{n}f\right\Vert _{X}/\left\Vert f\right\Vert _{X}\right)^{1/n}\ge\lambda_{2}(h)-\delta(n)$.
Since $h$ is self adjoint the last inequality means it has an eigenvalue
of absolute value $\ge\lambda_{2}(h)-\epsilon(n)$.
\end{proof}
\begin{rem}
Notice that given $h\in H$ one can estimate $\epsilon(n)$ in this
theorem, by analyzing the rate of convergence of $\left(\left\Vert h^{n}1_{C}\right\Vert _{2}\right)^{1/n}$
to $\lambda_{2}(h)$.
\end{rem}
\begin{rem}
To extend the result to all of random-like operators of $H$ one should
replace chambers by other faces. Then the same proof applies for $\tau(h)=\max_{I\,\text{spherical}}\limsup\left(\left\Vert h^{n}1_{\sigma}\right\Vert _{2}\right)^{1/n}$
where $\sigma$ is some face of color $I$, and by extension of Kesten's
argument $\tau(h)=\lambda_{2}(h)$. The details are left to the reader.
\end{rem}

\part{The Affine Case}

\section{\label{sec:Color-Rotations}Color Rotations}

Before discussing the affine case we should extend our algebra a little
by color rotations. It is useful since this way we can talk about
quotients by type rotating automorphisms. It will also be easier to
work with the affine Hecke algebra. Since the claims are simple and
similar to previous ones, we skip the proofs.
\begin{defn}
The \emph{automorphism group} of $S$ is the group of bijections $\omega:S\rightarrow S$
preserving the Coxeter values $m_{i,j}$- i.e. for every $s,t\in S$,
$m_{s,t}=m_{\omega(s),\omega(t)}$. Denote by $\hat{\Omega}$ a subgroup
of the automorphism group of $S$, such that $q_{s}=q_{\omega(s)}$
for every $s\in S$ and $\omega\in\hat{\Omega}$. 
\end{defn}
While the restriction $q_{s}=q_{\omega(s)}$ is not really necessary,
it will be simpler to assume it. The action of $\omega\in\hat{\Omega}$
on $S$ extends to a group automorphism $\omega:W\rightarrow W$.
We can therefore define:
\begin{defn}
The group $\hat{W}=W\rtimes\hat{\Omega}$ is called the \emph{$\hat{\Omega}$-extended
Coxeter group.}
\end{defn}
Our standard to semi-direct product is that multiplication in $\hat{W}$
is given by $\omega\cdot w=\omega(w)\cdot\omega$ and the relations
in $W,\Omega$. 
\begin{lem}
By defining $l(\omega)=0$ for $\omega\in\hat{\Omega}$, we can extend
the length function $l:W\rightarrow\mathbb{N}$ to $l:\hat{W}\rightarrow\mathbb{N}$.
\end{lem}
Notice that every $\omega\in\hat{\Omega}$ acts on distances as well
$\omega:W_{I_{1}}\backslash W/W_{I_{2}}\rightarrow W_{\omega(I_{1})}\backslash W/W_{\omega(I_{2})}$,
$\omega(W_{I_{1}}\backslash w/W_{I_{2}})=W_{\omega(I_{1})}\backslash\omega(w)/W_{\omega(I_{2})}$.

Denote $\hat{B}_{f}=B_{f}\times\hat{\Omega}$ and $\hat{B}_{I}=B_{I}\times\hat{\Omega}$.
For every face of the building, each $\omega\in\Omega$ can be associated
to a possible recoloring of it . Therefore $(\sigma,\omega)\in B_{f}\times\hat{\Omega}=\hat{B}_{f}$
can be seen as a ``recolored face'' in the ``recolored building''.
We define an action $h_{\omega}:\mathbb{C}^{\hat{B}_{f}}\rightarrow\mathbb{C}^{\hat{B}_{f}}$
by the ``recoloring'' $h_{\omega}f(\sigma,\omega')=f(\sigma,\omega'\omega$).
We also have an action of $H$ on $\mathbb{C}^{\hat{B}_{f}}$ acting
on every coloring separately, because $\mathbb{C}^{\hat{B}_{f}}\cong\bigoplus_{\omega\in\Omega}\mathbb{C}^{B\times\{\omega\}}$.
Then it is easy to notice that $h_{\omega}h_{d}=h_{\omega(d)}h_{\omega}$
for every $d\in W_{I_{1}}\backslash W/W_{I_{2}}$.
\begin{defn}
The \emph{$\hat{\Omega}$-extended Iwahori-Hecke algebra $\hat{H}_{\phi}$}is
the algebra generated by the $h_{\omega},\omega\in\Omega$ action
on $\mathbb{C}^{\hat{B}_{f}}$ and by $H_{\phi}$. The \emph{$\hat{\Omega}$-color
rotating all dimensional Hecke algebra $\hat{H}$} is the algebra
generated by $h_{\omega},\omega\in\hat{\Omega}$ and by $h\in H$. 
\end{defn}
The following proposition states some basic properties of the algebras:
\begin{prop}
We have $\hat{H}_{\phi}=H_{\phi}\rtimes\hat{\Omega}$ and $\hat{H}=H\rtimes\hat{\Omega}$,
i.e. as sets we have a direct product and we have $h_{\omega}h_{d}=h_{\omega(d)}h_{\omega}$
for every $d\in W_{I_{1}}\backslash W/W_{I_{2}}$

The algebra $\hat{H}_{\phi}$ is generated by the Iwahori-Hecke relations
as well as the relation $h_{\omega}h_{w}=h_{\omega(w)}h_{\omega}$
for $w\in W$, $\omega\in\Omega$. Define for $w'=w\cdot\omega\in\hat{W}$,
$h_{w'}=h_{w}h_{\omega}$. The algebra $\hat{H}_{\phi}$ is spanned
by $h_{w}$, $w\in\hat{W}$. For $w,w'\in\hat{W}$ with $l(w)+l(w')=l(ww')$
we have $h_{w}h_{w'}=h_{ww'}$. The algebra $\hat{H}$ is spanned
as a vector space by $h_{\omega}h_{d}$, $\omega\in\Omega$, $d\in W_{I_{1}}\backslash W/W_{I_{2}}$.
\end{prop}
Let us turn to some of the representation theory involved. First,
similar to proposition \ref{prop:Induction-and-restriction-full-proposition}
we have an equivalence of categories between $H_{\phi}$-representations
and $H$-representations. This equivalence preserves irreducible representations,
unitary representations, finite dimensional representations, $p$-finite
representations and $p$-tempered representations.

The induction and restriction operators of section \ref{sec:Induction-and-restriction}
can be used to study the relations between representations of the
ADH algebra $H$ (or the Iwahori-Hecke algebra $H_{\phi}$) and the
$\hat{\Omega}$-color rotating ADH algebra $\hat{H}$ (or $\hat{H}_{\phi}$)
from section \ref{sec:Color-Rotations}. This time $H$ is the smaller
algebra and $\hat{H}$ contains it.

The main difference between this case and section \ref{sec:Induction-and-restriction}
is that the unit is the same in both algebras, so the situation resembles
induction and restriction between representations of a group and a
subgroup. We will only state the proposition below. The proof is omitted.
\begin{prop}
\label{prop:Color-Rotating-Prop}Let $V$ be a representation of $H$
and $V'=\mbox{ind}_{H}^{\hat{H}}V$ the induced representation of
$\hat{H}$. Let $U$ be a representation of $\hat{H}$ and $U'=\mbox{res}_{H}^{\hat{H}}U$
the restricted representation of $H$. Then:

1. As a vector space $V'\cong V\otimes\mathbb{C}[\Omega]$ and therefore
$\dim V'=\left|\Omega\right|\dim V$.

2. As vector spaces $U'\cong U$ and therefore they have the same
dimension.

3. $\mbox{res}_{H}^{\hat{H}}V'=\mbox{res}_{H}^{\hat{H}}\mbox{ind}_{H}^{\hat{H}}V$
is isomorphic to a direct sum of $\left|\Omega\right|$ times the
representation $V$.

4. $V$ is unitary if and only if $V'$ is unitary and $U$ is unitary
if and only if $U'$ is unitary.

5. $V$ is $p$-tempered if and only if $V'$ is $p$-tempered and
$U$ is $p$-tempered if and only if $U'$ is $p$-tempered. 

6. If $V\cong V_{1}\oplus V_{2}$ then $V'\cong\mbox{ind}_{H}^{\hat{H}}V_{1}\oplus\mbox{ind}_{H}^{\hat{H}}V_{2}$.
If $U\cong U_{1}\oplus U_{2}$ then $U'\cong\mbox{res}_{H}^{\hat{H}}U_{1}\oplus\mbox{res}_{H}^{\hat{H}}U_{2}$.

7. Matrix coefficients give geometric realization of $U$ as a subrepresentation
of $\hat{H}$ on $\mathbb{C}^{\hat{B}_{f}}=\mathbb{C}^{B_{f}\times\Omega}$.

8. Let \textup{$X=B$} or $X=B/\Gamma$ be a building or a quotient
of a building. Then we have a unitary representation of $\hat{H}$
on $L_{2}(\hat{X}_{f})=L_{2}(X_{f}\times\Omega)$. The algebra $H$
acts on the same space by restriction. If $\Gamma$ is color preserving
then $L_{2}(\hat{X})\cong\mbox{ind}_{H}^{\hat{H}}L_{2}(X)$.
\end{prop}
Notice that irreducibility is not necessarily preserved by the induction
and restriction operations. In particular, we do not have an equivalence
of categories between $H$ representations and $\hat{H}$ representations.

\section{\label{sec:Affine-Buildings}Affine Root Systems}

Most of the following is very standard. We follow \cite{parkinson2006buildings}
for some details about reducible root systems that can be ignored
when first reading this.

Let $R$ be a possibly reduced, crystallographic\emph{ }irreducible\emph{
root system} in a euclidean space $V_{R}$ of dimension $n$. In other
words: (i)$R$ is a finite set of elements $\alpha\in V$ which span
$V$. (ii) For every $\alpha\in R$ we have $s_{\alpha}(R)=R$ where
$s_{\alpha}:V_{R}\rightarrow V_{R}$ is the reflection defined by
$s_{\alpha}(x)=x-2(\left\langle \alpha,x\right\rangle /\left\langle \alpha,\alpha\right\rangle )\alpha$.
(iii) We have $2\left\langle \alpha,\beta\right\rangle /\left\langle \alpha,\alpha\right\rangle \in\mathbb{Z}$
for every $\alpha,\beta\in R$. (iv) The $s_{\alpha}$-s do not stabilize
any non trivial proper subspace of $V_{R}$.

The set of \emph{simple roots} is denoted $\Delta=\left\{ \alpha_{i}:i=1,...,n\right\} $.
It is unique after the choice of positive roots. The corresponding
\emph{coroot system} is $R^{\vee}=\left\{ \alpha^{\vee}|\alpha^{\vee}=2\alpha/\left\langle \alpha,\alpha\right\rangle ,\alpha\in R\right\} $
with a set of simple coroots $\left\{ \alpha_{i}^{\vee}:i=1,...,n\right\} $.
The set of \emph{simple coweights} is $\hat{\Delta}=\left\{ \beta_{i}:i=1,...,n\right\} $
. It is the dual basis of $\Delta$, i.e. we have $\left\langle \alpha_{i},\beta_{j}\right\rangle =\delta_{i,j}$.

The \emph{coroot lattice} is $Q=\left\{ \sum_{i=1}^{n}z_{i}\alpha_{i}^{\vee}:\,z_{i}\in\mathbb{Z}\right\} $.
The \emph{coweight lattice }is $P=\left\{ \lambda\in V:\,\left\langle \lambda,\alpha\right\rangle \in\mathbb{Z}\,\,\forall\alpha\in R\right\} =\left\{ \sum_{i=1}^{n}z_{i}\hat{\alpha}_{i}:\,z_{i}\in\mathbb{Z}\right\} $.
The coroot lattice $Q$ is a sublattice of the coweight lattice $P$
and the group $\hat{\Omega}=P/Q$ is finite and abelian. The set of
\emph{dominant coweights} is $P^{+}=\left\{ \lambda\in V:\,\left\langle \lambda,\alpha\right\rangle \in\mathbb{N}\,\,\forall\alpha\in R\right\} =\left\{ \sum_{i=1}^{n}z_{i}\beta_{i}:\,z_{i}\in\mathbb{N}\right\} $.
From this description it is immediate that every $\beta\in P$ can
be written as $\beta_{1}-\beta_{2}$, $\beta_{1},\beta_{2}\in P^{+}$.

The \emph{spherical Weyl group is} $W_{0}=\left\langle s_{\alpha}\left|s_{\alpha}(x)=x-\left\langle \alpha^{\vee},x\right\rangle \alpha,\alpha\in R\right.\right\rangle $.
It is generated by the reflections determined by $R$. The set of
simple roots allows us to identify $W_{0}$ with the Coxeter group
generated by $s_{i}=s_{\alpha_{i}},i=1,...,n$.

The \emph{affine Weyl group} is $W=Q\rtimes W_{0}$. $W$ is the Coxeter
group generated by $s_{1},...,s_{n}$ and another affine reflection
$s_{0}$, defined by $s_{0}(x)=x-\left(\left\langle \alpha_{0}^{\vee},x\right\rangle -1\right)\alpha_{0}$
where $\alpha_{0}$ is the highest root.

The \emph{extended affine Weyl group is} $\hat{W}=P\rtimes W_{0}$.
The (finite and abelian) group $\hat{\Omega}=P/Q$ is isomorphic to
a subgroup of the automorphism group of $S$, and we have $\hat{W}=W\rtimes\hat{\Omega}$.
The results of section \ref{sec:Color-Rotations} apply to it. In
general $\hat{\Omega}$ is not the full automorphism group of $S$.

Since we will work with vertices we will call vertices of color $\{0,...,n\}-i$,
vertices of \emph{type} $i$. A vertex type $i$ is called \emph{good}
(as in the notations of \cite{parkinson2006buildings}, section 3.4)
if there exists $\omega_{i}\in\hat{\Omega}$ with $\omega_{i}(0)=i$.
The good vertices are equal to the more standard \emph{special vertices,
}except for root systems of type $BC_{n}$ or $C_{n}$. For every
$0\ne\omega\in\hat{\Omega}$, we have $\omega(0)\ne0$, and therefore
there exists a bijection between $\hat{\Omega}$ and the good types.
The coroot lattice $Q$ corresponds to the vertices of type $0$ in
$\mathbb{W}$ and the coweight lattice $P$ corresponds to the vertices
of good type in $\mathbb{W}$.

As explained in \cite{parkinson2006buildings} section 3.8, we may
assume that $q_{s}=q_{\omega(s)}$ for every $s\in S$ and $\omega\in\hat{\Omega}$.
This is the reason we do not assume the root system is reduced, see
also the bipartite graph example below. The results of section \ref{sec:Color-Rotations}
apply to this case. We define the \emph{extended ADH algebra} $\hat{H}=H_{\hat{\Omega}}=H\rtimes\hat{\Omega}$
and the \emph{extended Iwahori-Hecke algebra} $\hat{H}_{\phi}=H_{\phi}\rtimes\hat{\Omega}$. 

The Coxeter complex $\mathbb{W}$ is isomorphic as a topological space
to $V_{R}$. The different reflections cut $V_{R}$ into chambers
(sometimes called alcoves) and this defines a simplicial structure
on $V_{R}$ which is isomorphic to $\mathbb{W}$. The chambers correspond
to the elements of $W$, and two chambers share a panel of type $s$
if and only if they correspond to elements of the from $w$, $ws$.

The \emph{fundamental chamber} is the set $\left\{ v\in V_{R}:\left\langle \alpha_{i},v\right\rangle >0,\,i=1,...,n,\,\left\langle \alpha_{0},v\right\rangle <1\right\} $.
The\emph{ fundamental (or dominant) sector} is the set $\left\{ v\in V_{R}:\left\langle \alpha_{i},v\right\rangle >0,\,i=1,...,n\right\} $.
The \emph{fundamental parallelotope} is the set $\left\{ \sum_{i=1}^{n}x_{i}\beta_{i}:\,0\le x_{i}\le1\right\} $.
We denote by $A_{0}$ the set of $w\in W$ corresponding to the chambers
of the fundamental parallelotope and by $\hat{A}_{0}$ the set of
$\hat{w}\in\hat{W}$ corresponding to such chambers. We have $\hat{A}_{0}=A_{0}\cdot\Omega$
(as sets, multiplication takes place in $\hat{W}$). It is standard
that $\left|\hat{A}_{0}\right|=W_{0}$. 

We now state a basic structure theorem for the extended Coxeter group
$\hat{W}$. Surprisingly, we could not find a standard reference for
this theorem in the literature (It does appear however in \cite{gashi2012looping},
proof of theorem 8.2).
\begin{thm}
\label{thm:Structure Theorem}Each element $w\in\hat{W}$ can be written
\emph{uniquely} as $w=w_{0}\beta a$, with $w_{0}\in W_{0}$, \textup{$\beta\in P^{+}$}
and $a\in\hat{A}_{0}$. Moreover, this decomposition satisfies $l(w)=l(w_{0})+l(\beta)+l(a)$.

The element $\beta\in P^{+}$ satisfies $\beta=\prod_{i=1}^{n}\beta_{i}^{m_{i}}$
for some unique $m_{i}\in\mathbb{N}$ and we have $l(\beta)=\sum m_{i}l(\beta_{i})$.

As a corollary, $h_{w}=h_{w_{0}}\left(\prod_{i=1}^{n}h_{\beta_{i}}^{m_{i}}\right)h_{a}$
and $q_{w}=q_{w_{0}}\left(\prod_{i=1}^{n}q_{\beta_{i}}^{m_{i}}\right)q_{a}$.
\end{thm}
\begin{proof}
It is enough to prove the first statement, since the decomposition
of $\beta\in P^{+}$ is well known and the claims about $q_{w}$ and
$h_{w}$ are a direct corollary. 

Denote the chamber corresponding to $w\in\hat{W}$ by $C_{w}$. The
correspondence $w\to C_{w}$ is $\left|\hat{\Omega}\right|$ to $1$,
and the fundamental chamber is $C_{0}=C_{Id}$. 

The decomposition $\hat{W}=W_{I_{0}}{}^{I_{0}}\hat{W}=W_{0}{}^{I_{0}}\hat{W}$
is well known and is a version of lemma \ref{Coxeter decomposition lemma}
for the extended Coxeter group. It remains to prove that each $w\in{}^{I_{0}}\hat{W}$
can be written as $w=\beta a$, $\beta\in P^{+}$ and $a\in\hat{A}_{0}$.

The elements of $^{I_{0}}\hat{W}$ are elements $w\in\hat{W}$ such
that $l(sw)>l(w)$ for any $s\in I_{0}$. Since the length of element
in $\hat{W}$ is the number of hyperplanes separating $C_{w}$ from
$C_{0}$, $C_{w}$ is on the same side on the $s$-hyperplane of $C_{0}$.
Therefore, $C_{w}$ is in the fundamental sector. Choose now an internal
point $v_{w}\in C_{w}$. Then $\left\langle \alpha_{i},v_{w}\right\rangle >0$
for every $i=1,...,n$. Let $\beta\in P^{+}$ be the unique element
satisfying $\left\langle \alpha_{i},\beta\right\rangle =\left\lfloor \left\langle \alpha_{i},v_{w}\right\rangle \right\rfloor \ge0$,
and $a=\beta^{-1}w$. A point $v_{a}\in C_{a}$ satisfies $0\le\left\langle \alpha_{i},v_{w}\right\rangle \le1$
for $i=1,...,n$, and therefore $a\in\hat{A}_{0}$. By this description
it is also clear that $\beta$ is the only element in $P$ satisfying
$\beta^{-1}w\in\hat{A}_{0}$. Finally, each hyperplane separating
$C_{\beta}$ and $C_{0}$ also separates each point $v\in V_{R}$
satisfying $\left\langle \alpha_{i},v\right\rangle \ge\left\langle \alpha_{i},\beta\right\rangle $
for $i=1,...,n$. Therefore Each such hyperplane also separates $C_{w}$
from $C_{0}$, and therefore $l(w)=l(\beta)+l(a)$
\end{proof}
We have a direct nice corollary to the theorem. Define an \emph{abstract
parameter system} as a set of intermediates $\vec{u}=\left(u_{s}\right)_{s\in S}$,
satisfying the parameter system condition, i.e $u_{s}=u_{s'}$ when
$m_{s,s'}$ is odd, and also $u_{s}=u_{\omega(s)}$ for $\omega\in\hat{\Omega}$.
We also define $u_{\omega}=1$ for $\omega\in\hat{\Omega}$. Then
there exists for every $w\in\hat{W}$ a well defined monomial $u_{w}$
satisfying $u_{ww'}=u_{w}u_{w'}$ if $l(ww')=l(w)+l(w')$. In the
single parameter case we simply have $u_{w}=w^{l(w)}$.
\begin{defn}
For a subset $A\subset\hat{W}$ we define the formal series $P_{A}(\vec{u})=\sum_{w\in A}h_{w}u_{w}\in\hat{H}_{\phi}[[\vec{u}]]$. 

The formal series $P_{\hat{W}}(\vec{u})=\sum_{w\in\hat{W}}h_{w}u_{w}$
is called the\emph{ generalized Poincare series} of the Iwahori-Hecke
algebra.
\end{defn}
\begin{cor}
\label{cor:Gen_Poincare_Series}As a formal series, we have:

\[
P_{\hat{W}}(\vec{u})=P_{W_{0}}(\vec{u})\left(\prod_{i=1}^{n}\frac{1}{\left(1-h_{\beta_{i}}u_{\beta_{i}}\right)}\right)P_{\hat{A}_{0}}(\vec{u})
\]
\end{cor}
\begin{rem}
This generalized Poincare series was first considered by Gyoja in
\textbf{\uline{\mbox{\cite{gyoja1983generalized}}}} (see also
\cite{hoffman2003remarks}), where is was proven that it is a rational
function. The formal series $\sum_{w\in\hat{W}}u_{w}$ (or, in the
single parameter case $\sum_{w\in\hat{W}}u^{l(w)}$) is called the
\emph{Poincare series} of the extended Coxeter group $\hat{W}$. Explicit
formulas for it are classical. While it is usually defined for the
regular Coxeter group $W$ and not the extended version $\hat{W}$,
it does not really matter as by $\hat{W}=W\rtimes\hat{\Omega}$, $P_{\hat{W}}(u)=P_{W}(u)P_{\hat{\Omega}}(u)=P_{\hat{\Omega}}(u)P_{W}(u)$. 
\end{rem}
\begin{example}
\label{exa:Regular graph}Consider the root system of type $A_{1}$.
Let $V_{R}$ be $\mathbb{R}^{1}$ with the standard inner product.
We have $R=\{\pm e_{1}\}$, $R^{\vee}=\{\pm2e_{1}\}$. The simple
coroot is $\alpha_{1}^{\vee}=2e_{1}$ and the simple coweight is $\beta_{1}=e_{1}$.
We have $Q=\left\{ 2ze_{1}:z\in\mathbb{Z}\right\} $, $P=\left\{ ze_{1}:z\in\mathbb{Z}\right\} $
and $\hat{\Omega}=P/Q\cong\{Id,\omega\}$. The Coxeter group is $W=\left\langle s_{0},s_{1}:s_{0}^{2}=s_{1}^{2}=1\right\rangle $
and the extended Coxeter group is $\hat{W}=W\rtimes\hat{\Omega}=\left\langle s_{0},s_{1},\omega:s_{0}^{2}=s_{1}^{2}=\omega^{2}=1,\,\omega s_{0}=s_{1}\omega\right\rangle $.
We have as elements of $\hat{W}$, $\beta_{1}=s_{0}\omega$. We have
$\hat{A}_{0}=\hat{\Omega}=\left\{ Id,\omega\right\} $. Each element
of $w\in\hat{W}$ can be written uniquely as $w=s_{1}^{\delta_{1}}\beta_{1}^{m}\omega^{\delta_{\omega}}$
for $\delta_{\omega},\delta_{1}\in\{0,1\}$ and $m\ge0$. 

There is a single abstract parameter $u$ is the parameter system,
and the generalized Poincare series of the Iwahori-Hecke algebra is
\[
P_{\hat{W}}(u)=(1+h_{s_{1}}u)\frac{1}{1-h_{\beta_{1}}u}(1+h_{\omega})
\]

This case corresponds to the Iwahori-Hecke algebra of the regular
graph, as described in \cite{kamber2016lpgraph}, section 7. As explained
there (with slightly different notations), the operator $h_{\beta_{1}}$
is Hashimoto's non backtracking operator, used to define the graph
Zeta function.
\end{example}
\begin{example}
\label{exa:Bipartite graph}Consider the non-reduced root system of
type $BC_{1}$. Let again $V_{R}$ be $\mathbb{R}^{1}$ with the standard
inner product. We have $R=\{\pm e_{1},\pm2e_{1}\}$, $R^{\vee}=\{\pm e_{1},\pm2e_{1}\}$.
The simple coroot is $\alpha_{1}^{\vee}=e_{1}$ and the simple coweight
is $\beta_{1}=\alpha_{1}^{\vee}=e_{1}$. we have $P=Q=\left\{ ze_{1}:z\in\mathbb{Z}\right\} $
and $\hat{\Omega}=\{1\}$. The Coxeter group is $W=\left\langle s_{0},s_{1}:s_{0}^{2}=s_{1}^{2}=1\right\rangle $
and the extended Coxeter group is $\hat{W}=W$. We have as elements
of $\hat{W}$, $\beta_{1}=s_{0}s_{1}$. We have $\hat{A}_{0}=\left\{ Id,s_{0}\right\} $.
Each element of $w\in\hat{W}$ can be written uniquely as $w=s_{1}^{\delta_{1}}\beta_{1}^{m}s_{0}^{\delta_{0}}$
for $\delta_{0},\delta_{1}\in\{0,1\}$ and $m\ge0$. 

There are two abstract parameters $u_{0},u_{1}$ is the parameter
system, and the generalized Poincare series of the Iwahori-Hecke algebra
is 
\[
P_{\hat{W}}(u_{1},u_{2})=(1+h_{s_{1}}u_{1})\frac{1}{1-h_{s_{0}s_{1}}u_{1}u_{2}}(1+h_{s_{0}}u_{0})
\]

This case corresponds to the Iwahori-Hecke algebra of the bipartite
biregular graph, as described in \cite{kamber2016lpgraph}, section
11. The operator $h_{\beta_{1}}=h_{s_{0}s_{1}}$ is once again Hashimoto's
non backtracking operator in the bipartite case. See also the discussion
in \cite{hoffman2003remarks}.
\end{example}
\begin{example}
Let us describe the general $A_{n}$ case, let $V_{0}=R^{n+1}$ with
the standard inner product and $V_{R}=\left\{ v\in V_{0},\sum v_{i}=0\right\} $.
The set of roots (or coroots, which are equal) is $R=R^{\vee}=\left\{ e_{i}-e_{j}:0\le i\ne j\le n\right\} $
and the set of simple roots (and simple coroots) are $\alpha_{i}=\alpha_{i}^{\vee}=e_{i-1}-e_{i}$,
$i=1,...,n$ (they indeed span the subspace $V_{R}\subsetneq V$).
The coroot lattice is $Q=\left\{ (z_{0},...,z_{n})\in\mathbb{Z}^{n+1}:\sum z_{i}=0\right\} $.

The simple coweights are $\beta_{i}=e_{0}+...+e_{i}-\frac{i}{n+1}(1,...,1)$,
$i=1,...,n$. The coweight lattice is 
\[
P=\left\{ (z_{0},...,z_{n})-\frac{\sum z_{i}}{n+1}(1,...,1):(z_{0},...,z_{n})\in\mathbb{Z}^{n+1}\right\} =\left\{ \sum x_{i}\beta_{i}:x_{i}\in\mathbb{Z}\right\} 
\]
The dominant coweights are 
\[
P^{+}=\left\{ (z_{0},...,z_{n})-\frac{\sum z_{i}}{n+1}(1,...,1):(z_{0},...,z_{n})\in\mathbb{Z}^{n+1},\,z_{i}\ge z_{i+1}\right\} =\left\{ \sum x_{i}\beta_{i}:x_{i}\in\mathbb{N}\right\} 
\]

$W_{0}\cong S_{n+1}$ acts by permutations of the coordinates of $V$
(or $V_{R})$. We have $W=Q\rtimes S_{n+1}$, $\hat{W}=P\rtimes S_{n+1}$
with the obvious action on $V_{R}$.

The Coxeter generators are the transpositions $s_{i}=id\times(i-1,i)\in Q\rtimes S_{n+1}\subset W$
for $i=1,...,n$ and $s_{0}=(-1,0,...,0,1)\times(0,n)\in Q\rtimes S_{n+1}=W$
(the left multiplier is an element of $\mathbb{Z}^{n+1}$, the right
multiplier is a transposition in $S_{n+1}$). 

In this case every vertex type is good and special. The group $\hat{\Omega}$
is isomorphic to $\mathbb{Z}/n\mathbb{Z}$ and its elements are $\omega_{i}:S\rightarrow S$,
$\omega_{i}(s_{j})=s_{i+j\,mod\,n}$. If $n\ge2$, $\hat{\Omega}$
is a proper subgroup of index 2 of the full automorphism group $\Omega$
of $S$ that is isomorphic to the dihedral group $\mathbb{Z}/2\mathbb{Z}\ltimes\mathbb{Z}/n\mathbb{Z}$
and also contains the elements $\tau_{j}(s_{i})=s_{j-i\,mod\,n}$.
\end{example}

\section{\label{sec:Temperedness-in-the}Temperedness in the Affine case}

In this section we study different conditions for temperedness in
affine Coxeter groups.
\begin{defn}
The exponential growth rate of $W$ is $\mbox{limsup}_{m\rightarrow\infty}\#\{w\in W:l(w)=m\}^{1/m}$.
\end{defn}
If $W$ if affine is has a slow growth rate:
\begin{lem}
If $W$ is affine irreducible of dimension $n$, the number of $w\in W$
with $l(w)\le m$ is bounded by $G(m)=\left|W_{0}\right|^{2}\left(m+1\right)^{n}$.
Therefore $W$ has exponential growth rate $1$. 
\end{lem}
\begin{proof}
Using theorem \ref{thm:Structure Theorem}, all $w\in W$ with $w=w_{0}\prod_{i=1}^{n}\beta_{i}^{m_{i}}a$,
$l(w)\le m$, satisfy $m_{i}\le m$. There are at most $\left|W_{0}\right|^{2}\left(m+1\right)^{n}$
such $w$.
\end{proof}
A spherical Coxeter group has exponential growth rate $0$ and an
infinite Coxeter groups has growth rate $1$ if and only if it is
a direct product of an affine Coxeter group and a spherical Coxeter
group. See \cite{terragni2013growth} for more about this.
\begin{example}
\label{exa:Trivial repres}By Example \ref{exa:Trivial repres}, the
trivial representation is generated by a function $f\in\mathbb{C}^{B_{\phi}}$
having a constant value $1$ on every chamber. Such a function is
of course spherical around every chamber $C_{0}$. Since $f\in L_{\infty}(B_{\phi})$,
the trivial representation is $\infty$-tempered.

The trivial representation is $p$-tempered, $p<\infty$, if and only
if the series $\sum q_{w}(1-\delta)^{pl(w)}$ converges for every
$\delta>0$. Assume the building is thick, i.e. $q_{s}>1$ for every
$s\in S$, and $W$ infinite. Then $q_{w}>(1+\epsilon)^{l(w)}$ for
every $w\in W$ for some fixed $\epsilon>0$. Therefore the trivial
representation is not $p$-tempered for any $p<\infty$.

If the building is thin we have $q_{w}=1$ for any $w\in W$. The
trivial representation in this case is $p$-tempered, $p\ge1$ if
and only if the exponential growth rate is $\le1$. In any case it
is never $p$-finite. 
\end{example}
\begin{example}
By the proof of proposition \ref{claim:Steinberg Definition}, the
Steinberg representation is generated by a function $f\in\mathbb{C}^{B_{\phi}}$,
spherical around a fixed chamber $C_{0}$, with values 
\[
f(C)=(-1)^{l(d(C_{0},C))}/q_{d(C_{0},C)}
\]

In this case $f\in L_{p}(B_{\phi})$ if and only if 
\[
\sum_{C}\left|f(C)\right|^{p}=\sum_{w\in W}q_{w}(1/q_{w})^{p}=\sum_{w}q_{w}^{1-p}<\infty
\]

$f\in T_{p}(B_{\phi})$ if and only if for every $0<\delta<1$ 
\[
\sum_{w}q_{w}^{1-p}(1-\delta)^{p\cdot l(w)}<\infty
\]

Assume that $W$ is affine. If the building is thin $q_{w}=1$ for
every $w\in W$ and therefore $f\not\in L_{p}(B_{\phi})$ for every
$p<\infty$. However, using the previous lemma $f\in T_{1}(B_{\phi}$).
If the building is thick $\alpha_{1}^{l(w)}\le q_{w}\le\alpha_{2}^{l(w)}$
for some $\alpha_{1},\alpha_{2}>1$. Using the previous lemma, $f\in L_{p}(B_{\phi})$
for every $p>1$, $f\in T_{1}(B_{\phi})$ and $f\notin L_{1}(B_{\phi})$.
Therefore the Steinberg representation is always $1$-tempered.
\end{example}
Using the growth rate we can give a nicer equivalent definitions of
$p$-temperedness. First, we state an easy lemma. The proof is elementary
and is omitted.
\begin{lem}
\label{lem:Temperdness Lemma}Let $g:\mathbb{N}\rightarrow\mathbb{R}_{\ge0}$
be a series. Then the following conditions are equivalent:

1. For every $0<\delta<1$, $\sum_{l}g(l)(1-\delta)^{l}<\infty$.

2. $\limsup_{l}g(l)^{1/l}\le1$.

3. For every $\delta>0$, for almost every $l$, $g(l)\le(1+\delta)^{l}$.

Moreover, the conditions hold for the absolute value of any polynomial
and if the conditions hold for $g_{1}$ and $g_{2}$, then they hold
for $g_{1}\cdot g_{2},\,g_{1}^{\gamma}$($0<\gamma\in R)$.
\end{lem}
We can now state the equivalent conditions. We state them for the
Iwahori-Hecke algebra $H_{\phi}$ and as usual similar conditions
can be stated for $H$ itself.
\begin{prop}
\label{prop:Equivalent temperdness conditions}Assume that $\hat{W}$
is affine and $f\in\mathbb{C}^{B_{\phi}\times\hat{\Omega}}$ is spherical
around $C_{0}$. We may assign to $f$ a function $f_{\hat{W}}\in\mathbb{C}^{\hat{W}}$
defined by $f_{\hat{W}}(w)=(h_{w}f)(C_{0})$ for $w\in\hat{W}$.

The following are equivalent:

1. $f$ is p-tempered, i.e for every $0<\delta<1$, $\sum_{w}\left|f_{\hat{W}}(w)\right|^{p}q_{w}^{1-p}(1-\delta)^{l(w)}<\infty$.

2. $\limsup_{w}\left(q_{w}^{1-p}\left|f_{\hat{W}}(w)\right|^{p}\right)^{1/l(w)}\le1$.

3. For every $\delta>0$, for almost every $w\in\hat{W}$, $\left|f_{\hat{W}}(w)\right|<q_{w}^{(p-1)/p}(1+\delta)^{l(w)}$.

4. Assuming $B$ is thick: $\sum_{w}\left|f_{\hat{W}}(w)\right|q_{w}^{s}$
converges for every $s<(1-p)/p$.

5. For every parameter system $\overrightarrow{u}=(u_{i})_{i\in S}\in\mathbb{R}_{>0}^{S}$
satisfying $u_{s}<q_{s}^{(1-p)/p}$ for every $s\in S$, the series
$\sum_{w}\left|f_{\hat{W}}(w)\right|u_{w}$ converges.
\end{prop}
\begin{proof}
Let $g(l)=\sup_{w\in\hat{W},l(w)=l}\left|f_{\hat{W}}(w)\right|^{p}q_{w}^{1-p}$.
Since $\hat{W}$ is affine there exists a polynomial $P(l)$ such
that $g(l)\le\sum_{w:l(w)=l}\left|f_{\hat{W}}(l)\right|^{p}q_{w}^{1-p}\le P(l)g(l)$.
Now all the conditions in the proposition are equivalent to the fact
that $g$ satisfies the previous lemma. We leave the verification
to the reader.
\end{proof}
\begin{cor}
A $\hat{H}_{\phi}$ representation $V$ is $p$-tempered if and only
if the conditions of the previous proposition hold for $f_{\hat{W}}(w)=c_{v^{*},v}(w)=\left\langle v^{*},h_{w}v\right\rangle $
for every $v^{*}\in V^{*},\,v\in V$.
\end{cor}
\begin{cor}
A finite dimensional $\hat{H}_{\phi}$ representation $V$ is $p$-tempered
if and only if the generalized Poincare series $P_{\hat{W}}(\vec{u})=\sum_{w\in\hat{W}}h_{w}^{V}u_{w}$
absolutely converges (as a series of matrices) for every parameter
system $\vec{u}=\left(u_{s}\right)_{s\in S}$ satisfying $u_{s}<q_{s}^{(1-p)/p}$
for every $s\in S$.
\end{cor}
\begin{rem}
In the case of representations of dimension 1 a very simple case of
this corollary was used in by Borel in \cite{borel1976admissible}
to identify the one dimensional square integrable representations
of affine Iwahori-Hecke algebras. 
\end{rem}
\begin{defn}
\label{def:lambda V definition}Let $V$ be a finite dimensional representation
of $H$. Let $h\in H$. Define $\lambda_{V}(h)$ as the largest absolute
value of an eigenvalue of $h$ on $V$.
\end{defn}
\begin{prop}
\label{Temperedness_Affine_Case_eigenvalues}Assume $W$ is affine
and irreducible. Let $V$ be a representation of $H_{\phi}$. The
following are equivalent:

1. $V$ is $p$-tempered.

2. $\lambda_{V}(h_{\alpha})\le q_{\alpha}^{(p-1)/p}$ for every $\alpha\in P$.

3. $\lambda_{V}(h_{\alpha})\le q_{\alpha}^{(p-1)/p}$ for every $\alpha\in P^{+}$.

4. $\lambda_{V}(h_{\beta_{i}})\le q_{\beta_{i}}^{(p-1)/p}$ for $i=1,...,n$.
\end{prop}
\begin{proof}
The fact that $(2)\Rightarrow(3)\Rightarrow(4)$ is obvious. For $(1)\Rightarrow(2)$
notice that for every $\alpha\in P$, $l(\alpha^{m})=m\cdot l(\alpha)$
and therefore $h_{\alpha^{m}}=h_{\alpha}^{m}$ and $q_{\alpha^{m}}=q_{\alpha}^{m}$.
Assume $V$ is $p$-tempered. By the limsup condition of proposition
\ref{prop:Equivalent temperdness conditions} 
\[
\mbox{limsup}_{m}\left(\left|q_{\alpha^{m}}^{1-p}\left\langle v^{*},h_{\alpha^{m}}v\right\rangle \right|^{p}\right)^{1/l(\alpha^{m})}=\mbox{limsup}_{m}\left(q_{\alpha}^{1-p}\left|\left\langle v^{*},h_{\alpha}^{m}v\right\rangle \right|^{p/m}\right)^{1/l(\alpha)}\le1
\]

By lemma \ref{lem:Temperdness Lemma} we may change the exponent and
get
\[
\mbox{limsup}_{m}\left|\left\langle v^{*},h_{\alpha}^{m}v\right\rangle \right|^{1/m}\le q_{\alpha}^{(p-1)/p}
\]

Choose for $v$ an eigenvector for an eigenvalue $\lambda$ of $h$
with $\lambda_{V}(h_{\alpha})=\left|\lambda\right|$ , we get $\lambda_{V}(h_{\alpha})\le q_{\alpha}^{(p-1)/p}$
by the matrix equality stated above.

Assume $(4)$ that holds. Recall the matrix equality $\limsup\left\Vert A^{m}\right\Vert ^{1/m}=\lambda_{max}(A)$
where $A\in M_{n}(\mathbb{C})$, $\left\Vert \right\Vert $ is any
matrix norm and $\lambda_{max}(A)$ is the largest absolute value
of an eigenvalue of $A$. Applying this equality, since $\lambda_{V}(h_{\beta_{i}})\le q_{\beta_{i}}^{(p-1)/p}$
we know that for any $u^{*}\in V^{*},u\in V$ we have 
\[
\limsup_{m_{i}}\left(\left(q_{\beta_{i}}^{m_{i}}\right)^{1-p}\left|\left\langle u^{*},h_{\beta_{i}}^{m_{i}}u\right\rangle \right|^{p}\right)^{1/m_{i}}\le1
\]

Applying all the operators together, we deduce that for any $u^{*}\in V^{*},u\in V$
we have:

\[
\limsup_{m_{i}}\left(\left(\prod_{i=1}^{n}q_{\beta_{i}}^{m_{i}}\right)^{1-p}\left|\left\langle u^{*},\prod_{i=1}^{n}h_{\beta_{i}}^{m_{i}}u\right\rangle \right|^{p}\right)^{1/\sum m_{i}}\le1
\]

Now using \ref{thm:Structure Theorem}

\begin{eqnarray*}
 & \limsup_{w}\left(q_{w}^{1-p}\left|\left\langle v^{*},h_{w}v\right\rangle \right|^{p}\right)^{1/l(w)} & =\\
 & \sup_{a\in\hat{A}_{0},w_{0}\in W_{0}}\limsup_{m_{i},i=1,..,n}\left(\left(q_{w_{0}}\left(\prod_{i=1}^{n}q_{\beta_{i}}^{m_{i}}\right)q_{a}\right)^{1-p}\left|\left\langle v^{*},h_{w_{0}}\prod_{i=1}^{n}h_{\beta_{i}}^{m_{i}}h_{a}v\right\rangle \right|^{p}\right)^{1/(l(w_{0})+\sum m_{i}l(\beta_{i})+l(a))} & \le\\
 & \sup_{a\in\hat{A}_{0},w_{0}\in W_{0}}\limsup_{m_{i},i=1,..,n}\left(\left(\prod_{i=1}^{n}q_{\beta_{i}}^{m_{i}}\right)^{1-p}\left|\left\langle h_{w_{0}}^{*}v^{*},\prod_{i=1}^{n}h_{\beta_{i}}^{m_{i}}h_{a}v\right\rangle \right|^{p}\right)^{1/\sum m_{i}} & \le1
\end{eqnarray*}

and by proposition \ref{prop:Equivalent temperdness conditions},
$V$ is $p$-tempered.
\end{proof}
\begin{defn}
\label{def:Generalized Zeta}In the equal parameter case ($q_{s}=q$)
define 
\[
\zeta_{1,V}(u)=\frac{1}{\det(1-h_{\beta_{1}}u^{l(\beta_{1})})\cdot...\cdot\det(1-h_{\beta_{n}}u^{l(\beta_{n})})}
\]

In general define 
\[
\zeta_{2,V}(s)=\frac{1}{\det(1-h_{\beta_{1}}q_{\beta_{1}}^{s})\cdot...\cdot\det(1-h_{\beta_{n}}q_{\beta_{n}}^{s})}
\]
\end{defn}
Notice that the above definition is closely related to corollary \ref{cor:Gen_Poincare_Series}.
Proposition \ref{prop:Equivalent temperdness conditions} is equivalent
to:
\begin{cor}
Assume $W$ is affine and irreducible. Let $V$ be a representation
of $H_{\phi}$. The following are equivalent:

1. $V$ is $p$-tempered.

2. The poles of $\zeta_{2,V}(s)$ are all for $Re(s)\ge(1-p)/p$.

3. Assume the equal parameter case, the poles of $\zeta_{u,V}(s)$
are all for $\left|u\right|\ge q^{(1-p)/p}$.
\end{cor}
By Examples \ref{exa:Regular graph}, \ref{exa:Bipartite graph} and
the results of \cite{kamber2016lpgraph}, all the discussion above
is a direct generalization of the zeta functions of graphs, and its
connection to temperedness.

\section{\label{sec:Bounds-on-Hecke-operators}Bounds on Hecke Operators}

In theorem \ref{thm:Bound on H_beta} we will prove that:
\begin{thm*}
Let $p\ge2$. Let $q_{max}=\max_{s\in S}\left\{ q_{s}\right\} $ and
let $\tilde{w}_{0}$ be the longest element of $W_{0}$. The norm
of the operator $h_{\beta}$, $\beta\in P^{+}\subset\hat{W}$ is bounded
on $L_{p}(\hat{B}_{\phi})$ by $\left|W_{0}\right|\left|2q_{max}\right|^{l(\tilde{w}_{0})}\left(l(\beta)+1\right)^{l(\tilde{w}_{0})}q_{\beta}^{(p-1)/p}$
.
\end{thm*}
\begin{cor}
\label{cor:Bound on H_w}The norm of $h_{w}$, $w\in\hat{W}$ is bounded
on $L_{p}(\hat{B}_{\phi})$ by 
\[
\left\Vert h_{w}\right\Vert _{p}\le D(q_{max},l(w))q_{w}^{(p-1)/p}
\]

with $D(q_{max},l)=\left|W_{0}\right|2^{l(\tilde{w}_{0})}q_{max}^{4\cdot l(\tilde{w}_{0})}\left(l+1+l(\tilde{w}_{0})\right)^{l(\tilde{w}_{0})}$.

The same bound holds for any finite dimensional unitary $p$-tempered
$H_{\phi}$-representation $V$.
\end{cor}
\begin{proof}
Recall the decomposition $\hat{W}=P\rtimes W_{0}$. In addition, each
element of $P$ is conjugate by an element of $W_{0}$ to an element
of $P^{+}$, which is of the same length (see lemma \ref{lem:Lengthes and dominance}).
Therefore we can write for every $w\in\hat{W}$, $w=w_{0}\beta w_{0}^{\prime}$
with $w_{0},w_{0}^{\prime}\in W_{0}$, and $l(\beta)\le l(w)+l(\tilde{w}_{0})$,
$q_{\beta}\le q_{w}q_{max}^{l(\tilde{w}_{0})}$. 

If $w_{0}=s_{i_{0}}\cdot...\cdot s_{i_{l}}$, $w_{0}^{\prime}=s_{i_{0}^{\prime}}\cdot...\cdot s_{i_{k}^{\prime}}$,
then by the description of the Iwahori-Hecke algebra 
\[
h_{w}=h_{s_{i_{0}}}^{\epsilon_{i_{0}}}\cdot...h_{s_{i_{l}}}^{\epsilon_{i_{l}}}h_{\beta}h_{s_{i_{0}^{\prime}}}^{\epsilon_{i_{0}}^{\prime}}\cdot...h_{s_{i_{k}^{\prime}}}^{\epsilon_{i_{k}}^{\prime}}
\]
 Where $\epsilon_{i},\epsilon_{i}^{\prime}\in\{\pm1\}$. Since $\left\Vert h_{s}\right\Vert _{p}\le q_{s}\le q_{max},\,\left\Vert h_{s}^{-1}\right\Vert _{p}\le1\le q_{max}$,
we have 
\begin{align*}
\left\Vert h_{w}\right\Vert _{p} & \le q_{max}^{l}\left\Vert h_{\beta}\right\Vert _{p}q_{max}^{k}\le q_{max}^{l(\tilde{w}_{0})}\left|W_{0}\right|2^{l(\tilde{w}_{0})}q_{max}^{l(\tilde{w}_{0})}\left(l(\beta)+1\right)^{l(\tilde{w}_{0})}q_{\beta}^{(p-1)/p}q_{max}^{l(\tilde{w}_{0})}\\
 & \le\left|W_{0}\right|2^{l(\tilde{w}_{0})}q_{max}^{4\cdot l(\tilde{w}_{0})}\left(l(w)+l(\tilde{w}_{0})+1\right)^{l(\tilde{w}_{0})}q_{w}^{(p-1)/p}
\end{align*}

For the second claim, denote $F=D(q,l(w))q_{w}^{(p-1)/p}$. Notice
that $h_{w}^{\ast}=h_{w^{-1}}$ and therefore the $L_{p}$-norm of
$h_{w}h_{w}^{\ast}$ is bounded by $D(q,l(w))q_{w}^{(p-1)/p}D(q,l(w^{-1})+1)q_{w^{-1}}^{(p-1)/p}=F^{2}$.
Therefore the spectrum of $h_{w}h_{w}^{\ast}$ on $L_{p}(\hat{B}_{\phi})$
is bounded by $F^{2}$ and by corollary \ref{cor:Main Theorem proof},
$F^{2}$ is also a bound of the eigenvalues of any $p$-tempered finite
dimensional $H_{\phi}$-representation $V$. If $V$ is also unitary,
$h_{w}h_{w}^{\ast}$ is self adjoint and its norm is bounded by its
largest eigenvalue. Finally, $\left\Vert h_{w}\right\Vert _{V}=\sqrt{\left\Vert h_{w}h_{w}^{\ast}\right\Vert _{V}}\le F$.
\end{proof}
\begin{rem}
The $p>2$ of both theorem \ref{thm:Bound on H_beta} and corollary
\ref{cor:Bound on H_w} can be deduced (and actually slightly improved)
from the case $p=2$ and the trivial case $p=\infty$, by the Riesz-Thorin
interpolation theorem. However, even better $L_{p}$-bounds can be
deduced by better analysis in theorem \ref{thm:Bound on H_beta}.
\end{rem}
\begin{cor}
The spectrum of $h_{\beta}$, $\beta\in Q^{+}\subset\hat{W}$ on $L_{p}(\hat{B}_{\phi})$
is bounded in absolute value by $q_{\beta}^{(p-1)/p}$. 
\end{cor}
\begin{proof}
By Gelfand's formula,
\begin{align*}
\lambda_{L_{p}(\hat{B}_{\phi})}(h_{\beta}) & =\limsup\left\Vert h_{\beta}^{n}\right\Vert _{L_{p}(\hat{B}_{\phi})}^{1/n}=\limsup\left\Vert h_{n\beta}\right\Vert _{L_{p}(\hat{B}_{\phi})}^{1/n}\le\limsup\left(D\left(q,l(n\beta)\right)\cdot q_{\beta}^{n(p-1)/p}\right)^{1/n}=\\
 & =q_{\beta}^{(p-1)/p}\limsup D\left(q,n\cdot l(\beta)\right)^{1/n}=q_{\beta}^{(p-1)/p}
\end{align*}

The last equality holds by the limit $n^{1/n}\rightarrow_{n\rightarrow\infty}1$.
\end{proof}
\begin{cor}
If a finite dimensional $H$ (respectively $\hat{H}_{\phi},H_{\phi}$)
representation is weakly contained in $L_{p}(B_{f})$ (respectively
$L_{p}(\hat{B}_{\phi}),L_{p}(B_{\phi})$), then it is $p$-tempered.
\end{cor}
\begin{proof}
Follows by the last corollary and theorem \ref{cor:Main Theorem proof}.
\end{proof}
Another results of the above is a tight version of the Kunze-Stein
theorem. Define a twisted $p$-norm on $\hat{H}_{\phi}$ by $\left\Vert \sum\alpha_{w}h_{w}\right\Vert _{p}^{\prime}=\sum_{w}D(q,l(w))q_{w}^{1/p}\left|\alpha_{w}\right|$
(since this sum is finite it is always well defined). Let $\hat{H}_{\phi,p}$
be the completion of $\hat{H}_{\phi}$ with respects to this norm,
i.e. 
\[
\hat{H}_{\phi,p}=\left\{ \sum\alpha_{w}h_{w}:\sum_{w}D(q,l(w))^{-1}q_{w}^{1/p}\left|\alpha_{w}\right|<\infty\right\} 
\]
where the sums can be infinite.
\begin{cor}
There exists a bounded action of $\hat{H}_{\phi,(p-1)/p}$ on $L_{p}(\hat{B}_{\phi})$.
The norm of $h\in\hat{H}_{\phi,(p-1)/p}$ on $L_{p}(\hat{B}_{\phi})$
is bounded by $\left\Vert h\right\Vert _{p}^{\prime}$.
\end{cor}
If terms of the group $G$, this can be stated as follows- we have
an isomorphism $H_{G_{\phi}}(G)\cong H_{\phi}$. Let $L_{G_{\phi},p}(G)$
be the completion of $H_{G_{\phi}}(G)\subset H(G)\subset C_{c}(G)$
with respect to the usual $p$ norm and let $L_{G_{\phi},p}^{\prime}(G)$
be the completion of $H_{G_{\phi}}(G)$ with respect to the twisted
$p$-norm as above. Notice that both $L_{G_{\phi},p}(G)$ and $L_{G_{\phi},p}^{\prime}(G)$
are subspaces of $C(G_{I}\backslash G)$ and that $L_{G_{\phi},p}^{\prime}(G)\subset L_{G_{\phi},p'}^{\prime}$
if $p'<p$. Then the results above say that convolution is a bounded
bilinear operator $L_{G_{\phi},(p-1)/p}^{\prime}(G)\times L_{G_{\phi},p}(G)\to L_{G_{\phi},p}(G)$.
This is a strong version of the Kunze-Stein theorem for Iwahori-fixed
vectors.

\section{\label{sec:Application:-Average-Distance}Application: Average Distance
and Diameter}

We define a distance between chambers of the quotient $X$ as follows,
for $C,C'\in X_{\phi}$ let $l(C,C')$ be the length of the shortest
gallery between them. Equivalently, $l(C,C')=\min l(d(\tilde{C},\tilde{C}'))$,
where $\tilde{C},\tilde{C'}\in B_{\phi}$ cover $C,C'$.
\begin{thm}
Let $X$ be an $L_{p}$-expander of irreducible affine Coxeter group
$W$, with single parameter $q$, having $N$ chambers and $C_{0}\in X_{\phi}$.
Let $n$ be the dimension of $X$ and $\tilde{w}_{0}$ is the longest
element of the spherical Coxeter group $W_{0}$. Then all but $o(N)$
other chambers $C\in X_{\phi}$ are of gallery distance $l(C_{0},C)$
which satisfies 
\[
l(C_{0},C)\le\left\lceil \frac{p}{2}\log_{q}N+\left(l(\tilde{w}_{0})+1\right)\log_{q}\log_{q}N\right\rceil 
\]

and 
\[
l(C_{0},C)\ge\left\lfloor \log_{q}N-\left(n+1\right)\log_{q}\log_{q}N\right\rfloor 
\]

In addition, the diameter of $X$ for $N$ large enough is at most
$\left\lceil p\log_{q}N+2\left(l(\tilde{w}_{0})+1\right)\log_{q}\log_{q}N\right\rceil $.
\end{thm}
\begin{proof}
Let $w\in W$ and consider $q_{w}^{-1}h_{w}1_{C_{0}}$. Every chamber
$C$ for which $h_{w}q_{w}^{-1}1_{C_{0}}(C)\ne0$ is at a distance
at most $l(w)$ from $C_{0}$.

Let $\pi\in\mathbb{C}^{X_{\phi}}$ be the constant function $\pi(C)=1/N$
. We have $\left\Vert 1_{C_{0}}-\pi\right\Vert _{2}^{2}=\left(1-1/n\right)^{2}+(n-1)n^{-2}=1-1/n<1$.

Let $l=\left\lceil \frac{p}{2}\log_{q}N+K\log_{q}\log_{q}N\right\rceil $
and let $w\in W$ with $l(w)=l$. Since $h_{w}q_{w}^{-1}\pi=\pi$
and $h_{w}q_{w}^{-1}1_{C_{0}}-\pi\in L_{2}^{0}(X_{\phi})$, by corollary
\ref{cor:Bound on H_w} we have 
\[
\left\Vert h_{w}q_{w}^{-1}1_{C_{0}}-\pi\right\Vert _{2}\le D(q,l)q^{l(p-1)/p}q^{-l}\le D(q,l)N^{-1/2}\log_{q}N^{-K}
\]
.

Since by Cauchy-Schwartz $\left\Vert f\right\Vert _{1}\le N^{1/2}\left\Vert f\right\Vert _{2}$
for every $f\in\mathbb{C}^{X_{\phi}}$. Therefore $\left\Vert h_{w}q_{w}^{-1}1_{C_{0}}-\pi\right\Vert _{1}\le D(q,l)\log N^{-K}$.
As a result, $h_{w}q_{w}^{-1}1_{C_{0}}$ is $0$ on at most 
\[
ND(q,l)\log N^{-K}=N\left|W_{0}\right|2^{l(\tilde{w}_{0})}q^{4\cdot l(\tilde{w}_{0})}\left(\frac{p}{2}\log_{q}N+K\log_{q}\log_{q}N+1+l(\tilde{w}_{0})\right)^{l(\tilde{w}_{0})}\log_{q}N^{-K}
\]
 chambers. For $K>l(\tilde{w_{0})}$ this is $o(N)$. 

For the lower bound, there are at most $G(l)=\left|W_{0}\right|^{2}\left(l+1\right)^{n}$
elements $w\in W$ with $l(w)\le l$. Therefore there are at most
$G(l)q^{l}$ chambers $C\in X_{\phi}$ of with $l(C_{0},C)\le l$.
For $l\le\left\lfloor \log_{q}N-\left(n+1\right)\log_{q}\log_{q}N-1\right\rfloor $,
we have $G(l)q^{l}=o(N)$.

The claim about the diameter follows from the upper bound on the average
distance. Let $C_{0},C_{1}\in X_{\phi}$ . Let 
\[
l=\left\lceil \frac{p}{2}\log_{q}N+\left(l(\tilde{w}_{0})+1\right)\log_{q}\log_{q}N\right\rceil 
\]

For $N$ large enough so that 
\[
\left|W_{0}\right|2^{l(\tilde{w}_{0})}q^{4\cdot l(\tilde{w}_{0})}\left(\frac{p}{2}\log_{q}N+K\log_{q}\log_{q}N+1+l(\tilde{w}_{0})\right)^{l(\tilde{w}_{0})}\log_{q}N^{-K}<0.5
\]

There are more than $0.5N$ elements $C_{2}\in X_{\phi}$ with $l(C_{0},C_{2})\le l$
and more than $0.5N$ elements $C_{3}\in X_{\phi}$ with $l(C_{1},C_{3})\le l$.
By pigeonhole there is some $C_{2}=C_{3}$, so $l(C_{0},C_{1})\le2l$,
as required.
\end{proof}
\begin{rem}
The proof actually shows that for every $w\in W$, $l(w)>\frac{p}{2}\log_{q}N+\left(l(w_{0})+1\right)\log_{q}\log_{q}N$,
for almost every two chambers there is a gallery of type $w$ connecting
them.
\end{rem}

\section{\label{sec:Retraction-into-Apartments}The Bernstein Presentation}

As an introduction to this section, let us recall the most important
construction in affine Hecke algebras. Recall that any coweight $\beta\in P$
can be written as a difference of two dominant coweights- $\beta=\beta_{1}-\beta_{2}$,
with $\beta_{1},\text{\ensuremath{\beta}}_{2}\in P^{+}$. 
\begin{defn}
For $\beta\in P$, $\beta=\beta_{1}-\beta_{2}$, $\beta_{1},\text{\ensuremath{\beta}}_{2}\in P^{+}$,
we will denote by $Y_{\beta}\in\hat{H}_{\phi}$ the element $Y_{\beta}=q_{\beta_{1}}^{-1/2}q_{\beta_{2}}^{1/2}h_{\beta_{1}}h_{\beta_{2}}^{-1}$. 
\end{defn}
It is immediate to verify that $Y_{\beta}$ does not depend on the
choice of $\beta_{1},\text{\ensuremath{\beta}}_{2}\in P^{+}$, and
that $Y_{\beta}Y_{\beta'}=Y_{\beta+\beta^{\prime}}$ for $\beta,\beta^{\prime}\in P$. 

The following theorem is called the \emph{Bernstein-Luzstig presentation}
of the Hecke algebra (see \cite{parkinson2006buildings}, theorem
6.6 and references therein):
\begin{thm}
(Bernstein) The operators $Y_{\beta}h_{w_{0}}$, $w_{0}\in W_{0}$,
$\beta\in P$, are a basis for the extended Iwahori-Hecke algebra
$\hat{H}_{\phi}$. Multiplication in the algebra with respect to this
basis is given by the Iwahori-Hecke relations for $W_{0}$, $Y_{\beta}Y_{\beta'}=Y_{\beta+\beta^{\prime}}$
for $\beta,\beta^{\prime}\in P$ and the relations:

\[
\begin{array}{cccc}
Y_{\beta}h_{s_{i}} & = & h_{s_{i}}Y_{s_{i}(\beta)}+\left(q_{s_{i}}-1\right)\frac{Y_{\beta}-Y_{s_{i}(\beta)}}{1-Y_{-\alpha_{i}^{\vee}}} & \left(R,i\right)\ne\left(BC_{n},n\right)\\
Y_{\beta}h_{s_{n}} & = & h_{s_{n}}Y_{s_{n}(\beta)}+q_{s_{n}}^{1/2}\left(q_{s_{n}}^{1/2}-q_{s_{n}}^{-1/2}+\left(q_{0}^{1/2}-q_{0}^{-1/2}\right)Y_{-(2\alpha_{n})^{\vee}}\right)\frac{Y_{\beta}-Y_{s_{n}(\beta)}}{1-Y_{-2(2\alpha_{n})^{\vee}}} & \left(R,i\right)=\left(BC_{n},n\right)
\end{array}
\]
\end{thm}
The expression $\frac{Y_{\beta}-Y_{s_{i}(\beta)}}{1-Y_{-\alpha_{i}^{\vee}}}$
is actually a compact way of writing a finite sum of $\left|\left\langle \alpha_{i},\beta\right\rangle \right|$
different $Y_{\beta'}$. Recall $s_{i}(\beta)=\beta-\left\langle \alpha,\beta\right\rangle \alpha^{\vee}$.
Now:
\[
\frac{Y_{\beta}-Y_{s_{i}(\beta)}}{1-Y_{-\alpha_{i}^{\vee}}}=\begin{cases}
0 & \left\langle \alpha_{i},\beta\right\rangle =0\\
Y_{\beta}+Y_{\beta-\alpha_{i}^{\vee}}+...+Y_{\beta-\left(\left\langle \alpha,\beta\right\rangle -1\right)\alpha^{\vee}} & \left\langle \alpha_{i},\beta\right\rangle >0\\
-Y_{\beta+\alpha_{i}^{\vee}}-Y_{\beta+2\alpha_{i}^{\vee}}...-Y_{s_{i}(\beta)} & \left\langle \alpha_{i},\beta\right\rangle <0
\end{cases}
\]

Similar relations holds for the $\left(R,i\right)=\left(BC_{n},n\right)$,
this time the sum contains $\left|\left\langle 2\alpha_{n},\beta\right\rangle \right|$
elements.

We will need for later the some estimates on the resulting presentation.
Define (See \cite{macdonald2003affine}, 2.7):
\begin{defn}
The \emph{Bruhat order} on $W_{0}$ is defined by: $w'\le w$ if there
exits a decomposition $w'=t_{0}\cdot....\cdot t_{l}$, $t_{i}\in S$
and $w=t_{i_{0}}\cdot...\cdot t_{i_{k}}$ for some $0\le i_{0}<...<i_{k}\le l$.

Define a partial order on $P^{+}$ by $\beta>\beta'$ if $\beta-\beta'\in Q^{+}$,
where $Q^{+}\subset Q$ is the set of coroots that are a non negative
sum of the simple coroots. The \emph{coweight order }on $P$ is defined
by: for $\beta\in P$ let $\beta^{+}$ be the unique element $\beta^{+}\in P^{+}$
in the $W_{0}$ orbit of $\beta$. Let $w_{0}^{\beta}\in W_{0}$ be
the shortest element sending $\beta$ to $\beta^{+}$. Then $\beta\le\beta'$
if and only if $\beta^{+}<\beta^{\prime+}$ or $\left(\beta^{+}=\beta^{\prime+}\text{ and }w_{0}^{\beta^{+}}\ge w_{0}^{\beta^{\prime+}}\right)$.
\end{defn}
\begin{rem}
In our definition we let the dominant coweight be the largest in any
$W_{0}$-orbit, which is opposite to the standard where the anti-dominant
coweight is the largest.
\end{rem}
\begin{lem}
\label{lem:Lengthes and dominance}Let $\beta$ be dominant. The set
of $\beta'\in P$ such that $\beta'\le\beta$ is closed under $\beta'\to\beta'-j\alpha_{i}^{\vee}$
for all $j$ between 0 and $\left\langle \alpha_{i},\beta'\right\rangle $
(inclusive) (i.e. a saturated set as in \cite{macdonald2003affine},
2.6).

All the coweight $\beta'\le\beta$ satisfy for simple root $\alpha_{i}\in R$,
$\left|\left\langle \alpha_{i},\beta'\right\rangle \right|\le l(\beta')\le l(\beta)$
(and $2\left|\left\langle \alpha_{i},\beta'\right\rangle \right|\le l(\beta')$
in the $\left(R,i\right)=\left(BC_{n},n\right)$ case), and $q_{\beta'}\le q_{\beta}$.
\end{lem}
\begin{proof}
Assume first that we work in the non reduced case. The first statement
is proved in \cite{macdonald2003affine}, 2.6, 2.7.

The second statement follows from the well known formula in the non-reduced
case $l(\beta')=\frac{1}{2}\sum_{\alpha\in R}\left|\left\langle \alpha,\beta'\right\rangle \right|=\sum_{\alpha\in R^{+}}\left|\left\langle \alpha,\beta'\right\rangle \right|$
(\cite{macdonald2003affine}, 2.4.1), and similarly $q_{\beta}=\prod_{\alpha\in R^{+}}q_{\alpha}^{\left|\left\langle \alpha,\beta'\right\rangle \right|}$,
from which it follows that for every $w_{0}\in W_{0}$ $l(w_{0}(\beta'))=l(\beta')$
and that $\left|\left\langle \alpha,\beta'\right\rangle \right|\le l(\beta')$. 

If $\beta'$ is dominant, $\left|\left\langle \alpha,\beta'\right\rangle \right|=\left\langle \alpha,\beta'\right\rangle $,
so $l(\beta')=\sum_{\alpha\in R^{+}}\left\langle \alpha,\beta'\right\rangle =\left\langle \rho,\beta'\right\rangle $.
where $\rho=\sum_{\alpha\in R^{+}}\alpha$. Since $\rho$ is in the
dominant sector, $\left\langle \rho,\alpha^{\vee}\right\rangle >0$
for any $\alpha\in R^{+}$. Therefore if $\beta'\le\beta$ and both
are dominant then $l(\beta')\le l(\beta)$. The claim that $q_{\beta'}\le q_{\beta}$
is proven similarly.

In the reduced case, we may consider the corresponding non-reduced
Root system $R'$ containing the non divisible roots of $R$. All
the claims follow from the claims on $R'$. 
\end{proof}
\begin{prop}
\label{prop:Explicit_bernstein_relations}For $\beta,\beta'\in P$,
$w_{0},w_{0}^{\prime}\in W_{0}$ there exist constants $\alpha_{w_{0}^{\prime},w_{0},\beta',\beta}$
(depending on the parameter system $\overrightarrow{q}$) such that:

\[
Y_{\beta}h_{w_{0}}=\sum_{w_{0}^{\prime},\beta'}\alpha_{w_{0}^{\prime},w_{0},\beta',\beta}h_{w_{0}^{\prime}}Y_{\beta'}
\]

The constants $\alpha_{w_{0}^{\prime},w_{0},\beta',\beta}$ satisfy:
\begin{enumerate}
\item $\alpha_{w_{0},w_{0},\beta,w_{0}(\beta)}=1$.
\item If $w_{0}^{\prime}\not\le w_{0}$ in the Coxeter order of the Coxet1er
group $W_{0}$, then $\alpha_{w_{0}^{\prime},w_{0},\beta',\beta}=0$.
\item If $\beta$ is dominant, and $\beta'\not\le\beta$ in the coweight
order, then $\alpha_{w_{0}^{\prime},w_{0},\beta',\beta}=0$.
\item For any $w_{0}^{\prime}\in W_{0}$, $\sum_{\beta'}\left|\alpha_{w_{0}^{\prime},w_{0},\beta',\beta}\right|\le2^{l(w_{0})}(q_{max}\cdot(l(\beta)+1))^{l(w_{0})-l(w_{0}^{\prime})}$
for $q_{max}=\max\{q_{s}:s\in S\}$.
\end{enumerate}
\end{prop}
\begin{proof}
Everything is proved by induction on $l(w)$ using the Bernstein relations.
Write $w_{0}=\hat{w}_{0}s_{i}$, $l(w_{0})=l(\hat{w}_{0})+1$ and
assume the claim is true for $\hat{w}_{0}$. Then in the non $\left(R,i\right)=\left(BC_{n},n\right)$-case:
\begin{align*}
Y_{\beta}h_{w_{0}} & =Y_{\beta}h_{\hat{w}_{0}}h_{s_{i}}=\sum_{w_{0}^{\prime},\beta'}\alpha_{w_{0}^{\prime},\hat{w}_{0},\beta',\beta}h_{w_{0}^{\prime}}Y_{\beta'}h_{s_{i}}\\
 & =\sum_{w_{0}^{\prime},\beta'}\alpha_{w_{0}^{\prime},\hat{w}_{0},\beta',\beta}h_{w_{0}^{\prime}}\left(h_{s_{i}}Y_{s_{i}(\beta')}+(q_{s_{i}}-1)\frac{Y_{\beta'}-Y_{s_{i}(\beta')}}{1-Y_{-\alpha_{i}^{\vee}}}\right)
\end{align*}

This gives by induction 1,2 and 3.

We may bound the sum of coefficients for $h_{w_{0}^{\prime}}$ on
the right hand side using the induction hypothesis. For $l(w_{0}^{\prime}s)=l(w_{0}^{\prime})+1$:

\begin{align*}
\sum_{\beta'}\left|\alpha_{w_{0}^{\prime},w_{0},\beta',\beta}\right| & \le\sum_{\beta'}\left(\left|\alpha_{w_{0}^{\prime},\hat{w}_{0},\beta',\beta}\right|\left(q_{max}-1\right)\left|\left\langle \alpha_{i},\beta^{\prime}\right\rangle \right|+\left|\alpha_{w_{0}^{\prime}s_{0},\hat{w}_{0},\beta',\beta}\right|q_{max}\right)\\
 & \le2^{l(w_{0})-1}(q_{max}\cdot(l(\beta)+1))^{l(w_{0})-l(w_{0}^{\prime})-1}\left(q_{max}-1\right)l(\beta)+2^{l(w_{0})-1}(q_{max}\cdot(l(\beta)+1))^{l(w_{0})-l(w_{0}^{\prime})-2}q_{max}\\
 & \le2^{l(w_{0})}(q_{max}\cdot(l(\beta)+1))^{l(w_{0})-l(w_{0}^{\prime})}
\end{align*}

For $l(w_{0}^{\prime}s)=l(w_{0}^{\prime})-1:$ 

\begin{align*}
\sum_{\beta'}\left|\alpha_{w_{0}^{\prime},w_{0},\beta',\beta}\right| & \le\sum_{\beta'}\left(\left|\alpha_{w_{0}^{\prime},\hat{w}_{0},\beta',\beta}\right|\left(\left(q_{max}-1\right)+\left(q_{max}-1\right)\left|\left\langle \alpha_{i},\beta'\right\rangle \right|\right)+\left|\alpha_{w_{0}^{\prime}s_{0},\hat{w}_{0},\beta',\beta}\right|\right)\\
 & \le2^{l(w_{0})-1}(q_{max}\cdot(l(\beta)+1))^{l(w_{0})-l(w_{0}^{\prime})-1}\left(q_{max}-1\right)l(\beta)+2^{l(w_{0})-1}(q_{max}\cdot(l(\beta)+1))^{l(w_{0})-l(w_{0}^{\prime})}\\
 & \le2^{l(w_{0})}(q_{max}\cdot(l(\beta)+1))^{l(w_{0})-l(w_{0}^{\prime})}
\end{align*}

The case $\left(R,i\right)=\left(BC_{n},n\right)$ is similar. 
\end{proof}
We will also need the following variation of the above:
\begin{prop}
\label{prop:Explicit Bernstein Relations2}For $\beta,\beta'\in P$,
$w_{0},w_{0}^{\prime}\in W_{0}$ there exist constants $\alpha_{w_{0}^{\prime},w_{0},\beta',\beta}^{\prime}$
(depending on the parameter system $\overrightarrow{q}$) such that:
\[
Y_{\beta}h_{w_{0}}^{-1}=\sum_{w_{0}^{\prime}\le w_{0},\beta'\le\beta}\alpha_{w_{0}^{\prime},w_{0},\beta',\beta}^{\prime}h_{w_{0}^{\prime}}^{-1}Y_{\beta'}
\]

The constants $\alpha_{w_{0}^{\prime},w_{0},\beta',\beta}^{\prime}$
satisfy:
\begin{enumerate}
\item $\alpha_{w_{0},w_{0},\beta,w_{0}(\beta)}^{\prime}=1$.
\item If $w_{0}^{\prime}\not\le w_{0}$ in the Coxeter order of the Coxet1er
group $W_{0}$, then $\alpha_{w_{0}^{\prime},w_{0},\beta',\beta}^{\prime}=0$.
\item If $\beta$ is dominant, and $\beta'\not\le\beta$ in the coweight
order, then $\alpha_{w_{0}^{\prime},w_{0},\beta',\beta}^{\prime}=0$.
\item $\sum_{\beta'}\left|\alpha_{w_{0}^{\prime},w_{0},\beta',\beta}^{\prime}\right|\le2^{l(w_{0})}(l(\beta)+1)^{l(w_{0})-l(w_{0}^{\prime})}$.
\end{enumerate}
\end{prop}
\begin{proof}
We have $h_{s}^{-1}=q_{s}^{-1}\left(h_{s}-\left(q_{s}-1\right)Id\right)$.
Therefore:
\begin{align*}
Y_{\beta}h_{s_{i}}^{-1} & =h_{s_{i}}^{-1}Y_{s_{i}(\beta)}-q_{s}^{-1}(q_{s_{i}}-1)(Y_{\beta}-Y_{s_{i}(\beta)})+q_{s}^{-1}(q_{s_{i}}-1)\frac{Y_{\beta}-Y_{s_{i}(\beta)}}{1-Y_{-\alpha_{i}^{\vee}}}\\
 & =h_{s_{i}}^{-1}Y_{s_{i}(\beta)}+q_{s}^{-1}(q_{s_{i}}-1)Y_{-\alpha_{i}^{\vee}}\frac{Y_{\beta}-Y_{s_{i}(\beta)}}{1-Y_{-\alpha_{i}^{\vee}}}\\
 & =h_{s_{i}}^{-1}Y_{s_{i}(\beta)}+q_{s}^{-1}(q_{s_{i}}-1)\frac{Y_{\beta}-Y_{s_{i}(\beta)}}{Y_{\alpha_{i}^{\vee}}-1}
\end{align*}

And similarly in the $\left(R,i\right)=\left(BC_{n},n\right)$ case.

The rest of the proof is similar to proposition \ref{prop:Explicit_bernstein_relations},
using $h_{s}^{-2}=q_{s}^{-2}\left(q_{s}Id-(q_{s}-1)h_{s}\right)$.
The extra $q_{s}^{-1},q_{s}^{-2}$ factors allows us to give the slightly
better bound.
\end{proof}

\section{\label{sec:Sectorial-Retraction2}Sectorial Retraction}

The Bernstein presentation is a generalized version of the \emph{Iwasawa
decomposition}. The building analog of the Iwasawa decomposition is
based on the notion of a \emph{sector}, and \emph{sectorial retraction}.
The goal of this section is to explain the connection between the
two.
\begin{defn}
The\emph{ dominant sector} in $V_{R}$ is the set $\mathbb{S}_{0}^{R}=\left\{ v\in V_{R}:\left\langle \alpha_{i},v\right\rangle >0,\,i=1,...,n\right\} $.
A \emph{sector} $\mathbb{S}$ in $V_{R}$ is an image of $\mathbb{S}_{0}^{R}$
under the action of $\hat{W}$. A sector $\mathbb{S}$ (respectively
a dominant sector) in an apartment $\mathbb{A}\subset B$ is the preimage
of any sector $\mathbb{S}'\subset\mathbb{W}\cong V_{R}$ (respectively
$\mathbb{S}_{0}^{R}$) under an isomorphism of $\mathbb{A}$ with
the Coxeter complex $\mathbb{W}$. We identify a sector $\mathbb{S}$
with the set of chambers it contains.
\end{defn}
The following is very standard. 
\begin{lem}
Given an apartment $\mathbb{A}$ of $B$ and a chamber $C_{0}\in\mathbb{A}$,
there exists a retraction $\rho_{C_{0}}^{\mathbb{A}}:B_{\phi}\rightarrow\mathbb{A}$
such that $d(C_{0},C)=d(C_{0},\rho_{C_{0}}^{\mathbb{A}}(C))$.
\end{lem}
In affine buildings we also have another type of retraction into an
apartment $\mathbb{A}$, based on a sector $\mathbb{S}$ of $\mathbb{A}$.
Recall that given two apartments $\mathbb{A},\mathrm{\mathbb{A}'}\subset B$
with $\mathbb{A}\cap\mathrm{\mathbb{A}'}\ne\phi$ we have a unique
isomorphism (as colored simplicial complexes) $\phi_{\mathbb{A},\mathbb{A}'}:\mathbb{\mathbb{A}^{\prime}}\rightarrow\mathbb{\mathbb{A}}$,
that is the identity on $\mathrm{\mathbb{A}^{\prime}\cap\mathbb{A}}.$
\begin{thm}
\label{thm:Sector Retraction} (See \cite{abramenko2008buildings},
theorem 11.63 and lemma 11.64) Given a sector $\mathbb{S}$ of an
apartment $\mathbb{A}$ and a chamber $C\in B_{\phi}$ there exists
a subsector $\mathbb{S}'\subset\mathbb{S}$ and an apartment $\mathbb{A}'$,
such that $\mathbb{S}'\subset\mathbb{A}'$ and $C\in\mathbb{A}'$. 

We may define $\rho_{\mathbb{S}}^{\mathbb{A}}:B_{\phi}\rightarrow\mathbb{A}_{\phi}$
by $\rho_{\mathbb{S}}(C)=\phi_{\mathbb{A},\mathbb{A}'}(C)$, where
$\phi_{\mathbb{A},\mathbb{A}'}:\mathbb{\mathbb{A}^{\prime}}\rightarrow\mathbb{\mathbb{A}}$
is the simplicial isomorphism. This definition does not depend on
the choice of $\mathbb{S}'$ or $\mathbb{A}'$ and we furthermore
have for any $C_{0}\in\mathbb{S}'$, $\rho_{C_{0}}^{\mathbb{A}}(C)=\rho_{\mathbb{S}}^{\mathbb{A}}(C)$. 
\end{thm}
From now on we assume we fix a dominant sector $\mathbb{S}_{0}$ in
the building. We can extend $\rho_{\mathbb{S}_{0}}^{\mathbb{A}}:B_{\phi}\rightarrow\mathbb{A}_{\phi}$
into its ``extended'' version $\hat{\rho}_{\mathbb{S}_{0}}^{\mathbb{A}}:\hat{B}_{\phi}\rightarrow\hat{\mathbb{A}}_{\phi}$.
Now:
\begin{defn}
Define $\rho_{\mathbb{S}_{0}}:\hat{B}_{\phi}\rightarrow\hat{W}$ by
$\hat{f}_{\phi}\circ\hat{\rho}_{\mathbb{S}_{0}}^{\mathbb{A}}$ where
$\hat{f}_{\phi}:\mathbb{A}_{\phi}\rightarrow\hat{\mathbb{W}}_{\phi}\cong\hat{W}$
is an isomorphism of $\mathbb{A}$ and the Coxeter complex $\mathbb{W}$.

We call $\rho_{\mathbb{S}_{0}}(C)\in\hat{W}$ the \emph{sectorial
type} of $C\in\hat{B}_{\phi}$.
\end{defn}
In this section we consider the natural embedding $P\subset\hat{W}$,
so we write addition in $P$ multiplicatively. For example, recall
that every $\beta\in P$ can be written as $\beta=\beta_{1}\beta_{2}^{-1}$,
with $\beta_{1},\text{\ensuremath{\beta}}_{2}\in P^{+}$. 
\begin{defn}
For any element $w\in\hat{W}$, $w=\beta w_{0}$ , with $w_{0}\in W_{0}$,
$\beta\in P$, $\beta=\beta_{1}\beta_{2}^{-1}$, $\beta_{1},\text{\ensuremath{\beta}}_{2}\in P^{+}$,
define $L(w)=L_{\mathbb{S}_{0}}(w)=l(w_{0})+l(\beta_{2})-l(\beta_{1})$
and $Q_{w}=Q_{\mathbb{S}_{0},w}=q_{w_{0}}\cdot q_{\beta_{2}}\cdot q_{\beta_{1}}^{-1}$.
\end{defn}
One can verify that non of the definitions depends on the choice of
$\beta_{1},\beta_{2}\in P^{+}$. We have for $\beta\in P^{+}$, $L(\beta)=-l(\beta)$
and $Q_{\beta}=q_{\beta}^{-1}$.
\begin{lem}
\label{lem:Sector_distance_lemma}If $C_{1},...,C_{l}\in B_{\phi}$,
$\rho_{\mathbb{S}_{0}}(C_{i})=w_{i}$, then there exists a chamber
$C_{0}\in\mathbb{S}_{0}$ such that $\rho_{C_{0}}^{\mathbb{A}}(C_{i})=\rho_{\mathbb{S}_{0}}^{\mathbb{A}}(C_{i})$
and $L(w_{i})=l(d(C_{0},C_{i}))-l(\rho_{\mathbb{S}_{0}}(C_{0}))$.
\end{lem}
\begin{proof}
We may assume that the sector $\mathbb{S}'$ in theorem \ref{thm:Sector Retraction}
is contained in all the sectors $\mathbb{S}_{i}$, $i=1,...,l$ defined
as the sector with dominant direction based on the $0$-vertex of
$C_{i}$. Let $C_{0}\in\mathbb{S}'$ be such that $\rho_{\mathbb{S}_{0}}\left(C_{0}\right)=\beta_{1}\in P^{+}$.
There exists such $C_{0}$ since $\mathbb{S}'\subset\mathbb{S}_{0}$
and $\mathbb{S}_{0}$ is dominant. Since $C_{0}\in\mathbb{S}_{i}$
we have $d(C_{0},C_{i})=\beta_{2}w_{0,i}$, with $\beta_{2}^{-1}\in P^{+}$
(i.e. $\beta_{2}$ is anti dominant) and $w_{0,i}\in W_{0}$. Therefore
$w_{i}=\rho_{\mathbb{S}_{0}}(C)=\rho_{\mathbb{S}_{0}}(C_{0})\cdot d(C_{0},C)=\beta_{1}\beta_{2}w_{0,i}$,
and 
\[
L(w)=l(\beta_{2})-l(\beta_{1})+l(w_{0,i})=l(\beta_{2}w_{0,i})-l(\beta_{1})=l(d(C_{0},C_{i}))-l(\rho_{\mathbb{S}_{0}}(C_{0}))
\]
\end{proof}
\begin{lem}
\label{lem:distances lemma}If $C\in B_{\phi}$, $\rho_{\mathbb{S}_{0}}(C)=w$,
then: 

I. If $L(ws)=L(w)+1$, then $C$ has $q_{s}$ $s$-adjacent chambers
$C'$, with $\rho_{\mathbb{S}_{0}}(C')=ws$.

I. If $L(ws)=L(w)-1$, then $C$ has $q_{s}-1$ $s$-adjacent chambers
$C'$, with $\rho_{\mathbb{S}_{0}}(C')=w$ and one $s$-adjacent chamber
$C'$, with $\rho_{\mathbb{S}_{0}}(C')=ws$.
\end{lem}
\begin{proof}
Since $\rho_{\mathbb{S}_{0}}$ is a retraction, all those chambers
are with $\rho_{\mathbb{S}_{0}}(C')=w$ or $\rho_{\mathbb{S}_{0}}(C')=ws$.
Choose $C_{0}\in\mathbb{S}$ such that lemma \ref{lem:Sector_distance_lemma}
holds for all the $q_{s}+1$ chambers containing the $s$-panel of
$C$. Recall that one of the chambers is closer then the $q_{s}$
others to $C_{0}$. If $L(ws)=L(w)+1$ the closest chamber is $C$
and the $q_{s}$ others are with $\rho_{\mathbb{S}_{0}}(C')=ws$.
If $L(ws)=L(w)-1$, $C$ is not the closest and the claim follows.
\end{proof}
\begin{defn}
Let $w\in\hat{W}$ and $1_{w}^{\mathbb{S}_{0}}\in\mathbb{C}^{\hat{B}_{\phi}}$
be the function defined as $1_{w}^{\mathbb{S}_{0}}(C)=1$ if $\rho_{\mathbb{S}_{0}}(C)=w$
and $1_{w}^{\mathbb{S}_{0}}(C)=0$ otherwise.
\end{defn}
\begin{lem}
\label{lem:actionLemma1}If $w,w^{\prime}\in\hat{W}$ with $L(ww^{\prime-1})=L(w)+l(w^{\prime-1})$,
then $h_{w^{\prime}}1_{w}^{\mathbb{S}_{0}}=1_{ww^{\prime-1}}^{\mathbb{S}_{0}}$.
\end{lem}
\begin{proof}
It is enough to prove for $w^{\prime}=s\in S$. If $L(ws)=L(w)+1$
then by lemma \ref{lem:distances lemma}, for every chamber $C\in\hat{B}_{\phi}$
with $\rho_{\mathbb{S}_{0}}(C)=w$, the $q_{s}$ $s$-adjacent chambers
$C'$ to $C$, all satisfy $\rho_{\mathbb{S}_{0}}(C')=ws$. Moreover,
each chamber $C'$ with $\rho_{\mathbb{S}_{0}}(C')=ws$, has a unique
chamber $C$ with $\rho_{\mathbb{S}_{0}}(C)=w$. Therefore $h_{s}1_{w}^{\mathbb{S}_{0}}=1_{ws}^{\mathbb{S}_{0}}$.
\end{proof}
\begin{defn}
For $\beta\in P$, $\beta=\beta_{1}\beta_{2}^{-1}$, $\beta_{1},\text{\ensuremath{\beta}}_{2}\in P^{+}$,
we denote by $X_{\beta}\in\hat{H}_{\phi}$ the element $X_{\beta}=h_{\beta_{1}}h_{\beta_{2}}^{-1}$. 
\end{defn}
Notice that $Y_{\beta}=Q(\beta)^{1/2}X_{\beta}$, where $Y_{\beta}$
is as in the Bernstein presentation. 
\begin{lem}
For $w_{0}\in W_{0}$, $\beta\in P$ we have $1_{\beta^{-1}w_{0}^{-1}}^{\mathbb{S}_{0}}=h_{w_{0}}X_{\beta}1_{Id}^{\mathbb{S}_{0}}$.
\end{lem}
\begin{proof}
For $\beta\in P^{+}$, $\beta'\in P$, we have $L(\beta^{\prime}\beta^{-1})=L(\beta^{\prime})+l(\beta^{-1})$.
By lemma \ref{lem:actionLemma1}, we have $h_{\beta}1_{\beta'}^{\mathbb{S}_{0}}=1_{\beta'\beta^{-1}}^{\mathbb{S}_{0}}$. 

If $\beta'=\beta$, $1_{Id}^{\mathbb{S}_{0}}=h_{\beta}1_{\beta}^{\mathbb{S}_{0}}$,
or $h_{\beta}^{-1}1_{Id}^{\mathbb{S}_{0}}=1_{\beta}^{\mathbb{S}_{0}}$.
Therefore for $\beta\in P$, $\beta=\beta_{1}\beta_{2}^{-1}$, $\beta_{1},\text{\ensuremath{\beta}}_{2}\in P^{+}$,
\[
X_{\beta}1_{Id}^{\mathbb{S}_{0}}=h_{\beta_{1}}h_{\beta_{2}}^{-1}1_{Id}^{\mathbb{S}_{0}}=1_{\beta_{1}\beta_{2}^{-1}}^{\mathbb{S}_{0}}=1_{\beta^{-1}}^{S_{0}}
\]

Finally, $L(\beta w_{0}^{-1})=L(\beta)+l(w_{0})$ and therefore 
\[
h_{w_{0}}X_{\beta}1_{Id}^{\mathbb{S}_{0}}=h_{w_{0}}1_{\beta^{-1}}^{S_{0}}=1_{\beta^{-1}w_{0}^{-1}}^{\mathbb{S}_{0}}
\]
\end{proof}
\begin{cor}
\label{cor:Sectorial action1}The action of $h_{\beta}$, $\beta\in P^{+}$
on $\text{span}\left\{ 1_{w}^{\mathbb{S}_{0}}\right\} _{w\in\hat{W}}$
satisfies: 
\[
h_{\beta}1_{\gamma^{-1}w_{0}^{-1}}^{\mathbb{S}_{0}}=q_{\beta}^{1/2}\sum_{w_{0}^{\prime},\beta'}\alpha_{w_{0}^{\prime},w_{0},\beta',\beta}Q_{\beta'}^{1/2}1_{\gamma^{-1}\beta^{\prime-1}w_{0}^{\prime-1}}^{\mathbb{S}_{0}}
\]

where $\alpha_{w_{0},w_{0},\beta,\beta}$ are as in theorem \ref{prop:Explicit_bernstein_relations}.
\end{cor}
\begin{proof}
We apply the Bernstein relations of proposition \ref{prop:Explicit_bernstein_relations}
and get

\begin{align*}
h_{\beta}1_{\gamma^{-1}w_{0}^{-1}}^{\mathbb{S}_{0}} & =X_{\beta}h_{w_{0}}X_{\gamma}1_{Id}^{\mathbb{S}_{0}}=Q_{\beta}^{-1/2}Y_{\beta}h_{w_{0}}1_{\gamma^{-1}}^{\mathbb{S}_{0}}\\
 & =q_{\beta}^{1/2}\sum_{w_{0}^{\prime},\beta'}\alpha_{w_{0}^{\prime},w_{0},\beta',\beta}h_{w_{0}^{\prime}}Y_{\beta'}1_{\gamma^{-1}}^{\mathbb{S}_{0}}\\
 & =q_{\beta}^{1/2}\sum_{w_{0}^{\prime},\beta'}\alpha_{w_{0}^{\prime},w_{0},\beta',\beta}Q_{\beta'}^{1/2}1_{\gamma^{-1}\beta^{\prime-1}w_{0}^{\prime-1}}^{\mathbb{S}_{0}}
\end{align*}
\end{proof}
\begin{defn}
Let $\lambda_{w}^{\mathbb{S}_{0}}\in\mathbb{C}[\hat{B}_{\phi}]^{\ast}$
be the functionals $\lambda_{w}^{\mathbb{S}_{0}}(f)=\sum_{C:\rho_{\mathbb{S}_{0}}(C)=w}f(C)$.
\end{defn}
\begin{lem}
\label{lem:actionLemma2}If $w,w^{\prime}\in\hat{W}$ with $L(ww^{\prime})=L(w)+l(w^{\prime})$,
then $\lambda_{w}^{\mathbb{S}_{0}}h_{w^{\prime}}=\lambda_{ww^{\prime}}^{\mathbb{S}_{0}}$.
\end{lem}
\begin{proof}
By induction, it is enough to prove for $w^{\prime}=s\in S$. Then
it is a direct result of lemma \ref{lem:distances lemma}, as in lemma
\ref{lem:actionLemma1}.
\end{proof}
\begin{prop}
For $w_{0}\in W_{0}$, $\beta\in P$ we have $\lambda_{\beta^{-1}\tilde{w}_{0}w_{0}^{-1}}^{\mathbb{S}_{0}}=\lambda_{\tilde{w}_{0}}^{\mathbb{S}_{0}}X_{\beta}h_{w_{0}}^{-1}$.
\end{prop}
\begin{proof}
For $\beta\in P^{+}$, $\beta'\in P$ we have, $\beta^{\prime}\tilde{w}_{0}\beta=\beta^{\prime}\beta^{-1}\tilde{w}_{0}$,
so $L(\beta^{\prime}\tilde{w}_{0}\beta)=L(\beta^{\prime}\beta^{-1}\tilde{w}_{0})=l(\beta^{\prime})+l(\beta^{-1}\tilde{w}_{0})$.
Therefore by lemma \ref{lem:actionLemma2}, $\lambda_{\beta'\tilde{w}_{0}}^{\mathbb{S}_{0}}h_{\beta}=\lambda_{\beta'\beta^{-1}\tilde{w}_{0}}^{\mathbb{S}_{0}}$. 

If $\beta=\beta_{1}\beta_{2}^{-1}$, $\beta\in P$, $\beta_{1},\beta_{2}\in P^{+}$
then
\[
\lambda_{\tilde{w}_{0}}^{\mathbb{S}_{0}}X_{\beta}=\lambda_{\tilde{w}_{0}}^{\mathbb{S}_{0}}h_{\beta_{1}}h_{\beta_{2}}^{-1}=\lambda_{\beta_{1}^{-1}\beta_{2}\tilde{w}_{0}}^{\mathbb{S}_{0}}=\lambda_{\beta^{-1}\tilde{w}_{0}}^{\mathbb{S}_{0}}
\]

Similarly, since $L(\tilde{w}_{0}w_{0}^{-1}w_{0})=L(\tilde{w}_{0}w_{0}^{-1})+l(w_{0})$,
\[
\lambda_{\beta^{-1}\tilde{w}_{0}w_{0}^{-1}}^{\mathbb{S}_{0}}h_{w_{0}}=\lambda_{\beta^{-1}\tilde{w}_{0}}^{\mathbb{S}_{0}}
\]

Summarizing, we get $\lambda_{\beta^{-1}\tilde{w}_{0}w_{0}^{-1}}^{\mathbb{S}_{0}}=\lambda_{\tilde{w}_{0}}^{\mathbb{S}_{0}}X_{\beta}h_{w_{0}}^{-1}.$.
\end{proof}
\begin{cor}
\label{cor:Sectorial action2}We have 
\[
\lambda_{\gamma^{-1}\tilde{w}_{0}w_{0}^{-1}}^{\mathbb{S}_{0}}h_{\beta}=q_{\beta}^{1/2}\sum\alpha_{w_{0}^{\prime},w_{0},\beta',\beta}^{\prime}q_{w_{0}^{\prime}}Q_{\beta'}^{1/2}\lambda_{\gamma^{-1}\beta^{\prime}\tilde{w}_{0}w_{0}^{\prime-1}}^{\mathbb{S}_{0}}
\]
\end{cor}
\begin{proof}
Using the Bernstein relations of proposition \ref{prop:Explicit Bernstein Relations2}
we have:
\begin{align*}
\lambda_{\gamma^{-1}\tilde{w}_{0}w_{0}^{-1}}^{\mathbb{S}_{0}}h_{\beta} & =\lambda_{\tilde{w}_{0}}^{\mathbb{S}_{0}}X_{\gamma}h_{w_{0}}^{-1}X_{\beta}=\lambda_{\gamma^{-1}\tilde{w}_{0}}^{\mathbb{S}_{0}}Q_{\beta}^{-1/2}h_{w_{0}}^{-1}Y_{\beta}\\
 & =\lambda_{\gamma^{-1}\tilde{w}_{0}}^{\mathbb{S}_{0}}q_{\beta}^{1/2}\sum\alpha_{w_{0}^{\prime},w_{0},\beta',\beta}^{\prime}Y_{\beta'}h_{w_{0}^{\prime-1}}^{-1}\\
 & =\lambda_{\gamma^{-1}\tilde{w}_{0}}^{\mathbb{S}_{0}}q_{\beta}^{1/2}\sum\alpha_{w_{0}^{\prime},w_{0},\beta',\beta}^{\prime}Q_{\beta'}^{1/2}X_{\beta'}h_{w_{0}^{\prime-1}}^{-1}\\
 & =q_{\beta}^{1/2}\sum\alpha_{w_{0}^{\prime},w_{0},\beta',\beta}^{\prime}Q_{\beta'}^{1/2}\lambda_{\gamma^{-1}\beta^{\prime-1}\tilde{w}_{0}w_{0}^{\prime-1}}
\end{align*}
\end{proof}
The following theorem summarizes the discussion in this section. It
relates the Bernstein presentation of the Iwahori-Hecke algebra and
the sectorial geometry of the building:
\begin{thm}
\label{thm:Sectorial distance theorem}Let $w=\gamma^{-1}w_{0}^{-1}$,
$w^{\prime}=\gamma^{\prime-1}w_{0}^{\prime-1}$ be elements of $\hat{W}$
and $\beta\in P^{+}$. Then:
\begin{itemize}
\item For every $C'\in\hat{B}_{\phi}$ with $\rho_{\mathbb{S}_{0}}(C')=w^{\prime}$
there exist $N_{w^{\prime},w}=q_{\beta}^{1/2}Q_{\gamma^{\prime}\gamma^{-1}}^{1/2}\alpha_{w_{0}^{\prime},w_{0},\gamma^{\prime}\gamma^{-1},\beta}$
chambers $C\in\hat{B}_{\phi}$, $\rho_{\mathbb{S}_{0}}(C)=w$, with
$d(C^{\prime},C)=\beta$.
\item For every $C\in\hat{B}_{\phi}$ with $\rho_{\mathbb{S}_{0}}(C)=w$
there exist $N_{w^{\prime},w}^{\prime}=q_{\beta}^{1/2}Q_{\gamma\gamma^{\prime-1}}^{1/2}\alpha_{\tilde{w}_{0}w_{0},\tilde{w}_{0}w_{0}^{\prime},\gamma\gamma^{\prime-1},\beta}^{\prime}$
chambers $C'\in\hat{B}_{\phi}$, $\rho_{\mathbb{S}_{0}}(C')=w^{\prime}$,
with $d(C^{\prime},C)=\beta$.
\end{itemize}
\end{thm}
\begin{proof}
With notations as above, and by definition of $1_{w^{\prime}}^{\mathbb{S}_{0}}$
and $h_{\beta}$, $N_{w^{\prime},w}$ is the coefficient of $1_{w^{\prime}}^{\mathbb{S}_{0}}$
in the decomposition of $h_{\beta}1_{w}^{\mathbb{S}_{0}}$. By proposition
\ref{cor:Sectorial action1}, this number is $q_{\beta}^{1/2}Q_{\gamma'\gamma^{-1}}^{1/2}\alpha_{w_{0}^{\prime},w_{0},\gamma'\gamma^{-1},\beta}$.

Similarly, $N_{w^{\prime},w}^{\prime}$ is the coefficient of $\lambda_{w}^{\mathbb{S}_{0}}$
in the decomposition of $1_{w^{\prime}}^{\mathbb{S}_{0}}h_{\beta}$.
Let $w_{1}=w_{0}\tilde{w}_{0}$, $w_{1}^{\prime}=w_{0}^{\prime}\tilde{w}_{0}$.
Then $w=\gamma^{-1}\tilde{w}_{0}w_{1}^{-1}$, $w^{\prime}=\gamma^{\prime-1}\tilde{w}_{0}w_{1}^{\prime-1}$.
By proposition \ref{cor:Sectorial action1} this coefficient is $q_{\beta}^{1/2}Q_{\gamma\gamma^{\prime-1}}^{1/2}\alpha_{w_{1},w_{1}^{\prime},\gamma\gamma^{\prime-1},\beta}^{\prime}$.
\end{proof}

\section{\label{sec:How-to-Bound-operators}How to Bound Operators}

The goal is this section is to use theorem \ref{thm:Sectorial distance theorem}
to bound the norm of $h_{\beta}$, $\beta\in P^{+}$. For all this
section, $\beta$ is fixed.
\begin{lem}
Let $X=X_{0}\cup X_{1}$ be a (possibly infinite) a biregular graph,
such that every $x\in X_{0}$ is connected to at most $K_{0}$ vertices
in $X_{1}$, and every $y\in X_{1}$ is connected to at most $S_{1}$
vertices in $X_{0}$. 

Let $A_{X}:\mathbb{C}^{X_{0}}\rightarrow\mathbb{C}^{X_{1}}$ be the
adjacency operator from $X_{0}$ to $X_{1}$, i.e. $Af(y)=\sum_{x\sim y}f(x)$.
Then as operator $A:L_{p}(X_{0})\rightarrow L_{p}(X_{1})$, we have
$\left\Vert A\right\Vert _{p}\le K_{0}^{1/p}K_{1}^{(p-1)/p}$.
\end{lem}
\begin{proof}
For $X$ finite the result follows from the convexity of $s\rightarrow s^{p}$.
Then it extends easily to $X$ infinite.
\end{proof}
Let $w,w^{\prime}\in\hat{W}$. Create a graph bipartite graph $X_{w,w^{\prime}}$.
The vertices $X_{0}$ will be chambers $C\in B_{\phi}$ with $\rho_{\mathbb{S}_{0}}(C)=w=\gamma^{-1}w_{0}^{-1}$,
and the vertices $X_{1}$ will be the chambers $C'$ with $\rho_{\mathbb{S}_{0}}(C')=w^{\prime}=\gamma^{\prime-1}w_{0}^{\prime-1}$.
Connect two chambers $C_{0},C_{1}$ if $d(C_{1},C_{0})=\beta$. Then
by theorem \ref{thm:Sectorial distance theorem}, this graph is $\left(K_{0},K_{1}\right)$-biregular,
with $K_{0}=Q_{\gamma\gamma^{\prime-1}}^{1/2}\alpha_{\tilde{w}_{0}w_{0}^{-1},\tilde{w}_{0}w_{0}^{\prime-1},\gamma\gamma^{\prime-1},\beta}^{\prime}$
and $K_{1}=q_{\beta}^{1/2}Q_{\gamma'\gamma^{-1}}^{1/2}\alpha_{w_{0}^{\prime},w_{0},\gamma'\gamma^{-1},\beta}$.
By the above lemma we have:
\[
\left\Vert A_{w,w^{\prime}}\right\Vert _{p}\le K_{0}^{1/p}K_{1}^{(p-1)/p}=q_{\beta}^{1/2}Q_{\gamma'\gamma^{-1}}^{(p-2)/2p}\alpha_{\tilde{w}_{0}w_{0}^{-1},\tilde{w}_{0}w_{0}^{\prime-1},\gamma\gamma^{\prime-1},\beta}^{\prime1/p}\alpha_{w_{0}^{\prime},w_{0},\gamma\gamma^{\prime-1},\beta}^{(p-1)/p}
\]

Let us now extend this to a bound on $h_{\beta}$. First of all, one
can consider the results above by fixing $w_{0},w_{0}^{\prime}$ and
$\beta^{\prime}=\gamma^{\prime}\gamma^{-1}$, but letting $\gamma$
change. For each different $\gamma$, we have a different graph $X_{w,w^{\prime}}$,
with the same bound on $\left\Vert A_{w,w^{\prime}}\right\Vert _{p}$.
All those graphs are disjoint. Therefore the same bound holds for
the disjoint union $X_{w_{0},w_{0}^{\prime},\beta'}=\sqcup X_{w,w^{\prime}}$
of all those graphs, that is 
\begin{equation}
\left\Vert A_{w_{0},w_{0}^{\prime},\beta'}\right\Vert _{p}\le q_{\beta}^{1/2}Q_{\beta^{\prime}}^{(p-2)/2p}\alpha_{\tilde{w}_{0}w_{0}^{-1},\tilde{w}_{0}w_{0}^{\prime-1},\beta^{\prime},\beta}^{\prime1/p}\alpha_{w_{0}^{\prime},w_{0},\beta^{\prime-1},\beta}^{(p-1)/p}\label{eq:matrix coefficient bound}
\end{equation}

Let $w_{0}\in W_{0}$ and for $f\in L_{p}(\hat{B}_{\phi})$ define
$f^{w_{0}}\in L_{p}(B_{\phi})$ by: 
\[
f^{w_{0}}(C)=\begin{cases}
f(C) & \exists\gamma\in P:\,\rho_{\mathbb{S}_{0}}(C)=\gamma^{-1}w_{0}^{-1}\\
0 & \text{else}
\end{cases}
\]

Surely $f=\sum_{w_{0}\in W_{0}}f^{w_{0}}$ and $\left\Vert f\right\Vert _{p}^{p}=\sum\left\Vert f^{w_{0}}\right\Vert _{p}^{p}$. 
\begin{lem}
\label{lem:sectorial_matrix_lemma}Using the above notations, we have:
\[
\left(h_{\beta}f\right)^{w_{0}^{\prime}}=\sum_{w_{0}^{\prime}\le w_{0}}\left(\sum_{\beta'\le\beta}A_{w_{0},w_{0}^{\prime},\beta'}\right)f^{w_{0}}
\]
\end{lem}
\begin{proof}
Follows by definition of $f^{w_{0}}$ and , the graph $X_{w_{0},w_{0}^{\prime},\beta'}$
and the adjacency operator $A_{w_{0},w_{0}^{\prime},\beta'}$.
\end{proof}
We can now prove:
\begin{thm}
\label{thm:Bound on H_beta}With the notations above, for $f\in L_{p}(\hat{B}_{\phi})$
and $\beta\in P^{+},$ we have:
\[
\left\Vert (h_{\beta}f)^{w_{0}^{\prime}}\right\Vert _{p}\le q_{\beta}^{1/2}\sum_{w_{0}^{\prime}\le w_{0}}\left(\sum_{\beta'\le\beta}Q_{\beta'}^{(p-2)/2p}\alpha_{\tilde{w}_{0}w_{0}^{-1},\tilde{w}_{0}w_{0}^{\prime-1},\beta^{\prime},\beta}^{\prime1/p}\alpha_{w_{0}^{\prime},w_{0},\beta^{\prime-1},\beta}^{(p-1)/p}\right)\left\Vert f^{w_{0}}\right\Vert _{p}
\]

As a corollary, $\left\Vert h_{\beta}\right\Vert _{p}\le\left|W_{0}\right|\left|2q_{max}\right|^{l(\tilde{w}_{0})}\left(l(\beta)+1\right)^{l(\tilde{w}_{0})}q_{\beta}^{(p-1)/p}$.
\end{thm}
\begin{proof}
The first inequality follows from lemma \ref{lem:sectorial_matrix_lemma}
and the bound \ref{eq:matrix coefficient bound}.

We now turn to simplify the expression $q_{\beta}^{1/2}\sum_{\beta'\le\beta}Q_{\beta'}^{(p-2)/2p}\alpha_{w_{0}^{\prime},w_{0},\beta',\beta}^{(p-1)/p}\alpha_{\tilde{w}_{0}w_{0}^{-1},\tilde{w}_{0}w_{0}^{\prime-1},\beta^{\prime-1},\beta}^{\prime1/p}$.
since $\beta'\le\beta$ we have by lemma \ref{lem:Lengthes and dominance}
$Q_{\beta'}\le q_{\beta}$, so $q_{\beta}^{1/2}Q_{\beta'}^{(p-2)/2p}\le q_{\beta}^{(p-1)/p}$.
Moreover:
\[
\begin{array}{ccc}
\sum_{\beta'}\alpha_{\tilde{w}_{0}w_{0}^{-1},\tilde{w}_{0}w_{0}^{\prime-1},\beta^{\prime},\beta}^{\prime1/p}\alpha_{w_{0}^{\prime},w_{0},\beta^{\prime-1},\beta}^{(p-1)/p} & \le & \left(\sum_{\beta'}\alpha_{\tilde{w}_{0}w_{0}^{-1},\tilde{w}_{0}w_{0}^{\prime-1},\beta^{\prime-1},\beta}^{\prime}\right)^{1/p}\left(\sum_{\beta'}\alpha_{w_{0}^{\prime},w_{0},\beta',\beta}\right)^{(p-1)/p}\\
 & \le & \left(2^{l(w_{0})}(q_{max}\cdot(l(\beta)+1))^{l(w_{0})-l(w_{0}^{\prime})}\right)^{(p-1)/p}\left(2^{l(w_{0})}(l(\beta)+1)^{l(w_{0})-l(w_{0}^{\prime})}\right)^{1/p}\\
 & \le & \left|2q_{max}\right|^{l(\tilde{w}_{0})}\left(l(\beta)+1\right)^{l(\tilde{w}_{0})}
\end{array}
\]

The first inequality follows from Hölder's inequality. The second
from the bounds of propositions \ref{prop:Explicit_bernstein_relations}
and \ref{prop:Explicit Bernstein Relations2}.

Therefore we have $\left\Vert (h_{\beta}f)^{w_{0}^{\prime}}\right\Vert _{p}\le q_{\beta}^{(p-1)/p}\left|2q_{max}\right|^{l(\tilde{w}_{0})}\left(l(\beta)+1\right)^{l(\tilde{w}_{0})}\sum_{w_{0}^{\prime}\le w_{0}}\left\Vert f^{w_{0}}\right\Vert _{p}$.
Denote $\lambda=q_{\beta}^{(p-1)/p}\left|2q_{max}\right|^{l(\tilde{w}_{0})}\left(l(\beta)+1\right)^{l(\tilde{w}_{0})}$.
Using the convexity inequality $\left(\sum_{i=1}^{N}\alpha_{i}\right)^{p}\le N^{p-1}\sum_{i=1}\alpha_{i}^{p}$,
we have
\[
\begin{array}{ccc}
\left\Vert h_{\beta}f\right\Vert _{p}^{p} & = & \sum_{w_{0}^{\prime}}\left\Vert \left(h_{\beta}f\right)^{w_{0}^{\prime}}\right\Vert _{p}^{p}\\
 & \le & \sum_{w_{0}^{\prime}}\left(\sum_{w_{0}}\lambda\left\Vert f^{w_{0}}\right\Vert _{p}\right)^{p}\\
 & \le & \lambda^{p}\left|W_{0}\right|\left|W_{0}\right|^{p-1}\sum_{w_{0}}\left\Vert f^{w_{0}}\right\Vert _{p}^{p}\\
 & = & \lambda^{p}\left|W_{0}\right|^{p}\left\Vert f\right\Vert _{p}^{p}
\end{array}
\]

and $\left\Vert h_{\beta}\right\Vert _{p}\le\left|W_{0}\right|\lambda$
as needed.
\end{proof}
\begin{rem}
As said in the introduction, the proof presented here is based on
the proof of \cite{cowling1988almost}, theorem 2.
\end{rem}
\bibliographystyle{amsalpha}
\bibliography{0C__Users_Amitay_Dropbox_Thesis_database}

\end{document}